\documentclass[twoside,11pt]{article}

\usepackage[english]{babel}
\usepackage{amsmath}
\usepackage{amsthm}
\usepackage{cases}
\usepackage{amsfonts}
\usepackage{amssymb}
\usepackage{graphicx}
\usepackage{colortbl,dcolumn}
\usepackage{float}
\usepackage{paralist}  % generate "compactenum"
\allowdisplaybreaks[4]

\newtheorem{lemma}{\bf Lemma}[section]
\newtheorem{theorem}{\bf Theorem}[section]
\newtheorem{proposition}{\bf Proposition}[section]

\newtheorem{remark}{\bf Remark}[section]

\numberwithin{equation}{section}

\newcommand{\be}{\begin{equation}}
	\newcommand{\ce}{\end{equation}}

\newcommand\bes{\begin{eqnarray}}
	\newcommand\ees{\end{eqnarray}}
\newcommand\bess{\begin{eqnarray*}}
	\newcommand\eess{\end{eqnarray*}}

\begin{document}
	\title{{\Large Convergence to shock profiles for Burgers equation with\\ singular fast-diffusion and boundary effect}
		\footnotetext{\small
			*Corresponding author.}
		\footnotetext{\small E-mail addresses: xwli@math.pku.edu.cn (X. Li), ming.mei@mcgill.ca (M. Mei).} }
	
	\author{{Xiaowen Li$^{1}$ and Ming Mei$^{2,3,4}$$^\ast$}\\[2mm]
		\small\it $^1$School of Mathematical Sciences, Peking University, \\
		\small\it  Beijing, 100871, P.~R.~China \\
		 \small\it $^2$School of Mathematics and Statistics,  Jiangxi Normal University,\\
		\small\it    Nanchang, 330022, P.R.China \\
		\small\it $^3$Department of Mathematics, Champlain College Saint-Lambert,\\
		\small\it     Saint-Lambert, Quebec, J4P 3P2, Canada\\
		\small\it $^4$Department of Mathematics and Statistics, McGill University,\\
		\small\it     Montreal, Quebec, H3A 2K6, Canada
		  }

	\date{}
	
	\maketitle
	
	\begin{quote}
		\small \textbf{Abstract}: In this paper, we study the asymptotic stability of viscous shock profile for the Burgers equation $u_t +f(u)_x = (\frac{u_{x}}{u^{1-m}})_x$ on the half-space $(0,+\infty)$, subject to the boundary conditions $u|_{x=0}=u_->0$ and $u|_{x=+\infty}=0$. Here, the parameter $\frac{1}{2}<m<1$ measures the strength of fast diffusion. A key challenge arises from the pronounced singularity in the diffusivity $\left(\frac{u_x}{u^{1-m}} \right)_x$ at $u=0$ and the boundary layer. We demonstrate that the long-time behavior of $u$ converges to a shifted shock profile $U(x-st-d(t))$, where $d(t)$ is governed by the boundary layer dynamics at $x=0$ and driven by the initial data $u(x,0)$. To overcome the  singularity from fast diffusion  compounded by the bad effect of boundary layer for wave stability, some new techniques for weighted energy estimates are introduced artfully.

		
		\indent \textbf{AMS (2010) Subject Classification}: 35B40, 35L65

	\end{quote}

\tableofcontents	

\baselineskip=17pt
	
	%%%%%%  1. Introduction  %%%%%%
	
\section{Introduction}
We consider the initial-boundary value problem of Burgers equation on the half space $\mathbb{R_+}=(0,\infty)$
	\begin{equation}\label{original model}
		\begin{cases}
			u_t +f(u)_x = (\frac{u_{x}}{u^{1-m}})_x,\quad x\in \mathbb{R_+},~t>0,\\
			u(0,t)=u_-,\quad t>0,\\
			u(x,0)=u_0(x)=\begin{cases}
				u_-,\quad x=0,\\ 0,\quad x \rightarrow+\infty,
			\end{cases}
		\end{cases}
	\end{equation}
where $u$ is the velocity and $f$ is a smooth and convex flux function. Without loss of generality, we assume $f(0)=0$. The exponent $0<m<1$ indicates that the equation \eqref{original model} is of fast diffusion $(\frac{u_x}{u^{1-m}})_x$, which exhibits a pronounced
singularity at $u=0$.

The corresponding Cauchy problem in whole space $\mathbb{R}$ has a unique viscous shock wave $u=U(z)(z=x-st)$ with wave speed $s$ up to a shift \cite{Xu-Mei-Qin-Sheng}
\begin{equation}\label{shock profile equ}
	\begin{cases}
		-sU_{z}-\left(U^{m-1}U_{z}\right)_{z}+f(U)_{z}=0,\quad z\in \mathbb{R}, \\
		U(+\infty)=0,~ U(-\infty)=u_{-},
	\end{cases}
\end{equation}
under the Rankine-Hugoniot condition
\begin{equation}\tag{R-H}\label{R-H}
	s=\frac{f(0)-f(u_-)}{0-u_-},
\end{equation}
and the generalized shock condition
\begin{equation}\tag{E}\label{Lax's}
	g(u)\triangleq f(u)-f(u_{-})-s(u-u_{-})=f(u)-su<0, \quad\text{for } u\in(0,u_-).
\end{equation}
However, since \eqref{original model} is posed in the half-space $\mathbb{R_+}$, a boundary layer develops at $x=0$, namely,
\begin{equation}\nonumber
(u-U)\big|_{x=0}=u_--U(-st)\neq0,	
\end{equation}
the effects of which, together with those arising from singularity of the fast diffusion, are our main concern in this paper.

{\large\textbf{Background of study}}.
The stability of viscous shock profiles for the Cauchy problem of \eqref{original model} has been extensively studied in many works by various methods; however, the majority of these focus on the linear diffusion case (i.e.~$m=1$),
\begin{equation}\label{linear}\nonumber
	 u_t+f(u)_x=u_{xx},
\end{equation}
which corresponds to the classical Burgers equation. See \cite{Il'in} for the first result on the stability of viscous shock profiles by maximum principle, \cite{same line,nonconvex,conventional energy method,M. Mei} for stability and convergence rates derived through energy method, \cite{Howard-2,Howard-3,Howard-1} for sharper stability results from the construction of Green's function, and \cite{Kang-Vasseur} for $L^2$ contraction of shock waves of general perturbations via the relative entropy method. We also refer the significant contributions of \cite{Jones,spectural} based on spectral analysis and of \cite{Serre,nonconvex2} on $L^1$ stability within a semigroup framework.

When $m>1$, the equation \eqref{original model} corresponds to the slow diffusion regime in porous medium flows. In this case, the viscous shock profile degenerates and loses regularity at $u=0$, leading to a free boundary problem. Osher-Ralston  \cite{nonconvex2} first proved $L^1$ stability, which remains the only known result for the Cauchy problem of slow diffusion case of \eqref{original model}, so far. However, some related works have addressed reaction-diffusion models with slow diffusion, focusing on the stability of sharp traveling waves of Fisher-KPP equations \cite{Biro,Du-Quiros-Zhou,Du-Garriz-Quiros,Lou-Zhou} and Burgers-Fisher-KPP equations \cite{XJMY3,Xu-Ji-Mei-Yin,XJMY2,XJMY JDE}. In all these works, the analysis heavily relies on the comparison principles owing to the reaction term, with stability established via the method of upper and lower solutions.

When $m<1$, the equation \eqref{original model} falls into the fast diffusion regime in porous medium flows, encompassing the super-fast diffusion case ($m<0$), the critical fast diffusion case ($m=0$), and the regular fast diffusion case ($0<m<1$). In these regimes, the nonlinear diffusion $(\frac{u_x}{u^{1-m}})_x$ generates a singularity at $u=0$, which brings considerable difficulties to the stability analysis. Recently, the authors of this paper made a breakthrough in this direction by employing weighted energy estimates to overcome the pronounced singularity at the end state $u_+=0$ caused by nonlinear diffusivity $(\ln u)_{xx}=(\frac{u_x}{u})_x$ for the critical fast diffusion with $m=0$,  and established the stability of viscous shock profiles for the Cauchy problem in the critical fast diffusion case \cite{critical fast diffusion}. Then, the second author of this paper with collaborators \cite{Xu-Mei-Qin-Sheng} extended the analysis to the regular fast diffusion case with $0<m<1$, where a stronger singular weighted estimate is required. To this end, a convexity condition was also introduced in the stability results. Moreover, we \cite{super fast diffusion} have also treated the super-fast diffusion case for the shock wave stability, which exhibits the strongest singularity.

In consideration of the boundary effect, existing studies have addressed only \eqref{original model} with linear diffusion (i.e.~$m=1$). For the classical Burgers equation on the half-space with a Dirichlet boundary condition, Liu-Yu \cite{Liu-Yu} provided the first analysis of asymptotic convergence to the stationary viscous shock wave. To compensate the ``boundary gap'', they showed that
$$
u(x,t)\sim U(x+d(t)),\quad d(t)\sim \log t \text{ as }t\rightarrow \infty,
$$
in the case of $f=\frac{u^2}{2}$, employing the method of pointwise estimates. This was then significantly developed by Liu-Nishihara \cite{Liu-Nishihara}, see also   \cite{Nishihara}.  Moreover, Liu-Matsumura-Nishihara \cite{Liu-Matsumura-Nishihara} investigated the asymptotic behavior of the superposition of a viscous shock wave forming a boundary layer and a rarefaction wave propagating away from the boundary. For the classical Burgers equation in exterior domains on multidimensional
spaces, we refer to \cite{Hashimoto radially} for the
stability of radially symmetric stationary solutions, and \cite{Hashimoto} for non-radially symmetric perturbed fluid motion.	These results are entirely free from singularities.

Compared with the case of single viscous conservation laws, the same argument cannot be directly applied to the case of systems, not even to the physical $p$-system, since the shift $d(t)$ cannot be uniquely determined from two equations. However,  Matsumura-Mei \cite{Matsumura-Mei} observed that the boundary conditions cannot be imposed on $v(x,t)$ which allows the wave shift $\beta$ being a constant to be determined explicitly from the equation satisfied by $v$ and proved the shock stability of $U(x-st-\beta)$. While the global asymptotics toward the rarefaction wave for solutions of p-system with viscosity was also analyzed by Matsumura-Nishihara \cite{Matsumura-Nishihara} and Pan-Liu-Nishihara \cite{Pan} based on elementary $L^2$ energy method. Recently, Hashimoto-Matsumura \cite{Hashimoto-Matsumura} demonstrated that, in the
inviscid limit, the Navier-Stokes flow  converges uniformly to a linear superposition of the corresponding boundary layer profile and the Euler flow in the exterior of a ball in $\mathbb{R}^n$ ($n\geq2$).

In the fast diffusion regime of porous medium flows, the asymptotic behavior has been far less studied and remains poorly understanding, due to the strong singularity in the diffusivity and the pronounced influence of boundary effects. In this paper, we shall show that when $u_->u_+:=0$, if the initial perturbation is suitably small, then \eqref{original model} admits a unique global solution that converges to a shifted viscous shock profile $U(x-st-d(t))$. The shift $d(t)$ is governed by the boundary layer dynamics at $x=0$ and driven by the initial data $u(x,0)$, with the purpose of ensuring that the perturbation vanishes at the boundary.  To treat the singularity of the fast diffusion at $u=0$, we apply the weighted energy method, where the weight functions are artfully selected which are based on those targeted shock waves with singularity.

{\large\textbf{Difficulties}}.
In the aforementioned cases, determining the shift of the asymptotic viscous shock profile to mitigate the effects of the boundary layer is crucial for stability analysis.
Following the determination of the shift--chosen in the spirit of \cite{Liu-Nishihara}: the crucial difficulty lies in the construction of weight functions, tailored to the weighted energy method, for deriving the desired \textit{a priori} estimate to establish global existence of perturbation.

Merely adopting the weight functions from the whole-space Cauchy problem of \eqref{original model} as in \cite{Xu-Mei-Qin-Sheng} to preserve ellipticity of perturbation equation is insufficient, since it neither captures the boundary-induced effects nor guarantees the convergence of the shift-related term $d'(t)U(x-st-d(t))$ arising from the singularity. These considerations necessitate a carefully tailored and delicately constructed weight functions.

Unlike the conventional weights used in \cite{nonconvex,Liu-Nishihara,M. Mei}, the weight functions employed here retain a singularity. In combination with the singular diffusion term $(\frac{u_x}{u^{1-m}})$ at $u=0$, this generates nonzero boundary contributions at $x=+\infty$, requiring careful control of both the boundary terms at $x=0$ and those arising at infinity.

{\large\textbf{Strategies}}.
In order to overcome the singularity for the fast diffusion at $u=0$ and the strong boundary layer for the system, we propose  some new thoughts  for establishing the energy estimates of stability. These strategies are described as follows.
Firstly, to ensure the ellipticity of the perturbation equation of $\phi$, the weight functions in the first- and second-order estimates must be chosen so that
\begin{equation}\label{0sobolev embedding4}
	\sup_{t\in [0,T]}	\left\|\phi_x(\cdot,t)/U\right\|_{L^{\infty}} \leq CN(T)
\end{equation}
holds via the Sobolev embedding inequality, with a sufficiently small $N(T)$, that also plays a key role in controlling the nonlinear estimates of flux and diffusivity. The most natural choice would be $w(U(z))=U^{-2}$, which diverges as $z\to+\infty$, as in \cite{Xu-Mei-Qin-Sheng}. However, the weight functions with singularity need to be modified to $U^{-(1+m)}$ in the first-order estimate and $U^{m-3}$ in the second-order estimate, which, together with the decay of $U(x-st-d(t))$ as $x\rightarrow+\infty$, ensures the convergence of the shift-related term $d'(t)U(x-st-d(t))$ in the higher-order estimates. Moreover, based on the choice of weight functions for the higher-order estimates, we select the weight function $U^{1-3m}$ in $L^2$ estimate, where the condition $m>\frac{1}{2}$ is imposed to control the shift-related term. We remark that this carefully tailored choice of weight functions also guarantees \eqref{0sobolev embedding4} holds.

Secondly, to deal with the boundary-induced effects at $x=0$, a time-dependent algebraic weight $\langle\xi-\xi_\star\rangle^{\beta+j}$ is considered, with $j=0,1,2,3$ corresponding to different orders, $\xi=x-st-d(t)$ and $\xi_\star$ determined by $f'(U(\xi_\star))=s$. In fact, it is necessary to perform estimates up to the third-order derivatives to close the argument.

Thirdly, to avoid the appearance of non-zero boundary terms at $x=+\infty$ when performing integration by parts, we employ a cut-off technique in the estimates, which also used in the Cauchy problem of \eqref{original model} with fast diffusion on the whole space \cite{critical fast diffusion}.

{\large\textbf{Notations}}.
We denote a generic positive constant by $C$ that may vary between lines. We abbreviate the integrals
\[
\int_{0}^{+\infty}f(x)\mathrm{d}x=\int f(x), \mbox{ and }  \int_{0}^{t}\int_{0}^{+\infty}f(x,\tau)\mathrm{d}x\mathrm{d}\tau=\int_{0}^{t}\int f(x,\tau).
\]
 %$H^{k}(\mathbb{R_+})$ denotes the $k$-th order Sobolev space with the norm $\|f\|_{H^k(\mathbb{R_+})}:=\left(\sum_{j=0}^k\|\partial_x^j f\|_{L^2(\mathbb{R_+})}^2\right)^{1 / 2}$.
For a weight function $w(x)>0$, we denote $L^{2}_{w}(\mathbb{R_+})$  the weighted Sobolev space consisting of measurable functions $f$ such that $\sqrt{w}  f \in L^2$ with
\[
\|f\|_{L_w^2(\mathbb{R_+})}:=\left(\int w(x)| f|^2 d x\right)^{1 / 2}.
\]
 We also denote
\begin{align}
\langle x\rangle:= \sqrt{1+x^2},\quad \langle x\rangle_{+}:=\left\{\begin{aligned}
&\sqrt{1+x^2},\quad&\text{ for }x\geq0,\\
&1,\quad	&\text{ for }x\leq0.
\end{aligned}\right.\nonumber
\end{align}
When $w(x)=\langle x\rangle^\beta$ for $\beta>0$, we write $L^2_w(\mathbb{R_+})=L^2_{\langle x\rangle^\beta}(\mathbb{R_+})$ and
$
\|\cdot\|_{L^2_{w}(\mathbb{R_+})}=\|\cdot\|_{L^2_{\langle x\rangle^\beta}(\mathbb{R_+})}.
$
 The weighted Sobolev space $H_{\langle x\rangle^\beta}^k(\mathbb{R_+})$ denotes
the space of measurable functions $f$ such that $\partial_x^j f \in L^2_{\beta+j}$ for $0\leq j\leq k$ with norm
\[
\|f\|_{H_{\langle x\rangle^\beta}^k(\mathbb{R_+})}:=\left(\sum_{j=0}^k  \|\partial_x^j f\|_{L_{\langle x\rangle^{\beta+j}}^2(\mathbb{R_+})}^2 \right)^{1 / 2}.
\]
 For simplicity, we denote $\|\cdot\|:=\|\cdot\|_{L^2(\mathbb{R_+})}$, $\|\cdot\|_w:=\|\cdot\|_{L_w^2(\mathbb{R_+})}$, $\|\cdot\|_{\langle x\rangle^\beta}:=\|\cdot\|_{L^2_{\langle x\rangle^\beta}(\mathbb{R_+})}$ %$\|\cdot\|_{k, w}:=\|\cdot\|_{H_w^k(\mathbb{R_+})}$
and $\|\cdot\|_{k, \langle x\rangle^\beta}:=\|\cdot\|_{H_{\langle x\rangle^\beta}^k(\mathbb{R_+})}$.

Furthermore, $L_{\beta,w}^2(\mathbb{R_+})$ denotes the space of measurable functions $f$ such that $\sqrt{\langle x\rangle^\beta w} f \in L^2$ with norm 
\[
\|f\|_{L_{\langle x\rangle^\beta,w}^2(\mathbb{R_+})}:=\left( \int\langle x\rangle^\beta w(x) f^2 d x\right)^{1 / 2}.
\]
 $H_{\langle x\rangle^\beta,w}^k(\mathbb{R_+})$ denotes
the space of measurable functions $f$ such that $\partial_x^j f \in L^2_{\langle x\rangle^{\beta+j},w}$ for $0\leq j\leq k$ with norm 
\[
\|f\|_{H_{\langle x\rangle^\beta,w}^k(\mathbb{R_+})}\!\!:=\Big(\!\sum_{j=0}^k  \|\partial_x^j f\|_{L_{\langle x\rangle^{\beta+j},w}^2}^2 \Big)^{1 / 2}.
\]

\vspace{0.3cm}
The organization of this paper is as follows. After stating the notations, in Section \ref{section main results} we review the existence of viscous shock profiles for the Cauchy problem of \eqref{original model}, and state the main result on the asymptotic convergence to the viscous shock profile on the half-space $ \mathbb{R_+}$. In Section \ref{section stability}, we reformulate our problem and establish the global existence and the asymptotic behavior of the solution by combining local existence with \textit{a priori} estimate.

\section{Preliminaries and main results}\label{section main results}
Integrating \eqref{shock profile equ} in $z$, gives an ordinary differential equation of viscous shock wave
\begin{equation}\label{ODE}
	U_z=U^{1-m}\left[f(U)-f(u_{-})-s(U-u_{-})\right]=U^{1-m}\left(f(U)-sU\right).
\end{equation}
The existence of a viscous shock wave can be stated precisely as follows. See \cite{Xu-Mei-Qin-Sheng} for the proof.
\begin{lemma}[\cite{Xu-Mei-Qin-Sheng}, Theorem 2.1]\label{Existence of the shock profile}
	Let $0<m<1$.
	\begin{enumerate}
		\item [(i)] Suppose the Rankine-Hugoniot condition \eqref{R-H} and the generalized shock condition \eqref{Lax's} hold. Then there exists a unique  smooth monotone solution $U(z)$ (up to a shift) of \eqref{shock profile equ} satisfying
		\begin{equation}\nonumber
			U_z(z)<0,\quad \forall z\in \mathbb{R}.
		\end{equation}
		\item [(ii)]  Moreover, if $f'(0)<s<f'(u_-)$, it holds that
		\begin{equation}\label{f''>0}
			\left|U(z)-0\right|\sim|z|^{-\frac{1}{1-m}},\quad\text{as}~z\rightarrow+\infty,
		\end{equation}	
		\begin{equation}\label{decay u_-}		
			\left|U(z)-u_-\right|\sim\mathrm{e}^{-\lambda_-|z|},\quad\text{as}~z\rightarrow-\infty,
		\end{equation}
	with $\lambda_-=u_-^{1-m}(f'(u_-)-s)$.
		%If $f'(0)=s<f'(u_-)$, it holds that
		%\begin{equation}\nonumber
		%	\left|U(z)-0\right|\sim|z|^{-\frac{1}{k_{+}+1-m}},\quad\text{as}~z\rightarrow+\infty.
		%\end{equation}
		 %If $f'(0)<s=f'(u_-)$, it holds that
		%\begin{equation}\nonumber
		%	\left|U(z)-u_-\right|\sim|z|^{-\frac{1}{k_{-}}},	\quad\text{as}~z\rightarrow-\infty,
		%\end{equation}
	%	with $\lambda_-=u_-^{1-m}(f'(u_-)-s)$ and $k_{\pm}\geq1$.
	\end{enumerate}
\end{lemma}	

\begin{remark}
If the flux $f$ is convex, then the generalized shock condition	\eqref{Lax's} is equivalent to the well-known Lax's condition
	 \begin{equation}\label{degenerate lax's}
	 	f'(0)< s < f'(u_-).
	 \end{equation}
\end{remark}

\begin{remark}\label{remark 2.1}
	 Since $f$ is smooth and $f(0)=0$, it follows from \eqref{ODE} that
	 \begin{equation}\label{q1}
	 	 |U_z(z)|\leq C U^{2-m}(z),\quad \forall z\in \mathbb{R}.
	 \end{equation}
Moreover, a direct calculation gives
 \begin{equation}\label{direct Uzz}
 	U_{zz}=(1-m)\frac{U_z^2}{U}+(f'(U)-s)U^{1-m}U_z,
 \end{equation}
which implies that
 \begin{equation}\label{q2}\nonumber
 	|U_{zz}(z)| \leq  CU^{3-2m}(z),\quad \forall z\in\mathbb{R}.
 \end{equation}
Similarly, we have
\begin{equation}\label{q3}
	|U_{zzz}(z)|\leq C U^{4-3m}(z),\quad \forall z\in \mathbb{R},
\end{equation}
and
\begin{equation}\label{q4}
	|U_{zzzz}(z)|\leq C U^{5-4m}(z),\quad \forall z\in \mathbb{R}.
\end{equation}
\end{remark}	

Before stating our main result, let us define
	\begin{equation}\nonumber
		\phi_0(x):=\int_x^{+\infty}\left( u_0(y)-U(y-d_0)\right) d y,
	\end{equation}
where $d_0$ is the constant uniquely determined by
\begin{equation}\nonumber
	\int_{-d_0}^{0}U(y) \mathrm{d}y=\int_{0}^{+\infty}\left(u_0(x)-U(x)\right) \mathrm{d} x,
\end{equation}
and denote
\begin{eqnarray}\label{alphai}
	\alpha_1=\frac{3m-1}{1-m},~ \alpha_2=\frac{1+m}{1-m},~ \alpha_3=\frac{3-m}{1-m},~\alpha_4=\frac{2m}{1-m},~\alpha_5=\frac{2}{1-m},~ \alpha_6=\frac{4-2m}{1-m}.
\end{eqnarray}
	
\begin{theorem}[Stability of shock profiles]\label{stability theorem} For $\frac{1}{2}<m<1$ and $f\in C^4$, let
	  \eqref{R-H} with $s>0$ and \eqref{degenerate lax's} hold, and let $u_0(x)-U(x)\in L^1(\mathbb{R_+})$.  
When  $\phi_0\in  L_{\langle x\rangle_{+}^{\alpha_1}}^2\cap H_{\langle x\rangle^{\beta}}^3$, $\phi_{0x}\in L_{\langle x\rangle_{+}^{\alpha_2}}^2$ and $\phi_{0xx}\in L_{\langle x\rangle_{+}^{\alpha_3}}^2$ and that
\begin{equation}\label{new1}\nonumber
\|\phi_0\|_{\langle x\rangle_{+}^{\alpha_1}}+\|\phi_0\|_{3,\langle x\rangle^{\beta}}+\left\|u_0-U\right\|_{\langle x\rangle_{+}^{\alpha_2}}+\left\|(u_0-U)_x\right\|_{ \langle x\rangle_{+}^{\alpha_3}}+d_0^{-\frac{1}{2}}\leq \epsilon_1,
 \end{equation}
provided with $0<\epsilon_1\ll 1$, $0<\beta\leq\frac{3m-1}{1-m}$, then there exists a unique global solution $u(x,t)$ of \eqref{original model} and a time-dependent shift function $d(t)$ satisfying
\begin{equation}\label{new2}
	\begin{cases}
		-d'(t)U(-st-d(t))+f(u_-)-f\left(U(-st-d(t)) \right)=\left(u^{m-1}u_{x}-U^{m-1}U_{x}\right)\big|_{x=0},\\	
		d(0)=d_0,
	\end{cases}
\end{equation}		
such that
\begin{equation}\nonumber
u-U  \in C\left([0, \infty) ; L_{\langle x-st\rangle_{+}^{\alpha_2}}^2\cap H_{\langle x-st\rangle^{\beta+1}}^2\right)\cap L^2\left((0, \infty) ;  L_{\langle x-st\rangle_{+}^{\alpha_4}}^2\cap H_{\langle x-st\rangle^{\beta},~\langle x-st\rangle_{+}}^3\right),
\end{equation}	
\begin{equation}\nonumber
(u\!-\!U)_x \in C\left([0, \infty) ; L_{\langle x-st\rangle_{+}^{\alpha_3}}^2\right)\cap L^2\left((0, \infty) ; L_{\langle x-st\rangle_{+}^{\alpha_5}}^2\right),(u\!-\!U)_{xx} \in  L^2\left((0, \infty) ; L_{\langle x-st\rangle_{+}^{\alpha_6}}^2\right),		
\end{equation}
 and in particular, the convergence of $u(x,t)$ to the shifted shock wave $U(x-st-d(t)$:
\begin{equation}\label{new-2}
	\sup _{x \in \mathbb{R_+}}|u(x, t)-U(x-s t-d(t))| \rightarrow 0, \quad \text { as } t \rightarrow+\infty,
\end{equation}
\begin{equation}\label{new-3}
 d(t)\rightarrow d_{\infty},\text{ as } t\rightarrow \infty,
\end{equation}
for some constant $d_{\infty}$. 
\end{theorem}	

\begin{remark}
The restriction of $\frac{1}{2}<m<1$ is due to the following technical issue.
 In the basic $L^2$ estimate, to control the time-dependent shift term $d'(t)U(x-st-d(t))$ of the perturbation, particularly for the region of $x>st+d(t)$, we need to utilize the decay of $U(x-st-d(t))$ at infinity to ensure the convergence:
 \begin{equation}\nonumber
  \int_{s\tau+d(\tau)}^{+\infty}U^{m}\sim C \int_{s\tau+d(\tau)}^{+\infty}|x-s\tau-d(\tau)|^{\frac{-m}{1-m}}\leq C.
 \end{equation}
 Hence, the condition $m>1/2$ is required. For more details, we refer to the discussion  in \eqref{eq15} below. For $0<m<\frac{1}{2}$, unfortunately the stability of shock waves is still open.
\end{remark}

\section{Asymptotic stability}	\label{section stability}
\subsection{Reformulation of original problem}
Assume that the condition \eqref{R-H} holds with $s>0$, then by $f(0)=0$, we get
\begin{equation}\nonumber
	f(u_-)>0.
\end{equation}
Moreover, it follows from the Lax's condition \eqref{degenerate lax's} that
\begin{equation}\nonumber
\left( \frac{ f(U)-f(u_-)}{U-u_-}-s\right)(U-u_-)	<0 ,
\end{equation}
which along with $U\in(0,u_-)$ implies that
\begin{equation}\label{f'(u_-)}
 f'(u_-)>s>0.
\end{equation}
Let $U(x-st-d(t))$ be the asymptotic profile of the solution $u(x,t)$ as $t\rightarrow+\infty$, where $U(\cdot)$ is the shock profile of \eqref{original model} and $d(t)$ is a time-dependent shift. Hence, $U(x-st-d(t))$ satisfies
\begin{equation}\label{U equ}\nonumber
  U_t+d'(t)U_x+f(U)_{x}=\left(U^{m-1}U_{x}\right)_{x}.	
\end{equation}
This along with \eqref{original model} gives the perturbation equation
\begin{equation}\label{perturbation}
	\left(u-U\right)_t-d'(t)U_x+\left(f(u)-f(U)\right)_{x}=\left(u^{m-1}u_{x}-U^{m-1}U_{x}\right)_{x}.	
\end{equation}	
It is expected that
\[
\lim\limits_{x\rightarrow+\infty}u^{m-1}u_x(x,t)=\lim\limits_{x\rightarrow+\infty}U^{m-1}U_x(x-st-d(t))=0.
\]
Then, by $u(0,t)=u_-$, integrating \eqref{perturbation} in $x$ over $(0,+\infty)$ yields	
\begin{equation}\label{eq1}
	\begin{aligned}
		 \frac{\mathrm{d}}{\mathrm{d}t}&\int _{0}^{+\infty}\left(u-U\right)d x+d'(t)U(-st-d(t))-\left(f(u_-)-f(U(-st-d(t)))\right)\\&=-\left(u^{m-1}u_{x}-U^{m-1}U_{x}\right)\big|_{x=0} .	
	\end{aligned}
\end{equation}	
We define $d(t)$ as the solution to the following initial value problem
\begin{equation}\label{eq2}
	\begin{cases}
	-d'(t)U(-st-d(t))+f(u_-)-f\left(U(-st-d(t)) \right)=\left(u^{m-1}u_{x}-U^{m-1}U_{x}\right)\big|_{x=0},\\	
	d(0)=d_0,
	\end{cases}
\end{equation}	
where the initial data $d_0$ is given by
\begin{equation}\label{eq3}
	 \int _{0}^{+\infty}\left(u_0(x)-U(x-d_0)\right)d x=0.
\end{equation}	
Then, combining \eqref{eq1}-\eqref{eq2}, we get for all $t>0$ that
\begin{equation}\label{eq4}
	\int _{0}^{+\infty}\left(u(x,t)-U(x-st-d(t))\right)d x=0.
\end{equation}	
Therefore, we set
\begin{equation}\label{decompose}
	\phi(x, t):=-\int_x^{+\infty}\left(u(y, t)-U\left(y-s t-d(t)\right) \right)\mathrm{d} y,\quad\forall	x\in \mathbb{R_+},~t>0,
\end{equation}
and may expect $\phi(x,t)$ to be integrable in $L^2(\mathbb{R}_+)$. From
\eqref{eq4}, we have
\begin{equation}\label{eq5}
	\phi(+ \infty,t)=\phi(0,t)=0,\quad \text{for all }t>0.	
\end{equation}	
Now substituting \eqref{decompose} into \eqref{perturbation}, integrating the equation over $(x,+\infty)$, and using \eqref{eq5}, we reformulate the problem \eqref{original model} as the following system for $(\phi(x,t),d(t))$:
\begin{equation}\label{d equ}
	\begin{cases}
\phi_t+f'(U(x-st-d(t))) \phi_x-\left(\frac{\phi_x}{U^{1-m}(x-st-d(t))}\right)_x \\
 \ \ \ =d'(t)U(x-st-d(t))+F+\frac{G_x}{m},\\
-d'(t)U(-st-d(t))+f(u_-)-f\left(U(-st-d(t)) \right)-\left(\frac{\phi_x}{U^{1-m}(-st-d(t))}\right)_x\Big|_{x=0} \\
\ \ \ =\frac{G_x}{m}\Big|_{x=0}, 		
	\end{cases}
\end{equation}
with the initial and boundary values		
\begin{equation}\label{phiequ}
\phi(x,0):=\phi_0(x)=\int_x^{+\infty}\left( u_0(y)-U(y-d_0)\right) d y,
\end{equation}
\begin{equation}\label{d0}
	d(0)=d_0, 	
\end{equation}	
\begin{equation}\label{phiboundary}
	\phi(+ \infty,t)=\phi(0,t)=0,	
\end{equation}
where
\begin{equation}\label{F}\nonumber
	F:=-\left[f\left(U(x-st-d(t) )+\phi_x\right)-f(U(x-st-d(t) ))-f^{\prime}(U(x-st-d(t) )) \phi_x\right],
\end{equation}
and
\begin{equation}\label{G}
	G:= \left(U(x-st-d(t) )+\phi_x\right)^{m}- U^{m}(x-st-d(t) )-\frac{m\phi_x}{U^{1-m}(x-st-d(t) )}.
\end{equation}

\begin{remark}
	 Indeed, $d_0$ is uniquely determined as below. Assume that $u_0(x)-U(x)\in L^1(\mathbb{R_+})$, we get from \eqref{eq3} that
\begin{equation}\nonumber
	\begin{aligned}
		0&=\int_{0}^{+\infty}\left(u_0(x)-U(x)\right) \mathrm{d} x+\int_{0}^{+\infty}\left(U(x)-U\left(x-d_0\right)\right)\mathrm{d} x\\
		&=\int_{0}^{+\infty}\left(u_0(x)-U(x)\right) \mathrm{d} x+\int_{0}^{+\infty}\int_{-d_0}^{0} U'(x+y)\mathrm{d}y\mathrm{d} x\\
		&=\int_{0}^{+\infty}\left(u_0(x)-U(x)\right) \mathrm{d} x+\int_{-d_0}^{0}\int_{0}^{+\infty} U'(x+y)\mathrm{d} x\mathrm{d}y\\
		&=\int_{0}^{+\infty}\left(u_0(x)-U(x)\right) \mathrm{d} x-\int_{-d_0}^{0}U(y) \mathrm{d}y,
	\end{aligned}
\end{equation}	
which gives	
\begin{equation}\label{3.18}
	\int_{-d_0}^{0}U(y) \mathrm{d}y=\int_{0}^{+\infty}\left(u_0(x)-U(x)\right) \mathrm{d} x.
\end{equation}	
Since $U$ is monotonically decreasing, there exists a unique $d_0$ satisfying \eqref{3.18}.	
\end{remark}

\subsection{Energy estimates}
Let us define the following solution space for  $\phi(x,t)$ and  $d(t)$:
\begin{equation}\label{space tilde X}
	\begin{aligned}
		\tilde{X}(0, T):=&\left\{\left( \phi,d\right) \mid \phi \in C\left([0, T] ; L_{\langle x-st\rangle_+^{\alpha_1}}^2 \cap H_{\langle x-st\rangle^{\beta}}^3\right),\phi_{x xx} \in  L^2\left((0, T) ; L_{\langle x-st\rangle_+^{\alpha_6}}^2\right),\right.\\&\left. \qquad\qquad\phi_x \in C\left([0, T] ; L_{\langle x-st\rangle_+^{\alpha_2}}^2\right)\cap L^2\left((0, T) ;  L_{\langle x-st\rangle_+^{\alpha_4}}^2\cap H_{\langle x-st\rangle^{\beta},\langle x-st\rangle_+}^3\right),\right.\\&\left. \qquad\qquad\phi_{x x} \in C\left([0, T] ; L_{\langle x-st\rangle_+^{\alpha_3}}^2\right)\cap L^2\left((0, T) ; L_{\langle x-st\rangle_+^{\alpha_5}}^2\right),\right.\\&\left. \qquad\qquad\left(d_0+t\right)^{\beta+3}|\phi_{xx}(0,t)|^2+\int_{0}^{t}\left(d_0+\tau\right)^{\beta+3}|\phi_{xx}(0,\tau)|^2\mathrm{d} \tau<+\infty,\right.\\&\left. \qquad\qquad d(t)\in C^1\left( [0,T]\right),~ |d(t)-d_0|<+\infty\right\},
	\end{aligned}
\end{equation}
where $\alpha_i$ for $i=1,\cdots,6$ are given in \eqref{alphai},
and $\beta$ is a constant such that
\begin{equation}\nonumber
	0<\beta\leq\frac{3m-1}{1-m}.
\end{equation}
Denote
\begin{eqnarray}
		\tilde{N}(T)&:=&\sup_{t\in [0,T]}\Big\{\|\phi(\cdot,t)\|_{\langle x-st\rangle_+^{\alpha_1}}+\|\phi_x(\cdot,t)\|_{\langle x-st\rangle_+^{\alpha_2}}+\|\phi_{xx}(\cdot,t)\|_{\langle x-st\rangle_+^{\alpha_3}} \nonumber \\
&&\qquad \quad +\|\phi(\cdot,t)\|_{3,\langle x-st\rangle^{\beta}}+(d_0+t)^{\frac{\beta+3}{2}}|\phi_{xx}(0,t)| \nonumber \\
&&\qquad \quad +\Big( \int_{0}^{t} (d_0+\tau)^{\beta+3}|\phi_{xx}(0,\tau)|^2\mathrm{d} \tau \Big)^{\frac{1}{2}}\Big\}.	
\end{eqnarray}

Now we state Theorem \ref{stability theorem} in  the following equivalent theorem for $\phi(x,t)$ and $d(t)$.

\begin{theorem}\label{phi stability}
	Suppose that $\frac{1}{2}<m<1$, $f\in C^4$, the conditions \eqref{R-H} with $s>0$ and \eqref{degenerate lax's} hold.
	Then there exists a positive constant $\epsilon_3$ such that if $\phi_0\in  L_{\langle x\rangle_+^{\alpha_1}}^2\cap H_{\langle x\rangle^{\beta}}^3$, $\phi_{0x}\in L_{\langle x\rangle_+^{\alpha_2}}^2$, $\phi_{0xx}\in L_{\langle x\rangle_+^{\alpha_3}}^2$ and that $\tilde{N}(0)+d_0^{-\frac{1}{2}}\leq \epsilon_3$
	, then the reformulated problem \eqref{d equ} with \eqref{phiequ}-\eqref{phiboundary} has a unique global solution $(\phi,d)\in \tilde{X}(0,\infty)$ satisfying
	\begin{eqnarray}\label{priori estimate}
			&&\|\phi(\cdot,t)\|_{\langle x-st\rangle_+^{\alpha_1}}^2 + \|\phi_x(\cdot,t)\|_{ \langle x-st\rangle_+^{\alpha_2}}^2
+\|\phi_{xx}(\cdot,t)\|_{ \langle x-st\rangle_+^{\alpha_3}}^2 
+\|\phi(\cdot,t)\|_{3,\langle x-st\rangle^{\beta}}^2 \nonumber \\
&& \ \ \ +(d_0+t)^{\beta+3}|\phi_{xx}(0,t)|^2
+\int_0^t\Big(\|\phi_x(\cdot,\tau)\|_{\langle x-st\rangle_+^{\alpha_4}}^2+\|\phi_{xx}(\cdot,\tau)\|_{ \langle x-st\rangle_+^{\alpha_5}}^2 \nonumber \\
&& \ \ \ +\|\phi_{xxx}(\cdot,\tau)\|_{ \langle x-st\rangle_+^{\alpha_6}}^2 +\|\phi_x(\cdot,t)\|_{H_{\langle x-st\rangle^{\beta},\langle x-st\rangle_+}^3}^2
+(d_0+\tau)^{\beta+3}|\phi_{xx}(0,\tau)|^2\Big) \nonumber \\
&&			\leq C(1+d_0^{\beta+3})\Big(\|\phi_0\|_{\langle x\rangle_+^{\alpha_1}}^2+\|\phi_{0x}\|_{ \langle x\rangle_+^{\alpha_2}}^2
+\|\phi_{0xx}\|_{ \langle x\rangle_+^{\alpha_3}}^2+\|\phi_0\|_{3,\langle x\rangle^{\beta}}^2\Big)+Cd_0^{-1} \nonumber \\
&&\leq C(1+d_0^{\beta+3})N^2(0)+Cd_0^{-1},
	\end{eqnarray}
	where $\alpha_i$ ($i=1,\cdots,6$) are defined by \eqref{alphai}, and that
	\begin{equation}\label{asymptotic}
		\sup _{x \in \mathbb{R_+}}|\phi_{x}(x,t)| \rightarrow 0, \quad \text { as } t \rightarrow+\infty,
	\end{equation}
	\begin{equation}\label{asymptotic1}
		 d'\in L^1(0,\infty) \text{ and } d(t)\rightarrow d_{\infty},\text{ as } t\rightarrow \infty,
	\end{equation}
for some constant $d_{\infty}$.	
\end{theorem}

Note that the global existence for $\phi(x,t)$ and $d(t)$ can be established by the continuity argument based on  the local existence and the \textit{a priori} estimate. Now we are going to sketch the proof for  the local existence of $(\phi(x,t),d(t))$ as follows. The main issue of the paper is to establish the {\it a priori} estimate.

Let us define the sequence of the approximate solutions $(\phi^{(n)},d^{(n)})$:
\begin{equation}\label{new-1}
	\begin{cases}
		\displaystyle{\phi_t^{(n)}+f'(U(x-st-d^{(n-1)}(t))) \phi_x^{(n-1)}-\left(\frac{\phi_x^{(n)}}{U^{1-m}(x-st-d^{(n-1)}(t))}\right)_x} \\
\displaystyle{\quad =d^{(n-1)}_t(t) U(x-st-d^{(n-1)}(t))+F^{(n-1)}+\frac{G_x^{(n-1)}}{m},}\\
\displaystyle{		-d^{(n)}_t(t) U(-st-d^{(n-1)}(t))+f(u_-)-f\left(U(-st-d^{(n-1)}(t)) \right) } \\
\displaystyle{ \quad -\left(\frac{\phi_x^{(n-1)}}{U^{1-m}(-st-d^{(n-1)}(t))}\right)_x\Big|_{x=0}
=\frac{G_x^{(n-1)}}{m}\Big|_{x=0}},
	\end{cases}
\end{equation}
where,
\begin{eqnarray}
	F^{(n-1)}&=&-\Big[f\left(U(x-st-d^{(n-1)}(t) )+\phi_x^{(n-1)}\right) \nonumber \\
&&-f(U(x-st-d^{(n-1)} (t)))-f^{\prime}(U(x-st-d^{(n-1)}(t))) \phi_x^{(n-1)}\Big],\nonumber
\end{eqnarray}
and
\begin{eqnarray}
	G^{(n-1)}&=& \left(U(x-st-d^{(n-1)}(t) )+\phi_x^{(n-1)}\right)^{m}- U^{m}(x-st-d^{(n-1)}(t) ) \nonumber \\
&&-\frac{m\phi_x^{(n-1)}}{U^{1-m}(x-st-d^{(n-1)}(t))}.\nonumber
\end{eqnarray}
For any $T>0$, the solution space $X^{(n)}(0,T)$ is defined as
\begin{equation}\nonumber
	\begin{aligned}
		X^{(n)}(0, T):=&\left\{(\phi^{(n)},d^{(n)}) \mid \phi^{(n)} \in C\left([0, T] ; L_{w_1}^2 \cap H_{\langle x-st-d^{(n-1)}(t)\rangle^{\beta}}^3\right),\phi^{(n)}_{x xx} \in  L^2\left((0, T) ; L_{w_6}^2\right), \right.\\
&\left. \qquad\qquad\qquad\phi^{(n)}_x \in C\left([0, T] ; L_{w_2}^2\right)\cap L^2\left((0, T) ;  L_{w_4}^2\cap H_{\langle x-st-d^{(n-1)}(t)\rangle^{\beta},w_7}^3\right),\right.\\
&\left. \qquad\qquad\qquad\phi^{(n)}_{x x} \in C\left([0, T] ; L_{w_3}^2\right)\cap L^2\left((0, T) ; L_{w_5}^2\right),\right.\\
&\left. \qquad\qquad\qquad\left(d_0+t\right)^{\beta+3}|\phi^{(n)}_{xx}(0,t)|^2+\int_{0}^{t}\left(d_0+\tau\right)^{\beta+3}|\phi^{(n)}_{xx}(0,\tau)|^2\mathrm{d} \tau<+\infty,\right.\\
&\left. \qquad\qquad \qquad d^{(n)}(t)\in C^1\left( [0,T]\right), |d^{(n)}(t)-d_0|<+\infty\right\},
	\end{aligned}
\end{equation}
where
\begin{equation}\label{wU}
	\begin{aligned}
		&\quad w_1=U^{1-3m},~ w_2=U^{-(1+m)},~ w_3=U^{m-3},\\&		
		w_4=U^{-2m},~ w_5=U^{-2},~ w_6=U^{2m-4},~w_7=U^{m-1},
	\end{aligned}
\end{equation}
with $U=U(x-st-d^{(n-1)}(t))$, and its norm is denoted by
\begin{equation}\nonumber
	\begin{aligned}
		N^{(n)}(T):&=\sup_{t\in [0,T]}\left\{\left\|\phi^{(n)}(\cdot,t)\right\|_{w_1}+\left\|\phi^{(n)}_x(\cdot,t)\right\|_{w_2}+\left\|\phi^{(n)}_{xx}(\cdot,t)\right\|_{w_3}
+\left\|\phi^{(n)}(\cdot,t)\right\|_{3,\langle x-st-d^{(n-1)}(t)\rangle^{\beta}}\right.\\
&\left.\quad+\left(d_0+t\right)^{\frac{\beta+3}{2}}|\phi^{(n)}_{xx}(0,t)|+\left( \int_{0}^{t} \left(d_0+\tau\right)^{\beta+3}|\phi^{(n)}_{xx}(0,\tau)|^2\mathrm{d} \tau\right)^{\frac{1}{2}}\right\}.		
	\end{aligned}
\end{equation}

Given initial data $(\phi_{0},d_0)=:(\phi^{(0)},d^{(0)})$ satisfying $(\phi_0,d_0)\in X^{(0)}(0,0)$ with $T=0$,  since the equations \eqref{new-1} are linear, by a straightforward but tedious computation,
  it can be proved that, for given $(\phi^{(n-1)},d^{(n-1)})\in X^{(n-1)}(0,T)$ with $T>0$, the iterative sequence determined by \eqref{new-1} yields a unique solution $(\phi^{(n)},d^{(n)})\in X^{(n)}(0,T)$,  and that the sequence converges to $(\phi,d)\in X(0,T_0)$ for some small number $T_0\ll 1$, where $X(0,T_0)=\lim_{n\to\infty} X^{(n)}(0,T_0)$ is defined by
\begin{equation}\label{space X}
	\begin{aligned}
		X(0, T_0):=&\left\{(\phi,d) \mid \phi \in C\left([0, T_0] ; L_{w_1}^2 \cap H_{\langle x-st-d(t)\rangle^{\beta}}^3\right),\phi_{x xx} \in  L^2\left((0, T_0) ; L_{w_6}^2\right), \right.\\
&\left. \qquad\qquad\phi_x \in C\left([0, T_0] ; L_{w_2}^2\right)\cap L^2\left((0, T_0) ;  L_{w_4}^2\cap H_{\langle x-st-d(t)\rangle^{\beta},w_7}^3\right),\right.\\
&\left. \qquad\qquad\phi_{x x} \in C\left([0, T_0] ; L_{w_3}^2\right)\cap L^2\left((0, T_0) ; L_{w_5}^2\right),\right.\\
&\left. \qquad\qquad\left(d_0+t\right)^{\beta+3}|\phi_{xx}(0,t)|^2+\int_{0}^{t}\left(d_0+\tau\right)^{\beta+3}|\phi_{xx}(0,\tau)|^2\mathrm{d} \tau<+\infty,\right.\\
&\left. \qquad\qquad d(t)\in C^1\left( [0,T_0]\right), |d(t)-d_0|<+\infty\right\},
	\end{aligned}
\end{equation}
subject to
\begin{equation}\label{N(t)}
	\begin{aligned}
		N(T):&=\sup_{t\in [0,T]}\left\{\left\|\phi(\cdot,t)\right\|_{w_1}+\left\|\phi_x(\cdot,t)\right\|_{w_2}+\left\|\phi_{xx}(\cdot,t)\right\|_{w_3}+\left\|\phi(\cdot,t)\right\|_{3,\langle x-st-d(t)\rangle^{\beta}}\right.\\&\left.\quad+\left(d_0+t\right)^{\frac{\beta+3}{2}}|\phi_{xx}(0,t)|+\left( \int_{0}^{t} \left(d_0+\tau\right)^{\beta+3}|\phi_{xx}(0,\tau)|^2\mathrm{d} \tau\right)^{\frac{1}{2}}\right\},		
	\end{aligned}
\end{equation}
where $w_i(U)(i=1,...,7)$ are defined by \eqref{wU} with $U=U(x-st-d(t))$.

Thus, we can  prove the following local existence and uniqueness theorem. The detail of proof is omitted.

\begin{proposition}[Local existence]\label{local existence}
	Suppose $\phi_0\in  L_{w_{1,0}}^2\cap H_{\langle x-d_0\rangle^{\beta}}^3$, $\phi_{0x}\in L_{w_{2,0}}^2$ and $\phi_{0xx}\in L_{w_{3,0}}^2$, where $w_{1,0}=w_1|_{t=0}=U^{1-3m}(x-d_0)$, $w_{2,0}=w_2|_{t=0}=U^{-(1+m)}(x-d_0)$, and $w_{3,0}=w_3|_{t=0}=U^{m-3}(x-d_0)$. For any $\epsilon_{0}>0$, there exists a positive constant $T_0$ depending on $\epsilon_{0}$ such that if $N(0)+d_0^{-\frac{1}{2}}\leq \epsilon_{0}$, then the problem \eqref{d equ}-\eqref{d0} has a unique solution $(\phi,d)\in X(0,T_0)$ satisfying $N(t)\leq 2(N(0)+d_0^{-\frac{1}{2}})$ for any $0\leq t\leq T_0$.
\end{proposition}
\begin{remark}
	 Due to the uniform boundedness of $d(t)$ on $[0,T_0]$, we find that $x-st-d(t)\sim x-st$ for $0\leq t\leq T_0$. As a result, the solution space $X(0,T_0)$ and $\tilde{X}(0,T_0)$ are equivalent, as are $N(T_0)$ and $\tilde{N}(T_0)$.
\end{remark}

 To prove Theorem \ref{phi stability}, the crucial step is  to establish the following \textit{a priori} estimate.

\begin{proposition}[A priori estimate]\label{proposition priori estimate}
Let $(\phi,d)\in X(0,T)$ be a solution of \eqref{d equ} for a positive constant $T$. Then there exists a positive constant $\epsilon_{2}$ independent of $T$, such that if $N(T)+d_0^{-\frac{1}{2}}\leq \epsilon_{2}$, then the estimate \eqref{priori estimate} holds for all $t\in [0,T]$.	
\end{proposition}
Before establishing the \textit{a priori} estimate in Proposition \ref{proposition priori estimate}, we first present some preliminary calculations.

\begin{lemma}
	Under the same assumptions of Proposition \ref{proposition priori estimate}, there exists a constant $C>0$ independent of $T$ such that
	\begin{equation}\label{sobolev embedding3}
		\sup_{t\in [0,T]}	\left\|\phi(\cdot,t)/U^{m}\right\|_{L^{\infty}} \leq CN(T),
	\end{equation}
	and
	\begin{equation}\label{sobolev embedding4}
		\sup_{t\in [0,T]}	\left\|\phi_x(\cdot,t)/U\right\|_{L^{\infty}} \leq CN(T).
	\end{equation}
\end{lemma}
%\begin{equation}
%	\sup_{t\in [0,T]}	\left\|\langle\xi-\xi_\star\rangle ^{\frac{\beta+1 }{2}}\phi_x(\cdot,t)\right\|_{L^{\infty}}\leq CN(T).
%\end{equation}
\begin{proof}
	Since $\phi(0,t)=0$, then by virtue of \eqref{q1} and \eqref{N(t)}, we get
	\begin{equation}\nonumber
		\begin{aligned}
			\frac{\phi^2(x,t)}{U^{2m}( x-st-d(t))}&=\int_{0}^{x}\left(\frac{\phi^2(y,t)}{U^{2m}(y-st-d(t))} \right)_y\mathrm{d}y=\int_{0}^{x}\left(\frac{2\phi\phi_{y}}{U^{2m}}-\frac{2mU_y\phi^2}{U^{2m+1}} \right)\mathrm{d}y\\&\leq C\left( \int_{0}^{+\infty}\frac{\phi^2}{U^{3m-1}}\mathrm{d}y+\int_{0}^{+\infty}\frac{\phi_{y}^2}{U^{1+m}}\mathrm{d}y\right)\leq CN^2(T),~\forall(x,t)\in\mathbb{R_+}\times [0,T],
		\end{aligned}
	\end{equation}
	which implies \eqref{sobolev embedding3}. Furthermore, given that $\phi_x\in L_{w_2}^2(\mathbb{R_+})$, one has for any $(x,t)\in\mathbb{R_+}\times [0,T]$,
	\begin{equation}\nonumber
		\begin{aligned}
			\phi_x^2(x,t)&=-\int_{x}^{+\infty}\left(\phi_y^2(y,t) \right)_y\mathrm{d}y=-\int_{x}^{+\infty}2\phi_y\phi_{yy}\mathrm{d}y\\&\leq C \int_{0}^{+\infty}\frac{\phi_y^2}{U^{1+m}}\mathrm{d}y+C\int_{0}^{+\infty}\frac{\phi_{yy}^2}{U^{3-m}}\mathrm{d}y\leq CN^2(T).
		\end{aligned}
	\end{equation}
	This gives
	\begin{equation}\label{sobolev embedding2}
		\sup_{t\in [0,T]}	|\phi_x(0,t)|\leq CN(T).
	\end{equation}
	Thus, it follows from \eqref{N(t)} and \eqref{sobolev embedding2} that
	\begin{align}
		\frac{\phi_{x}^2(x,t)}{U^2( x-st-d(t))}&=\int_{0}^{x}\left(\frac{\phi_{y}^2(y,t)}{U^2(y-st-d(t))} \right)_y\mathrm{d}y+\frac{\phi_{y}^2(0,t)}{U^2(-st-d(t))}\nonumber\\&=\int_{0}^{x}\left(\frac{2\phi_{y}\phi_{yy}}{U^2}-\frac{2U_y\phi_{y}^2}{U^3} \right)\mathrm{d}y+\frac{\phi_{y}^2(0,t)}{U^2(-st-d(t))}\nonumber\\&\leq C\left( \int_{0}^{+\infty}\frac{\phi_{y}^2}{U^{1+m}}\mathrm{d}y+\int_{0}^{+\infty}\frac{\phi_{yy}^2}{U^{3-m}}\mathrm{d}y+\phi_{y}^2(0,t)\right)\leq CN(T)\nonumber,
	\end{align}
	which implies \eqref{sobolev embedding4}.
\end{proof}

\begin{lemma}\label{lemma d(t)}
	\begin{enumerate}
		\item [(i)] Given a constant $d_0$ determined by \eqref{3.18}, there exists a small constant $\epsilon_{4}>0$ such that if $N(T)\leq\epsilon_{4} $, then
		\begin{equation}\label{3.37}
			-\frac{st}{2}-d(t)\leq -d_0,\quad \forall0\leq t\leq T.
		\end{equation}
		\item [(ii)] Furthermore, it holds that
		\begin{equation}\label{e1}
			|d'(t)|\leq C\left( e^{-\gamma\left(d_0+t\right)}+|\phi_{xx}(0,t)|\right),
		\end{equation}
		\begin{equation}\label{e2}
			\int_{0}^{t}\left(d_0+\tau\right)^{\beta+3}	|d'(\tau)|^2\leq C\left( e^{-2\gamma d_0}+\int_{0}^{t}\left(d_0+\tau\right)^{\beta+3}|\phi_{xx}(0,\tau)|^2\right),
		\end{equation}
		where $\gamma>0$ is a constant.
	\end{enumerate}	
\end{lemma}
\begin{remark}
	 From \eqref{e1}, we know for all $t\in[0,T]$ that
	 \begin{equation}\label{3.28}
	|d'(t)|\leq C\left( e^{-\gamma d_0}+|\phi_{xx}(0,t)|\right) \leq C\left( d_0^{-1}+N(T)\right).	
	 \end{equation}
\end{remark}

\begin{proof}
	(i) Let $h(t)\triangleq-\frac{st}{2}-d(t)$, $\forall t\in[0,T]$. Noting $h(0)=-d_0$, we only need to show for any $t\in[0,T]$ that $h'(t)=-\frac{s}{2}-d'(t)<0$. In fact, by the equation \eqref{d equ}, one has
	\begin{equation}\label{eq7}
		\begin{aligned}
			-d'(t)U(-st-d(t))\leq -\left[ f(u_-)-f\left(U(-st-d(t)) \right)\right]+\Big|\left(\frac{\phi_x}{U^{1-m}}\right)_x(0,t)+\frac{1}{m}G_x(0,t)\Big|.		
		\end{aligned}
	\end{equation}	
	According to the
	Sobolev embedding inequality $H^1(\mathbb{R_+})\hookrightarrow C^0(\mathbb{R_+})$ and \eqref{N(t)}, we know that $\left\|\phi_x(\cdot,t)\right\|_{L^{\infty}} \leq CN(T)$. Together with \eqref{q1} and \eqref{G}, it follows from Taylor's expansion that
	\begin{equation}\label{eq10}
		\begin{aligned}
		 &\Big|\left(\frac{\phi_x}{U^{1-m}}\right)_x(0,t)+\frac{1}{m}G_x(0,t)\Big|\\
&=\Big|\left(U+\phi_x\right)^{m-1}\left(U_x+\phi_{xx}\right)\big|_{x=0}-U^{m-1}U_x\big|_{x=0}\Big|\\
&=\bigg|\left(\left(U+\phi_x\right)^{m-1}- U^{m-1}\right)U_x\big|_{x=0}+\left(U+\phi_x\right)^{m-1}\phi_{xx}\big|_{x=0}\bigg|\\
&\leq C\left(\frac{|U_x|}{U^{2-m}} |\phi_x|+\frac{|\phi_{xx}| }{U^{1-m} }	 \right)\bigg|_{x=0}\\
&\leq	C\left(|\phi_{x}(0,t)|+|\phi_{xx}(0,t)| \right)\leq CN(T),
		\end{aligned}
	\end{equation}
	where we have used \eqref{N(t)} and  \eqref{sobolev embedding2} in the last inequality. This  with \eqref{eq7} together gives
	\begin{equation}\nonumber
		\begin{aligned}
			-d'(t)U(-st-d(t))&\leq -\left[ f(u_-)-f\left(U(-st-d(t)) \right)\right]+CN(T)\\&\leq-\left[ f(u_-)-f\left(U(-st-d(t)) \right)\right]+C\epsilon_{4},
		\end{aligned}
	\end{equation}
	where $\epsilon_{4}>0$ is a small constant. We observe that $-\left[ f(u_-)-f\left(U(-st-d(t)) \right)\right]\leq0$ due to the fact that $f$ is convex and $0<U(-st-d(t))\leq u_-$. Thus, if $-\left[ f(u_-)-f\left(U(-st-d(t)) \right)\right]+C\epsilon_{4}<0$, then we have $-d'(t)<0$; if $-\left[ f(u_-)-f\left(U(-st-d(t)) \right)\right]+C\epsilon_{4}\geq0$, which implies that $-d'(t)\leq C\epsilon_{4}$, then we have $\frac{-s}{2}-d'(t)<\frac{-s}{2}+C\epsilon_{4}<0$ if $\epsilon_{4}$ is small enough. Therefore, in both cases, $h'(t)=\frac{-s}{2}-d'(t)<0$ holds for $t\in[0,T]$.
	
	(ii) From \eqref{decay u_-}, \eqref{decompose} and the Taylor's expansion, the result of (i) means that
	\begin{equation}\nonumber
		f(u_-)-f\left(U(-st-d(t))\right)\sim u_-- U(-st-d(t))\sim \mathrm{e}^{-\lambda_-(st+d(t))},	
	\end{equation}
	and
	\begin{equation}\nonumber
		\phi_{x}(0,t)=u_--U(-st-d(t))\sim \mathrm{e}^{-\lambda_-(st+d(t))}.
	\end{equation}
	Hence, by \eqref{d equ} and \eqref{eq10}, one obtains
	\begin{equation}\label{eq11}\nonumber
		\begin{aligned}
			|d'(t)|U(-st-d(t))&\leq \Big| f(u_-)-f\left(U(-st-d(t)) \right)\Big|+\Big|\left(\frac{\phi_x}{U^{1-m}}\right)_x(0,t)+\frac{1}{m}G_x(0,t)\Big|\\&
			\leq \Big| f(u_-)-f\left(U(-st-d(t)) \right)\Big|+C\left(|\phi_{x}(0,t)|+|\phi_{xx}(0,t)| \right)\\&\leq C\left(\mathrm{e}^{-\lambda_-(st+d(t))}+|\phi_{xx}(0,t)| \right).	 
		\end{aligned}	
	\end{equation}
	Since $-st-d(t)\leq -d_0-\frac{st}{2}$ for $t\in[0,T]$, we get $U(-st-d(t))>U(-d_0)>U(0)>0$. Thus,
	\begin{equation}\nonumber
		|d'(t)|\leq	 C\left(\mathrm{e}^{-\lambda_-(st+d(t))}+|\phi_{xx}(0,t)| \right)\leq	 C\left(\mathrm{e}^{-\gamma(d_0+t)}+|\phi_{xx}(0,t)| \right)\text{ for some }\gamma>0.
	\end{equation}
	Next, we  show \eqref{e2}. In fact, from \eqref{e1}, we get
	\begin{equation}\nonumber
		\begin{aligned}
			&\int_{0}^{t}\left(d_0+\tau\right)^{\beta+3}	|d'(\tau)|^2\mathrm{d} \tau\\&\leq C\left( \int_{0}^{t}\left(d_0+\tau\right)^{\beta+3}\mathrm{e}^{-2\gamma(d_0+\tau)}+  \int_{0}^{t}\left(d_0+\tau\right)^{\beta+3}|\phi_{xx}(0,t)|^2 \right)\\&
			\leq C\left( e^{-2\gamma d_0}+\int_{0}^{t}\left(d_0+\tau\right)^{\beta+3}|\phi_{xx}(0,\tau)|^2\mathrm{d} \tau\right),	
		\end{aligned}
	\end{equation}
	where we have used $\int_{0}^{t}\left(d_0+\tau\right)^{\beta+3}\mathrm{e}^{-2\gamma\tau}\leq C$.
\end{proof}

\begin{lemma}\label{F G}
	Under the same assumptions of Proposition \ref{proposition priori estimate}, if $N(T)$ is sufficiently small, then there exists a constant $C>0$ independent of $T$ such that
	\begin{equation}\label{F1}
		|F|\leq C \phi_{x}^2,\quad |F_x|\leq C\left( |U_x|\phi_{x}^2+|\phi_{x}||\phi_{xx}|\right),
	\end{equation}	
	\begin{equation}\label{Fzz}
		|F_{xx}|\leq C\left[ \left(U_x^2+|U_{xx}|\right) \phi_{x}^2+\phi_{xx}^2+|U_{x}||\phi_{x}||\phi_{xx}|+|\phi_{x}||\phi_{xxx}|\right],
	\end{equation}				
	\begin{equation}\label{Fzzz}
		\begin{aligned}
			|F_{xxx}|&\leq C\left[ \left(|U_x|^3+|U_x||U_{xx}|+|U_{xxx}|\right) \phi_{x}^2+\left(U_x^2+|U_{xx}| \right)|\phi_{x}||\phi_{xx}|+|U_x||\phi_{x}||\phi_{xxx}|\right.\\&\left.\quad+|\phi_{x}||\phi_{xxxx}|+|U_x|\phi_{xx}^2+|\phi_{x x}|^3+|\phi_{xx}||\phi_{xxx}|\right],
		\end{aligned}
	\end{equation}
\begin{equation}\label{Gz}
	|G|\leq C \frac{\phi_{x}^2}{U^{2-m}},\quad|G_x|\leq C\left( \frac{|U_x|}{U^{3-m}}\phi_x^2+\frac{|\phi_{x}||\phi_{xx}| }{U^{2-m}}\right),
\end{equation}	
\begin{equation}\label{Gzz}
	|G_{xx}| \leq C\left[\left(\frac{U_x^2 }{U^{4-m}}+\frac{|U_{xx}| }{U^{3-m}} \right)\phi_{x}^2+ \frac{\phi_{xx}^2 }{U^{2-m}}+\frac{|U_x| }{U^{3-m}}|\phi_{x}||\phi_{xx}|+\frac{ |\phi_{x}||\phi_{xxx}|}{U^{2-m}}\right],
\end{equation}					
\begin{equation}\label{Gzzz}
	\begin{aligned}
		|G_{xxx}| &\leq C\left[\left(\frac{|U_x|^3 }{U^{5-m}}+\frac{|U_x||U_{xx}| }{U^{4-m}}+\frac{|U_{xxx}| }{U^{3-m}} \right)\phi_{x}^2+\left(\frac{U_x^2 }{U^{4-m}}+\frac{|U_{xx}| }{U^{3-m}} \right)|\phi_{x}||\phi_{xx}|\right.\\&\left.\quad+\frac{|U_x| }{U^{3-m}}|\phi_{x}||\phi_{xxx}|+\frac{|\phi_{x}||\phi_{xxxx}| }{U^{2-m}}+\frac{|U_x| }{U^{3-m}}\phi_{xx}^2+\frac{|\phi_{xx}|^3}{U^{3-m}}+\frac{|\phi_{xx}||\phi_{xxx}|}{U^{2-m}}\right].
	\end{aligned}
\end{equation}
\end{lemma}	
\begin{proof}
Recalling that $F=-\left(f\left(U+\phi_x\right)-f(U)-f^{\prime}(U) \phi_x\right)$, and noting \eqref{sobolev embedding4}, we then get the first inequality of \eqref{F1} by applying Taylor's expansion. Furthermore, a direct calculation gives
\begin{equation}\nonumber
		F_x=-\left(f^{\prime}\left(U+\phi_{x}\right)-f^{\prime}(U)-f^{\prime \prime}(U) \phi_{x}\right) U_x-\left(f^{\prime}\left(U+\phi_{x }\right)-f^{\prime}(U)\right) \phi_{xx},	 
\end{equation}
\begin{equation}\nonumber
\begin{aligned}
F_{xx}&=-\left(f^{\prime\prime}\left(U+\phi_{x}\right)-f^{\prime\prime}(U)-f^{\prime \prime\prime}(U) \phi_{x}\right) U_x^2-\left(f^{\prime}\left(U+\phi_{x}\right)-f^{\prime}(U)-f^{\prime \prime}(U) \phi_{x}\right) U_{xx}\\&\quad-f^{\prime\prime}\left(U+\phi_{x}\right)\phi_{xx}^2-2\left(f^{\prime\prime}\left(U+\phi_{x }\right)-f^{\prime\prime}(U)\right)U_x \phi_{xx}\\&\quad-\left(f^{\prime}\left(U+\phi_{x }\right)-f^{\prime}(U)\right) \phi_{xxx},	
\end{aligned}	
\end{equation}
and
\begin{equation}\nonumber
	\begin{aligned}
F_{xxx}&=-\left(f^{\prime\prime\prime}\left(U+\phi_{x}\right)-f^{\prime\prime\prime}(U)-f^{4}(U) \phi_{x}\right) U_x^3-3\left(f^{\prime\prime}\left(U+\phi_{x}\right)-f^{\prime\prime}(U)-f^{\prime \prime\prime}(U) \phi_{x}\right) U_xU_{xx}\\&\quad-\left(f^{\prime}\left(U+\phi_{x}\right)-f^{\prime}(U)-f^{\prime \prime}(U) \phi_{x}\right) U_{xxx}	 -3\left(f^{\prime\prime\prime}\left(U+\phi_{x}\right)-f^{\prime\prime\prime}(U)\right)U_x^2\phi_{x x}\\&\quad-3\left(f^{\prime\prime}\left(U+\phi_{x}\right)-f^{\prime\prime}(U)\right)U_{xx}\phi_{x x}-3\left(f^{\prime\prime}\left(U+\phi_{x}\right)-f^{\prime\prime}(U)\right)U_{x}\phi_{xxx}\\&\quad -\left(f^{\prime}\left(U+\phi_{x}\right)-f^{\prime}(U)\right)\phi_{xxxx}-3f^{\prime\prime\prime}\left(U+\phi_{x}\right)U_x\phi_{xx}^2\\&\quad-2f^{\prime\prime}\left(U+\phi_{x}\right)\phi_{xx}\phi_{xxx}-\!f^{\prime\prime\prime}\left(U+\phi_{x}\right)\phi_{xx}^3,
	\end{aligned}
\end{equation}
where, by virtue of $\|\phi_{x}(\cdot,t)\|_{L^{\infty}}\leq CN(t)$, the second inequality of \eqref{F1}, \eqref{Fzz} and \eqref{Fzzz} immediately follow from Taylor's expansion. The proof of \eqref{Gz}-\eqref{Gzzz} is similar to that \eqref{F1}-\eqref{Fzzz}. We omit it.
\end{proof}

\begin{lemma}
	 For $0\leq t\leq T$, the following inequalities hold:
\begin{equation}\label{boundary phixxx}
	\frac{\phi_{xx}\phi_{xxx}(0,t)}{U^{1-m}(-st-d(t))}\geq \left( \frac{3}{4}-CN(T)\right)f'(u_-)\phi_{xx}^2(0,t)-C\left( U_x^2+U_{xx}^2\right)\big|_{x=0},
\end{equation}	
\begin{eqnarray}\label{boundary phixxxx}
		\frac{\phi_{xxx}\phi_{xxxx}(0,t)}{U^{1-m}(-st-d(t))}&\geq&\left(\frac{3}{4}\!-\!CN(T)\right)
		\frac{f'(u_-) }{2}\frac{\mathrm{d}}{\mathrm{d}t}\phi_{xx}^2(0,t)\!+\!\left( \frac{3}{4}-CN(T)\right)f'(u_-)\phi_{xxx}^2(0,t) \nonumber \\ &&+\frac{1}{U^{1-m}}\frac{\mathrm{d}}{\mathrm{d}t}\left[\phi_{xx}\left( u_-^{1-m}f'(u_-)U_{x} \!-\!U_{xx}\!-\!\frac{ (m\!-\!1)}{u_-}\left(U_{x}\!+\!\phi_{xx}\right)^2\!\right)\!\right]\!\bigg|_{x=0}	\nonumber \\
&&-CN(T)\phi_{x x}^2(0,t)-C\left(U_x^2+U_{xx}^2+U_{xxx}^2 \right)\big|_{x=0}.
\end{eqnarray} 	
\end{lemma}

\begin{proof}
	 To estimate the value of $\phi_{xxx}(0,t)$, we integrate $\partial_x\eqref{original model}_1$ over $(0,+\infty)$ to have
	 \begin{equation}\nonumber
	 	\frac{\mathrm{d}}{\mathrm{d}t }u(0,t)+f'(u_-)u_x(0,t)=\left(m-1\right)u_-^{m-2}u_x^2(0,t)+u_-^{m-1}u_{xx}(0,t),
	 \end{equation}
	 which, along with \eqref{decompose}, gives
	 \begin{equation}\label{phixxx 0}
	 	\phi_{xxx}(0,t)=\left\{ -U_{xx}+u_-^{1-m}f'(u_-)\left(U_{x}+\phi_{xx}\right)-\frac{ (m-1)}{u_-}\left(U_{x}+\phi_{xx}\right)^2\right\}\bigg|_{x=0}.	
	 \end{equation}
	 Therefore, we have
	 \begin{equation}\label{eq39}
	 	\begin{aligned}
	 		&\frac{\phi_{xx}\phi_{xxx}(0,t)}{U^{1-m}(-st-d(t))}\\
&=\left\{\frac{\phi_{xx}}{U^{1-m}}\left[  -U_{xx}+u_-^{1-m}f'(u_-)\left(U_{x}\!+\!\phi_{xx}\right)\!-\!\frac{ (m-1)}{u_-}\left(U_{x}\!+\!\phi_{xx}\right)^2\right]\right\}\bigg|_{x=0}\\&\geq f'(u_-)\phi_{xx}^2(0,t)\!+\!\frac{\phi_{xx}}{U^{1-m}}\!\left[ u_-^{1-m}f'(u_-)U_{x} \!-\!U_{xx}\!-\!\frac{ (m-1)}{u_-}\left(U_{x}\!+\!\phi_{xx}\right)^2\right]\!\bigg|_{x=0},
	 	\end{aligned}
	 \end{equation}
	 where, by $|\phi_{xx}(0,t)|\leq CN(T)$ and the Cauchy-Schwarz inequality, it holds that
	 \begin{equation}\label{eq40}
	 	\begin{aligned}
	 		&\bigg|\frac{\phi_{xx}}{U^{1-m}}\!\left[ u_-^{1-m}f'(u_-)U_{x} \!-\!U_{xx}\!-\!\frac{ (m-1)}{u_-}\left(U_{x}\!+\!\phi_{xx}\right)^2\right]\bigg|_{x=0}\bigg|\\&\leq \left(\frac{1}{4}+CN(T)\right) f'(u_-)\phi_{xx}^2(0,t)+C\left( U_x^2+U_{xx}^2\right)\big|_{x=0}.
	 	\end{aligned}	
	 \end{equation}
	 Therefore, combining \eqref{eq39} and \eqref{eq40}, one obtains
	 \begin{equation}\nonumber
	 	\frac{\phi_{xx}\phi_{xxx}(0,t)}{U^{1-m}(-st-d(t))}\geq\left(\frac{3}{4}-CN(T)\right)f'(u_-)\phi_{xx}^2(0,t)-C\left( U_x^2+U_{xx}^2\right)\big|_{x=0},
	 \end{equation}
which implies \eqref{boundary phixxx}.

It remains to show \eqref{boundary phixxxx}. To estimate the boundary value of $\phi_{xxxx}(0,t)$, we integrate $\partial_x^2\eqref{original model}_1$ over $(0,+\infty)$ to have
\begin{equation}\nonumber
	\begin{aligned}
		&\frac{\mathrm{d}}{\mathrm{d}t }u_x(0,t)+f''(u_-)u_x^2(0,t)+f'(u_-)u_{xx}(0,t)\\&=\left(m-1\right)(m-2)u_-^{m-3}u_x^3(0,t)+3(m-1)u_-^{m-2}u_x(0,t)u_{xx}(0,t)+u_-^{m-1}u_{xxx}(0,t),
	\end{aligned}
\end{equation}
which, along with \eqref{decompose}, gives
\begin{equation}\nonumber
	\begin{aligned}
		\phi_{xxxx}(0,t)&= \frac{\mathrm{d}}{\mathrm{d}t}\phi_{xx}(0,t)+u_-^{1-m}f'(u_-)\phi_{xxx}(0,t)+\frac{ 3(m-1)}{u_-}U_{x}\phi_{xxx}(0,t)-U_{xxx}(-st-d(t))\\&\quad-\frac{3 (m-1)}{u_-}\phi_{xx}(0,t)\phi_{xxx}(0,t)+u_-^{1-m}\left\{U_{xx}\left[-s-d'+f'(u_-)\!-\!3(m-1)u_-^{m-2}(\phi_{xx}\right.\right.\\&\left.\left.\quad+U_x)\right]+f''(u_-)(\phi_{xx}+U_x)^2-(m-1)(m-2)u_-^{m-3}(\phi_{xx}+U_x)^3\right\}\big|_{x=0}.
	\end{aligned}	
\end{equation}
Thus, we have
\begin{align}	 &\frac{\phi_{xxx}\phi_{xxxx}(0,t)}{U^{1-m}(-st-d(t))}\nonumber\\
&=\frac{\phi_{xxx}}{U^{1-m}}\frac{\mathrm{d}\phi_{xx}}{\mathrm{d}t}\bigg|_{x=0}+\frac{u_-^{1-m}}{U^{1-m}}f'(u_-)\phi_{xxx}^2(0,t)+\frac{ 3(m-1)}{u_-}\frac{U_{x}}{U^{1-m}}\phi_{xxx}^2(0,t)\nonumber\\&\quad-U_{xxx}\frac{\phi_{xxx}(0,t)}{U^{1-m}}-\frac{3 (m-1)}{u_-}\frac{\phi_{xx}(0,t)}{U^{1-m}}\phi_{xxx}^2(0,t)+\frac{u_-^{1-m}}{U^{1-m}}\phi_{xxx}\left\{U_{xx}\right.\nonumber\\&\left.\quad \cdot\left[-s-d'+f'(u_-)-3(m-1)u_-^{m-2}(\phi_{xx}+U_x)\right]+ f''(u_-)(\phi_{xx}+U_x)^2\right.\nonumber\\
&\left.\quad-(m-1)(m-2)u_-^{m-3}(\phi_{xx}+U_x)^3\right\}\big|_{x=0}\nonumber\\&\equiv B_1+\cdots+B_6.	 \label{42}	
\end{align}
By \eqref{phixxx 0}, it holds that
\begin{align}
	B_1&=	\frac{1}{U^{1-m}}\frac{\mathrm{d}\phi_{xx}}{\mathrm{d}t}\left[ -U_{xx}+u_-^{1-m}f'(u_-)\left(U_{x}+\phi_{xx}\right)-\frac{ (m-1)}{u_-}\left(U_{x}+\phi_{xx}\right)^2\right]\bigg|_{x=0}\nonumber\\&%=\frac{u_-^{1-m}}{U^{1-m}}\frac{f'(u_-) }{2}\frac{\mathrm{d}}{\mathrm{d}t}\phi_{xx}^2(0,t)+\frac{1}{U^{1-m}}\frac{\mathrm{d}\phi_{xx}}{\mathrm{d}t}\left[u_-^{1-m}f'(u_-)U_{x} -U_{xx}\right.\nonumber\\&\left.\quad-\frac{ (m-1)}{u_-}\left(U_{x}+\phi_{xx}\right)^2\right]\bigg|_{x=0}\nonumber\\&
	=\frac{u_-^{1-m}}{U^{1-m}}
	\frac{f'(u_-) }{2}\frac{\mathrm{d}}{\mathrm{d}t}\phi_{xx}^2(0,t)+\frac{1}{U^{1-m}}\frac{\mathrm{d}}{\mathrm{d}t}\left[\phi_{xx}\left( u_-^{1-m}f'(u_-)U_{x} -U_{xx}\right.\right.\nonumber\\&\left.\left.\quad-\frac{ (m-1)}{u_-}\left(U_{x}+\phi_{xx}\right)^2\right)\right]\bigg|_{x=0}-\frac{\phi_{xx}}{U^{1-m}}\frac{\mathrm{d}}{\mathrm{d}t}\left[\left( u_-^{1-m}f'(u_-)U_{x} -U_{xx}\right.\right.\nonumber\\&\left.\left.\quad-\frac{ (m-1)}{u_-}\left(U_{x}+\phi_{xx}\right)^2\right)\right]\bigg|_{x=0}	,\label{43}
\end{align}
where, noting that $|\phi_{xx}(0,t)|\leq CN(T)$, one has
\begin{equation}\nonumber
	\begin{aligned}
		&\text{the third trem on RHS of \eqref{43}}\\&= -\frac{\phi_{xx}}{U^{1-m}}(-s-d')\left(u_-^{1-m}f'(u_-)U_{xx} -U_{xxx}-\frac{ 2(m-1)}{u_-}U_{xx}\left(U_{x}+\phi_{xx}\right)\right)\bigg|_{x=0}\\&\quad+\frac{ (m-1)}{u_-}\frac{U_{x} }{U^{1-m}}\frac{\mathrm{d}}{\mathrm{d}t}\phi_{xx}^2(0,t)+\frac{ (m-1)}{u_-}\frac{\phi_{xx}(0,t) }{U^{1-m}}\frac{\mathrm{d}}{\mathrm{d}t}\phi_{xx}^2(0,t)\\&\leq CN(T)\phi_{xx}^2(0,t)+\left(\frac{f'(u_-) }{8}+CN(T) \right)\frac{u_-^{1-m}}{U^{1-m}}\frac{\mathrm{d}}{\mathrm{d}t}\phi_{xx}^2(0,t)+C\left(U_x^2+U_{xx}^2+U_{xxx}^2 \right)\big|_{x=0}.	
	\end{aligned}
\end{equation}
Moreover, we utilize $|\phi_{xx}(0,t)|\leq CN(T)$ and the Cauchy-Schwarz inequality to get
\begin{equation}\nonumber
	\begin{aligned}
		B_3&=\frac{ 3(m-1)}{u_-}\frac{U_{x}}{U^{1-m}}\left[ -U_{xx}+u_-^{1-m}f'(u_-)\left(U_{x}+\phi_{xx}\right)-\frac{ (m-1)}{u_-}\left(U_{x}+\phi_{xx}\right)^2\right]^2\bigg|_{x=0}\\&
		%\leq C\left(U_x^2+U_x^4+U_x^8+U_{xx}^4+\phi_{x x}^4+\phi_{x x}^8 \right)\big|_{x=0}	\\&
		\leq CN(T) \phi_{x x}^2(0,t)+C\left(U_x^2+U_{xx}^2 \right)\big|_{x=0},
	\end{aligned}
\end{equation}
and
\begin{align} \label{44}
	B_4+B_5+B_6 \leq & \left(\frac{1}{4}+CN(T) \right)\frac{u_-^{1-m}}{U^{1-m}}f'(u_-)\phi_{xxx}^2(0,t) \notag \\
&+CN(T) \phi_{x x}^2(0,t)+C\left(U_x^2\!+\!U_{xx}^2\!+\!U_{xxx}^2 \right)\big|_{x=0}.
\end{align}
Therefore, combining \eqref{43}-\eqref{44} and noting $U^{1-m}\leq u_-^{1-m}$, we show \eqref{42} by
\begin{equation}\nonumber
	\begin{aligned}
		\frac{\phi_{xxx}\phi_{xxxx}(0,t)}{U^{1-m}(-st-d(t))}&\geq\left(\frac{3}{4}\!-\!CN(T)\right)
		\frac{f'(u_-) }{2}\frac{\mathrm{d}}{\mathrm{d}t}\phi_{xx}^2(0,t)\!+\!\left( \frac{3}{4}-CN(T)\right)f'(u_-)\phi_{xxx}^2(0,t)\\&\quad+\!\frac{1}{U^{1-m}}\frac{\mathrm{d}}{\mathrm{d}t}\left[\phi_{xx}\left( u_-^{1-m}f'(u_-)U_{x} \!-\!U_{xx}\!-\!\frac{ (m\!-\!1)}{u_-}\left(U_{x}\!+\!\phi_{xx}\right)^2\!\right)\!\right]\!\bigg|_{x=0}	\\&\quad- CN(T)\phi_{x x}^2(0,t)-C\left(U_x^2+U_{xx}^2+U_{xxx}^2 \right)\big|_{x=0}	.
	\end{aligned}
\end{equation}
We complete the proof of \eqref{boundary phixxxx}.
\end{proof}

\subsubsection{Basic estimates}

We now give the following basic $L^2$ estimate.
\begin{lemma}\label{L2}
	Let the assumptions of Proposition \ref{proposition priori estimate} hold. If $N(T)+d_0^{-\frac{1}{2}}\ll1$, then there exists a constant $C>0$ independent of $T$ such that
		\begin{align} \label{L2 estimate}
		&\|\phi(\cdot,t)\|^2+\int_0^t\left\|\sqrt{|U_x|}\phi(\cdot,\tau)\right\|^2+\left\|\phi_x(\cdot,\tau)\right\|_{w_7}^2 \mathrm{d} \tau  \notag \\
&\leq C\left(\|\phi_0\|^2+d_0^{-1} \right)+C\int_{0}^{t}(d_0+\tau)^{3}|\phi_{xx}(0,\tau)|^2\mathrm{d} \tau,
	\end{align}	
where $w_7=U^{m-1}(x-st-d(t))$.
\end{lemma}

\begin{proof}
Multiplying \eqref{phiequ} by $\phi(x,t)$ yields
	\begin{align} \label{eq13}
		&\left(\frac{1}{2}\phi^2\right)_t+\left(\frac{1}{2} g^\prime(U)\phi^2+\frac{s}{2}\phi^2- \frac{\phi\phi_x}{U^{1-m}}-\frac{1}{m}G\phi\right)_x 
               -\frac{1}{2} g^{\prime\prime}(U)U_x\phi^2+ \frac{\phi_x^2}{U^{1-m}}\notag \\
		&=d'(t)U\phi+F\phi-\frac{1}{m} G\phi_x.
	\end{align}	
Noting that $\phi(0,t)=0$, integrating \eqref{eq13} over $(0,+\infty)\times(0,t)$, we get
	\begin{align} \label{eq14}
		&\int \phi^2+\int_{0}^{t}\int|U_x|\phi^2+\int_{0}^{t}\int\frac{\phi_x^2}{U^{1-m}} \notag \\
		&  \leq C\left(\int \phi_0^2+\int_{0}^{t}|d'(\tau)|\int U|\phi|+\int_{0}^{t}\int\left| F\phi\right|+\int_{0}^{t}\int\left| G \phi_x\right|\right).
	\end{align}
We decompose the second term on the right hand side of \eqref{eq14} as
	\begin{align} \label{eq19}
	 &\int_{0}^{t}|d'(\tau)|\int U(x-s\tau-d(\tau))|\phi(x,\tau)| \notag \\
&=\int_{0}^{t}|d'(\tau)|\left( \int_{0}^{s\tau+d(\tau)}+\int_{s\tau+d(\tau)}^{+\infty}\right)U(x-s\tau-d(\tau))|\phi(x,\tau)|.
 \end{align}
In the region $x>st+d(t)$, by \eqref{ODE}, one obtains
\begin{equation}\nonumber
	U_x= U^{2-m}\left(\frac{f(U)}{U}-s\right)\sim U^{2-m},
\end{equation}
which implies that
\begin{equation}\label{e3}
	|U|\leq C|U_x|^{\frac{1}{2-m}} \text{ for }x>st+d(t).
\end{equation}
Then, utilizing H\"{o}lder's inequality and the Cauchy-Schwarz inequality, we have
	 \begin{align} \label{eq15}
&\int_{0}^{t}|d'(\tau)|\int_{s\tau+d(\tau)}^{+\infty}U(x-s\tau-d(\tau))|\phi(x,\tau)| \notag \\
&\leq C\int_{0}^{t}|d'(\tau)|\int_{s\tau+d(\tau)}^{+\infty}|U_x|^{\frac{1}{2-m}}	 |\phi| \notag \\
&\leq C\int_{0}^{t}|d'(\tau)|\left(\int_{s\tau+d(\tau)}^{+\infty}|U_x|\phi^2 \right)^{\frac{1}{2}}\left(\int_{s\tau+d(\tau)}^{+\infty}|U_x|^{\frac{m}{2-m}}\right)^{\frac{1}{2}}	 \notag \\
&\leq\frac{1}{2}\int_{0}^{t}\int|U_x|\phi^2+C\int_{0}^{t}|d'(\tau)|^2\int_{s\tau+d(\tau)}^{+\infty}|U_x|^{\frac{m}{2-m}},
	 \end{align}
where, in view of \eqref{f''>0}, \eqref{q1} and the fact that $\frac{-m}{1-m}<-1$ due to $m>1/2$, it holds that
\begin{equation}\nonumber
\int_{s\tau+d(\tau)}^{+\infty}|U_x|^{\frac{m}{2-m}}\leq C \int_{s\tau+d(\tau)}^{+\infty}U^{m}\sim C \int_{s\tau+d(\tau)}^{+\infty}|x-s\tau-d(\tau)|^{\frac{-m}{1-m}}\leq C.
\end{equation}
This, along with \eqref{eq15} and \eqref{e1}, yields
	\begin{align} \label{eq17}
&\int_{0}^{t}|d'(\tau)|\int_{s\tau+d(\tau)}^{+\infty}U(x-s\tau-d(\tau))|\phi(x,\tau)| \notag \\
&\leq\frac{1}{2}\int_{0}^{t}\int|U_x|\phi^2+C\int_{0}^{t}|d'(\tau)|^2 \notag \\
&\leq\frac{1}{2}\int_{0}^{t}\int|U_x|\phi^2+C\left(\mathrm{e}^{-2\gamma d_0}+\int_{0}^{t}|\phi_{xx}(0,\tau)|^2\right).	
\end{align}
In the region of $x<st+d(t)$, we apply Fubini's theorem and the H\"{o}lder's inequality to obtain
	\begin{align} \label{eq16}
\int_{0}^{s\tau+d(\tau)}|\phi(x,\tau)|\mathrm{d}x	
&\leq\int_{0}^{s\tau+d(\tau)}\left(\int_{0}^{x}|\phi_y(y,\tau)|\mathrm{d}y \right) \mathrm{d}x \notag \\
&=\int_{0}^{s\tau+d(\tau)}\left(\int_{y}^{s\tau+d(\tau)}|\phi_y(y,\tau)|\mathrm{d}x \right) \mathrm{d}y \notag \\
&=\int_{0}^{s\tau+d(\tau)}|\phi_y(y,\tau)|\left(s\tau+d(\tau)-y\right) \mathrm{d}y
\notag \\
&\leq|s\tau+d(\tau)|^{\frac{3}{2}}\left(\int_{0}^{s\tau+d(\tau)}|\phi_y(y,\tau)|^2\mathrm{d}y\right)^{\frac{1}{2}}.
	\end{align}
Note that
\begin{equation}\label{3.58}\nonumber
	\int_{0}^{t}|\phi_{xx}(0,\tau)|\leq\left( \int_{0}^{t}\left(d_0+\tau\right)^{\beta+3}|\phi_{xx}(0,\tau)|^2\right)^{\frac{1}{2}}\left( \int_{0}^{t}\left(d_0+\tau\right)^{-(\beta+3)}\right)^{\frac{1}{2}}\leq CN(T),	
\end{equation}
then, by \eqref{e2} and $d_0^{-\frac{1}{2}}+N(T)\ll1$, we have
\begin{equation}\nonumber
|d(t)-d_0|\leq \int_{0}^{t}|d'(\tau)|\leq C\int_{0}^{t}\left( e^{-\gamma\left(d_0+\tau\right)}+|\phi_{xx}(0,\tau)|\right)\leq C\left( e^{-\gamma d_0}+N(T)\right)\ll1,	
\end{equation}
and
\begin{equation}\label{st+d(t)}
	|st+d(t)|\sim (d_0+t) \text{ for any } t\in[0,T].
\end{equation}
Combining \eqref{eq16} with \eqref{st+d(t)}, and utilizing \eqref{e2}, we get
	\begin{align} \label{eq18}
	&\int_{0}^{t}|d'(\tau)|\int_{0}^{s\tau+d(\tau)}U(x-s\tau-d(\tau))|\phi(x,\tau)| \notag \\
	&\leq C\int_{0}^{t}(d_0+\tau)^{\frac{3}{2}}|d'(\tau)|\left(\int_{0}^{s\tau+d(\tau)}\phi_x^2\right)^{\frac{1}{2}} \notag \\
&\leq\frac{1}{2u_-^{1-m}}\int_{0}^{t}\int\phi_x^2+C\int_{0}^{t}(d_0+\tau)^{3}|d'(\tau)|^2 \notag \\
&\leq\frac{1}{2}\int_{0}^{t}\int\frac{\phi_x^2 }{U^{1-m}}+C\left(\mathrm{e}^{-2\gamma d_0}+\int_{0}^{t}(d_0+\tau)^{3}|\phi_{xx}(0,\tau)|^2 \right),
\end{align}
where we have used the fact that $\int_{0}^{t}(d_0+\tau)^{3}\mathrm{e}^{-2\gamma\left(d_0+\tau\right)}\leq C\mathrm{e}^{-2\gamma d_0}$.
Therefore, by virtue of \eqref{eq17} and \eqref{eq18}, \eqref{eq19} can be further estimated as
	\begin{align} \label{eq20}
	&\int_{0}^{t}|d'(\tau)|\int U(x-s\tau-d(\tau))|\phi(x,\tau)| \notag \\
&\leq\frac{1}{2}\int_{0}^{t}\int\left( |U_x|\phi^2+\frac{\phi_x^2 }{U^{1-m}}\right)+C\left(\mathrm{e}^{-2\gamma d_0}+\int_{0}^{t}(d_0+\tau)^{3}|\phi_{xx}(0,\tau)|^2 \right).
		\end{align}
Furthermore, by \eqref{F1}, \eqref{Gz} and the fact that $\left\|\phi\left(\cdot,t\right)\right\|_{L^{\infty}}\leq CN(T)$ and $\left\|\phi_x(\cdot,t)/U\right\|_{L^{\infty}}\leq CN(T)$ from \eqref{sobolev embedding4}, one can see that
\begin{equation}
	\int_{0}^{t}\int\left| F\phi\right|\leq CN(T)\int_{0}^{t}\int \phi_x^2,
\end{equation}	
and
\begin{equation}\label{eq21}
	\int_0^t\int |G\phi_x|\leq C\int_0^t\int \frac{|\phi_x|^3}{U^{2-m}}\leq C N(T)\int_0^t\int \frac{\phi_x^2}{U^{1-m}}.
\end{equation}
Substituting \eqref{eq20}-\eqref{eq21} into \eqref{eq14}, it holds
\begin{equation}\nonumber
	\begin{aligned}
		&\int \phi^2+\left(\frac{1}{2}-CN(T)\right)\int_{0}^{t}\int|U_x|\phi^2+\left(\frac{1}{2}-CN(T)\right)\int_{0}^{t}\int\frac{\phi_x^2}{U^{1-m}}\\
		&  \leq C\int \phi_0^2+C\left(\mathrm{e}^{-2\gamma d_0}+\int_{0}^{t}(d_0+\tau)^{3}|\phi_{xx}(0,\tau)|^2 \right).
	\end{aligned}
\end{equation}
Taking $N(T)$ small enough such that $CN(T)\leq1/4$, we then get \eqref{L2 estimate}.
\end{proof}

\begin{lemma}\label{L2-1}
	Let the assumptions of Proposition \ref{proposition priori estimate} hold. If $N(T)+d_0^{-\frac{1}{2}}\ll1$, then there exists a constant $C>0$ independent of $T$ such that
	\begin{equation}\label{L2-1 estimate}
			\|\phi(\cdot,t)\|_{w_1}^2+\int_0^t\left\|\phi_x(\cdot,\tau)\right\|_{w_4}^2 \mathrm{d} \tau   \leq C\left(\|\phi_0\|_{w_{1,0}}^2+d_0^{-1} \right)+C\int_{0}^{t}(d_0+\tau)^{3}|\phi_{xx}(0,\tau)|^2\mathrm{d} \tau,
	\end{equation}	
where $w_1=U^{1-3m}$, $w_4=U^{-2m}$ with $U=U(x-st-d(t))$ and $w_{1,0}=U^{1-3m}(x-d_0)$.
\end{lemma}

\begin{proof}
Multiplying \eqref{phiequ} by $\frac{\phi}{U^{3m-1}}$ and noting that
\begin{equation}\nonumber
\frac{\phi\phi_t}{U^{3m-1}}=\frac{1}{2}\left(\frac{\phi^2}{U^{3m-1}} \right)_t+\frac{3m-1}{2}\frac{U_x}{U^{3m}}(-s-d'(t))\phi^2,
\end{equation}
%\begin{equation}\nonumber
%f'(U)\frac{\phi\phi_x}{U^{3m-1}}=\frac{1}{2}\left(\frac{f'(U)\phi^2}{U^{3m-1}} \right)_x-\frac{1}{2}\frac{f''(U)U_x }{U^{3m-1}}\phi^2+\frac{3m-1}{2}\frac{f'(U)U_x}{U^{3m}}\phi^2,	
%\end{equation}
%\begin{equation}\nonumber
%-\left(\frac{\phi_x}{U^{1-m}}\right)_x\frac{\phi}{U^{3m-1}}=-\left(\frac{\phi\phi_x}{U^{2m}}\right)_x+\frac{\phi_x^2}{U^{2m}}-(3m-1)\frac{U_x}{U^{2m+1}}\phi\phi_x,
%\end{equation}
and
\begin{equation}\nonumber
	\bigg|(3m-1)\frac{U_x}{U^{2m+1}}\phi\phi_x\bigg|\leq \eta \frac{\phi_x^2}{U^{2m}}+\frac{(3m-1)^2}{4\eta}\frac{U_x^2}{U^{2m+2}}\phi^2,	
\end{equation}
we obtain
%\begin{equation}\label{eq22}
%	\begin{aligned}
%&\frac{1}{2}\int\frac{\phi^2}{U^{3m-1}} -\frac{3m-1}{2}\int_{0}^{t}d'(\tau)\int\frac{U_x}{U^{3m}}\phi^2+\int_{0}^{t}\int\frac{\phi_x^2}{U^{2m}}\\&-\frac{1}{2}\int_{0}^{t}\int\left[f''(U)-(3m-1)\frac{g'(U)}{U} \right] \frac{U_x }{U^{3m-1}}\phi^2-(3m-1)\int_{0}^{t}\int\frac{U_x}{U^{2m+1}}\phi\phi_x	=\frac{1}{2}\int\frac{\phi_0^2}{U^{3m-1}} \\&+\int_{0}^{t}d'(\tau)\int \frac{\phi}{U^{3m-2}}+\int_{0}^{t}\int F\frac{\phi}{U^{3m-1}}+\frac{3m-1}{m}\int_{0}^{t}\int\frac{U_x}{U^{3m}} G \phi-\frac{1}{m}\int_{0}^{t}\int G\frac{\phi_x}{U^{3m-1}}.
%	\end{aligned}
%\end{equation}
	\begin{align} \label{eq23}
		&\frac{1}{2}\int\frac{\phi^2}{U^{3m-1}} +\left(1-\eta\right)\int_{0}^{t}\int\frac{\phi_x^2}{U^{2m}}-\frac{1}{2}\int_{0}^{t}\int\mathbb{A_\eta} \frac{U_x }{U^{3m-1}}\phi^2  \notag\\
& \ \ \ -\frac{3m-1}{2}\int_{0}^{t}d'(\tau)\int\frac{U_x}{U^{3m}}\phi^2	\notag \\
&\leq\frac{1}{2}\int\frac{\phi_0^2}{U^{3m-1}} +\int_{0}^{t}\!d'(\tau)\int \frac{\phi}{U^{3m-2}}+\int_{0}^{t}\int \frac{F\phi}{U^{3m-1}} \notag \\
& \ \ \ +\frac{3m-1}{m}\int_{0}^{t}\int\frac{U_x}{U^{3m}} G \phi-\frac{1}{m}\int_{0}^{t}\int \frac{G\phi_x}{U^{3m-1}},
	\end{align}
where $\eta\in(0,1)$ is a constant to be determined later and
\begin{equation}\nonumber
	\mathbb{A_\eta}\triangleq f''(U)-(3m-1)\frac{g'(U)}{U} +\frac{(3m-1)^2}{2\eta}\frac{U_x}{U^{3-m}}.
\end{equation}
In the region of $x-st-d(t)<0$, since $U(x-st-d(t))>U(0)>0$, we get from \eqref{3.28} that
\begin{equation}\label{eq24}\nonumber
\bigg|	 \frac{\mathbb{A_\eta}}{2}\frac{U_x }{U^{3m-1}}\phi^2\bigg|+\frac{3m-1}{2}|d'(t)|\frac{|U_x|}{U^{3m}}\phi^2\leq C|U_x|\phi^2.
\end{equation}
In the region of $x-st-d(t)>0$, we claim that $\mathbb{A_\eta}>0$. In fact, we are primarily concerned with the case as $z\rightarrow+\infty$ due to singularity. Since \eqref{ODE} and \eqref{degenerate lax's}, when $x\rightarrow+\infty$,
	\begin{align} \label{eq25}
	\mathbb{A_\eta}&\sim f''(0)-(3m-1)\frac{f'(0)-s}{U} +\frac{(3m-1)^2}{2\eta}\frac{f'(0)-s}{U} \notag \\
&\sim f''(0)-(3m-1)\frac{f'(0)-s}{U}\left[1-\frac{(3m-1)}{2\eta}\right].	
\end{align}
We thus take $\eta=m$, then $1-\frac{(3m-1)}{2\eta}=1-\frac{(3m-1)}{2m}=\frac{1}{2}\left(\frac{1}{m}-1\right)>0$. It then follows from $m>1/3$, $f''>0$, $f'(0)-s<0$, $|d'(t)|\ll1$ and \eqref{3.28} that
\begin{equation}\nonumber
	 \begin{aligned}
\frac{\mathbb{A_\eta}}{2}\frac{|U_x| }{U^{3m-1}}\phi^2+\frac{3m-1}{2}d'(t)\frac{|U_x|}{U^{3m}}\phi^2\geq 	\frac{C_0}{2}\frac{|U_x| }{U^{3m}}\phi^2 ,	
	 \end{aligned}
\end{equation}
where $C_0$ is a positive constant.
%Combining \eqref{eq24} and \eqref{eq25}, we deduce from \eqref{eq23} that
%\begin{equation}\label{eq26}
%	\begin{aligned}
%		&\int\frac{\phi^2}{U^{3m-1}}+\int_{0}^{t}\int\frac{|U_x|}{U^{3m}}\phi^2+\int_{0}^{t}\int\frac{\phi_x^2}{U^{2m}}	\leq\int\frac{\phi_0^2}{U^{3m-1}} \\&+\int_{0}^{t}|d'(\tau)|\int \frac{|\phi|}{U^{3m-2}}+\int_{0}^{t}\int \big|F\frac{\phi}{U^{3m-1}}\big|+\int_{0}^{t}\int\big|\frac{U_x}{U^{3m}} G \phi\big|+\int_{0}^{t}\int \big|G\frac{\phi_x}{U^{3m-1}}\big|,
%	\end{aligned}
%\end{equation}
Next we estimate the terms on the right hand side (RHS) of \eqref{eq23}. The second term on the RHS of \eqref{eq23} is
\begin{equation}\nonumber
	\begin{aligned}
		& \int_{0}^{t}|d'(\tau)|\int \frac{|\phi(x,\tau)|}{U^{3m-2}(x-s\tau-d(\tau))}\\&=\int_{0}^{t}|d'(\tau)|\left( \int_{0}^{s\tau+d(\tau)}+\int_{s\tau+d(\tau)}^{+\infty}\right)\frac{|\phi(x,\tau)|}{U^{3m-2}(x-s\tau-d(\tau))}.
	\end{aligned}
\end{equation}
In the region of $x-st-d(t)>0$, utilizing \eqref{e3}, H\"{o}lder's inequality and the Cauchy-Schwarz inequality, we get
	\begin{align} \label{eq28}
		&\int_{0}^{t}|d'(\tau)|\int_{s\tau+d(\tau)}^{+\infty}\frac{|\phi(x,\tau)|}{U^{3m-2}(x-s\tau-d(\tau))} \notag \\
&\leq C\int_{0}^{t}|d'(\tau)|\int_{s\tau+d(\tau)}^{+\infty}\frac{ |U_x|^{\frac{2}{2-m}}}{U^{3m}}|\phi| \notag \\
&\leq C\int_{0}^{t}|d'(\tau)|\left(\int_{s\tau+d(\tau)}^{+\infty}\frac{|U_x|}{U^{3m}}\phi^2 \right)^{\frac{1}{2}}\left(\int_{s\tau+d(\tau)}^{+\infty}\frac{|U_x|^{\frac{2+m}{2-m}} }{U^{3m} }\right)^{\frac{1}{2}}	\notag \\
&\leq\frac{C_0}{4}\int_{0}^{t}\int\frac{|U_x|}{U^{3m}}\phi^2+C\int_{0}^{t}|d'(\tau)|^2,
	\end{align}
where we have used
\begin{equation}\nonumber
	\int_{s\tau+d(\tau)}^{+\infty}\frac{|U_x|^{\frac{2+m}{2-m}} }{U^{3m} }\leq C \int_{s\tau+d(\tau)}^{+\infty}U^{2-2m}\sim C \int_{s\tau+d(\tau)}^{+\infty}|x-s\tau-d(\tau)|^{-2}\leq C,
\end{equation}
due to $|U_x|\leq CU^{2-m}$ and \eqref{f''>0}. In the region of $x-st-d(t)<0$, since $0<U(0)<U(x-st-d(t))<u_-$, by virtue of \eqref{eq18}, one obtains
	\begin{align}
		&\int_{0}^{t}|d'(\tau)|\int_{0}^{s\tau+d(\tau)}\frac{|\phi(x,\tau)|}{U^{3m-2}(x-s\tau-d(\tau))}\nonumber\\&\leq C\int_{0}^{t}(d_0+\tau)^{\frac{3}{2}}|d'(\tau)|\left(\int_{0}^{s\tau+d(\tau)}\phi_x^2\right)^{\frac{1}{2}}\nonumber\\&\leq\frac{1-m}{4u_-^{2m}}\int_{0}^{t}\int\phi_x^2+C\left(\mathrm{e}^{-2\gamma d_0}+\int_{0}^{t}(d_0+\tau)^{3}|\phi_{xx}(0,\tau)|^2 \right)\nonumber\\&\leq\frac{1-m}{4}\int_{0}^{t}\int\frac{\phi_x^2 }{U^{2m}}+C\left(\mathrm{e}^{-2\gamma d_0}+\int_{0}^{t}(d_0+\tau)^{3}|\phi_{xx}(0,\tau)|^2 \right).\label{eq29}
	\end{align}
Therefore, combining \eqref{eq28} and \eqref{eq29}, noting \eqref{e2}, we arrive at
	\begin{align} \label{eq30}
&\int_{0}^{t}|d'(\tau)|\int \frac{|\phi(x,\tau)|}{U^{3m-2}(x-s\tau-d(\tau))} \notag \\
&\leq\frac{1}{4}\int_{0}^{t}\int\left(C_0\frac{|U_x|}{U^{3m} } \phi^2+(1-m)\frac{\phi_x^2 }{U^{2m}}\right)+C\left(\mathrm{e}^{-2\gamma d_0}+\int_{0}^{t}(d_0+\tau)^{3}|\phi_{xx}(0,\tau)|^2 \right).
\end{align}
Furthermore, by \eqref{sobolev embedding4}, \eqref{F1} and the fact that $\left\|\phi\left(\cdot,t\right)\right\|_{L^{\infty}}\leq C \left\|\phi/U^{m}\left(\cdot,t\right)\right\|_{L^{\infty}}\leq CN(T)$, one can see that
\begin{equation}\nonumber
	\int_{0}^{t}\int \big|F\frac{\phi}{U^{3m-1}}\big|\leq CN(T)\int_{0}^{t}\int \frac{\phi_x^2}{U^{3m-1}},
\end{equation}	
\begin{equation}\nonumber
	 \int_{0}^{t}\int\big|\frac{U_x}{U^{3m}} G \phi\big|\leq C\int_{0}^{t}\int\frac{ |\phi|\phi_x^2}{U^{3m}}\leq C N(T)\int_0^t\int \frac{\phi_x^2}{U^{2m}},
\end{equation}
and
\begin{equation}\label{eq31}
\int_{0}^{t}\int \big|G\frac{\phi_x}{U^{3m-1}}\big|\leq C\int_0^t\int \frac{|\phi_x|^3}{U^{2m+1}}\leq C N(T)\int_0^t\int \frac{\phi_x^2}{U^{2m}}.
\end{equation}
Noting that $\frac{1}{U^{3m-1}}\leq \frac{1}{U^{2m}}$ due to $m<1$, substituting \eqref{eq30}-\eqref{eq31} into \eqref{eq23} and adding \eqref{L2 estimate}, we have
\begin{equation}\nonumber
	\begin{aligned}
		&\int\frac{\phi^2}{U^{3m-1}}+\frac{C_0}{4}\int_{0}^{t}\int\frac{|U_x|}{U^{3m}}\phi^2+(1-m)\left(\frac{3}{4}-CN(T)\right)\int_{0}^{t}\int\frac{\phi_x^2}{U^{2m}}	 \\&\leq\int\frac{\phi_0^2}{U^{3m-1}} +C\left(\mathrm{e}^{-2\gamma d_0}+\int_{0}^{t}(d_0+\tau)^{3}|\phi_{xx}(0,\tau)|^2 \right).
	\end{aligned}
\end{equation}
We thus obtain \eqref{L2-1 estimate}
provided that $ CN(T)\leq 1/4$.
\end{proof}
	
\begin{lemma}\label{L2-2}
	Let the assumptions of Proposition \ref{proposition priori estimate} hold. Assume that $0<\beta \leq \frac{3m-1}{1-m}$, if $N(T)+d_0^{-\frac{1}{2}}\ll1$, then there exists a constant $C>0$ independent of $T$ such that
		\begin{align} \label{L2-2 estimate}
		&\|\phi(\cdot,t)\|_{\langle x-st-d(t)\rangle^\beta}^2+\int_0^t\left(\|\phi(\cdot,t)\|_{\langle x-st-d(t)\rangle^{\beta-1}}^2+ \left\|\phi_x(\cdot,\tau)/U^{\frac{1-m}{2}}\right\|_{\langle x-st-d(t)\rangle^\beta}^2 \right)\mathrm{d} \tau \notag \\
&  \leq C\left(\|\phi_0\|_{\langle x-d_0\rangle^\beta}^2+\|\phi_0\|_{w_{1,0}}^2+d_0^{-1} \right) \notag \\
&\ \ \ +C\int_{0}^{t}\left((d_0+\tau)^{\beta+\frac{5}{2}}+(d_0+\tau)^3 \right)|\phi_{xx}(0,\tau)|^2\mathrm{d} \tau,
	\end{align}
where  $w_{1,0}=U^{1-3m}(x-d_0)$.	
\end{lemma}	
	
\begin{proof}
Taking $\xi=x-st-d(t)$ and $\xi_\star$ such that $f'(U(\xi_\star))=s$, multiplying \eqref{eq13} by $\langle\xi-\xi_\star\rangle^\beta$ with $0<\beta\leq \frac{3m-1}{1-m}$ and integrating the result over $(0,+\infty)\times(0,t)$, using $\phi(0,t)=0$, we get	
%\begin{equation}\label{2}
%	\begin{aligned}
%		&\left(\frac{\langle\xi-\xi_\star\rangle^\beta}{2}\phi^2\right)_t+\left[ \left(\frac{1}{2} g^\prime(U)\phi^2+\frac{s}{2}\phi^2- \frac{\phi\phi_x}{U^{1-m}}-\frac{1}{m} G\phi\right)\langle\xi-\xi_\star\rangle^\beta\right]_x\\
%		&+\frac{\langle\xi-\xi_\star\rangle^{\beta-1}}{2}\mathbb{A}_\beta\phi^2+ \frac{\langle\xi-\xi_\star\rangle^\beta }{U^{1-m}}\phi_x^2\\&=-\frac{\beta}{2}d'(t)\langle\xi-\xi_\star\rangle^{\beta-2}\left( \xi-\xi_\star\right)\phi ^2-\beta\frac{\langle\xi-\xi_\star\rangle^{\beta-2} }{U^{1-m}}\left( \xi-\xi_\star\right)\phi\phi_x+d'(t)\langle\xi-\xi_\star\rangle^\beta U\phi \\&\quad+\langle\xi-\xi_\star\rangle^{\beta} F\phi-\frac{1}{m}\langle\xi-\xi_\star\rangle^{\beta} G\phi_x-\frac{\beta}{m}\langle\xi-\xi_\star\rangle^{\beta-2}(\xi-\xi_\star) G\phi,
%	\end{aligned}
%\end{equation}	
	\begin{align}\label{5}
		&\frac{1}{2}\int\langle\xi-\xi_\star\rangle^\beta\phi^2+\frac{1}{2}\int_{0}^{t}\int\langle\xi-\xi_\star\rangle^{\beta-1}\mathbb{A_\beta}\phi^2+\int_{0}^{t}\int \frac{\langle\xi-\xi_\star\rangle^\beta }{U^{1-m}}\phi_x^2 \notag \\
&= \frac{1}{2}\int\langle\xi_0-\xi_\star\rangle^\beta\phi_0^2 - \frac{\beta}{2}\int_{0}^{t}d'(\tau)\int\langle\xi-\xi_\star\rangle^{\beta-2} (\xi-\xi_\star)\phi ^2 \notag \\
& \ \ \ -\beta\int_{0}^{t}\int\frac{\langle\xi-\xi_\star\rangle^{\beta-2} }{U^{1-m}}\left( \xi-\xi_\star\right)\phi\phi_x \notag \\
&\quad+\int_{0}^{t}d'(\tau)\int\langle\xi-\xi_\star\rangle^\beta U\phi +\int_{0}^{t}\int\langle\xi-\xi_\star\rangle^{\beta} F\phi-\frac{1}{m}\int_{0}^{t}\int\langle\xi-\xi_\star\rangle^{\beta} G\phi_x \notag \\
&\quad-\frac{\beta}{m}\int_{0}^{t}\int\langle\xi-\xi_\star\rangle^{\beta-2}(\xi-\xi_\star) G\phi,
	\end{align}
where $\xi_0\triangleq x-d_0$ and
\begin{equation}\nonumber
	\mathbb{A}_\beta\triangleq- \langle\xi-\xi_\star\rangle g^{\prime\prime}(U)U_x-\beta\frac{\xi-\xi_\star }{\langle\xi-\xi_\star\rangle }g'(U).
\end{equation}
Since $g''=f''>0$, $U_x<0$ and $f'(U(\xi_\star))=s$, it follows that for any $\delta>0$, when $|\xi-\xi_\star|\geq\delta$,
\begin{equation}\nonumber
\mathbb{A}_\beta\geq-\beta\frac{\xi-\xi_\star }{\langle\xi-\xi_\star\rangle }g'(U)=-\beta\frac{\xi-\xi_\star }{\langle\xi-\xi_\star\rangle } (f'(U)-f'(U(\xi_\star)))\geq c_0(\delta)\beta.
\end{equation}
Thanks to $f\in C^2$ and $0<U<u_-$, there exists a constant $\delta_\star>0$ such that for $|\xi-\xi_\star|\leq \delta_\star$,
\begin{equation}\nonumber
	\mathbb{A}_\beta\geq- \langle\xi-\xi_\star\rangle g^{\prime\prime}(U)U_x=-\langle\xi-\xi_\star\rangle g^{\prime\prime}(U)U^{1-m}g(U)\geq c_1 .
\end{equation}
We take $\bar{ c}=\min\{c_0(\delta_\star),\frac{c_1}{\beta}\}$, then for any $\xi\in \mathbb{R_+}$, it has
\begin{equation}\label{3}
	\mathbb{A}_\beta\geq \bar{ c}\beta.
\end{equation}
%Noting that $\phi/U^{\frac{3m-1}{2}}\in L^2$,  $\phi_x/U^{\frac{1+m}{2}}\in L^2$ and $\langle\xi-\xi_\star\rangle^\beta\sim U^{-\beta(1-m)}\leq U^{1-3m}$ as $\xi\rightarrow+\infty$ due to \eqref{f''>0} and $\beta\leq \frac{3m-1}{1-m}$, which implies that
%\begin{equation}\label{4}
%\left[ \left(- \frac{\phi\phi_x}{U^{1-m}}-\frac{1}{m} G\phi\right)\langle\xi-\xi_\star\rangle^\beta\right]\bigg|_{x=+\infty}=0.	
%\end{equation}	 		
By virtue of \eqref{3.28} and $|\phi_{xx}(0,t)|\leq N(T)$ due to \eqref{N(t)}, it holds that
	 \begin{align}\label{13}
&\frac{\beta}{2}\int_{0}^{t}|d'(\tau)|\int\langle\xi-\xi_\star\rangle^{\beta-2}| \xi-\xi_\star|\phi ^2 \notag \\
&\leq\beta \sup_{\tau\in [0,t]}|d'(\tau)|\int_{0}^{t}\int\langle\xi-\xi_\star\rangle^{\beta-1}\phi^2 \notag \\
& \leq C\beta\left(N(T)+d_0^{-1}\right)\int_{0}^{t}\int\langle\xi-\xi_\star\rangle^{\beta-1}\phi^2.
	 \end{align}	
Moreover, we get by the Cauchy-Schwarz inequality,
	\begin{align}\label{6}
 &\beta\int_{0}^{t}\int\bigg|\frac{\langle\xi-\xi_\star\rangle^{\beta-2} }{U^{1-m}}\left( \xi-\xi_\star\right)\phi\phi_x\bigg| \notag \\
 &\leq C\beta\int_{0}^{t}\int\frac{\langle\xi-\xi_\star\rangle^{\beta-1} }{U^{1-m}}|\phi\phi_x| \notag \\
 &\leq	\frac{\bar{c}\beta}{8} \int_{0}^{t}\int\langle\xi-\xi_\star\rangle^{\beta-1}\phi^2	 +C\int_{0}^{t}\int\frac{\langle\xi-\xi_\star\rangle^{\beta-1} }{U^{2-2m}}\phi_x^2,
 	\end{align}
where $\bar{c}$ is the constant in \eqref{3} and
\begin{equation}\nonumber
\int_{0}^{t}\int\frac{\langle\xi-\xi_\star\rangle^{\beta-1} }{U^{2-2m}}\phi_x^2	= \int_{0}^{t} \left(\int_{|\xi-\xi_\star|<\frac{4}{U^{1-m}(\xi_\star)}}+\int_{|\xi-\xi_\star|\geq\frac{4}{U^{1-m}(\xi_\star)}} \right)\frac{\langle\xi-\xi_\star\rangle^{\beta-1} }{U^{2-2m}}\phi_x^2.
\end{equation}	
For $|\xi-\xi_\star|<\frac{4}{U^{1-m}(\xi_\star)}$, it is verified that  $0<U<u_-$, and	
\begin{equation}\label{7}
\int_{0}^{t} \int_{|\xi-\xi_\star|<\frac{4}{U^{1-m}(\xi_\star)}}\frac{\langle\xi-\xi_\star\rangle^{\beta-1} }{U^{2-2m}}\phi_x^2\leq	C\int_{0}^{t} \int \phi_x^2.
\end{equation}	
For $|\xi-\xi_\star|\geq\frac{4}{U^{1-m}(\xi_\star)}$, if $\xi-\xi_\star\leq-\frac{4}{U^{1-m}(\xi_\star)}$, that is $\xi\leq\xi_\star-\frac{4}{U^{1-m}(\xi_\star)}\leq\xi_\star$, we have $U(\xi)\geq U(\xi_\star)$ due to $U_\xi<0$, thus we derive
\begin{equation}\nonumber
	\begin{aligned}
	\int_{0}^{t} \int_{\xi-\xi_\star\leq-\frac{4}{U^{1-m}(\xi_\star)}}\frac{\langle\xi-\xi_\star\rangle^{\beta-1} }{U^{2-2m}}\phi_x^2&=\int_{0}^{t} \int_{\xi-\xi_\star\leq-\frac{4}{U^{1-m}(\xi_\star)}}\frac{1}{U^{1-m}\langle\xi-\xi_\star\rangle  }\frac{\langle\xi-\xi_\star\rangle^{\beta} }{U^{1-m}}\phi_x^2\\&\leq\int_{0}^{t} \int_{\xi-\xi_\star\leq-\frac{4}{U^{1-m}(\xi_\star)}}\frac{1}{U^{1-m}(\xi_\star)|\xi-\xi_\star|  }\frac{\langle\xi-\xi_\star\rangle^{\beta} }{U^{1-m}}\phi_x^2\\&\leq\frac{1}{4} \int_{0}^{t} \int\frac{\langle\xi-\xi_\star\rangle^{\beta} }{U^{1-m}}\phi_x^2,
\end{aligned}
\end{equation}
and if $\xi-\xi_\star\geq\frac{4}{U^{1-m}(\xi_\star)}$, by \eqref{f''>0}, we get $U^{2-4m}\sim \langle\xi-\xi_\star\rangle^{\frac{3m-1}{1-m}-1}\geq \langle\xi-\xi_\star\rangle^{\beta-1}$, and
	\begin{align} \label{8}
	\int_{0}^{t} \int_{\xi-\xi_\star\geq\frac{4}{U^{1-m}(\xi_\star)}}\frac{\langle\xi-\xi_\star\rangle^{\beta-1} }{U^{2-2m}}\phi_x^2
&=\int_{0}^{t} \int_{\xi-\xi_\star\geq\frac{4}{U^{1-m}(\xi_\star)}}\frac{\phi_x^2 }{U^{2m}}\frac{ \langle\xi-\xi_\star\rangle^{\beta-1}}{U^{2-4m}} \notag \\
&\leq C\int_{0}^{t} \int\frac{\phi_x^2 }{U^{2m}}.
\end{align}
Therefore, adding \eqref{7} and \eqref{8}, and noting $1\leq \frac{C }{U^{2m}}$, we  note that \eqref{6} is transformed into
	\begin{align}\label{10}
	&\beta\int_{0}^{t}\int\bigg|\frac{\langle\xi-\xi_\star\rangle^{\beta-2} }{U^{1-m}}\left( \xi-\xi_\star\right)\phi\phi_x\bigg| \notag \\
&\leq	\frac{\bar{ c}\beta}{8} \int_{0}^{t}\int\langle\xi-\xi_\star\rangle^{\beta-1}\phi^2+\frac{1}{4} \int_{0}^{t} \int\frac{\langle\xi-\xi_\star\rangle^{\beta} }{U^{1-m}}\phi_x^2+C\int_{0}^{t} \int\frac{\phi_x^2 }{U^{2m}}.			
\end{align}
We write the fourth term on RHS of \eqref{5} as
\begin{equation}\nonumber
\int_{0}^{t}|d'(\tau)|\int\langle\xi-\xi_\star\rangle^\beta U|\phi|=\int_{0}^{t}|d'(\tau)|\left(\int_{0}^{s\tau+d(\tau)}+\int_{s\tau+d(\tau)}^{+\infty} \right)\langle\xi-\xi_\star\rangle^\beta U|\phi|.
\end{equation}	
For $\xi>0$, by the Cauchy-Schwarz inequality, we have	
	\begin{align}\label{11}
&\int_{0}^{t}|d'(\tau)|\int_{s\tau+d(\tau)}^{+\infty} \langle\xi-\xi_\star\rangle^\beta U|\phi| \notag \\
&\leq\frac{\bar{ c}\beta}{8}\int_{0}^{t}\int_{s\tau+d(\tau)}^{+\infty}\langle\xi-\xi_\star\rangle^{\beta-1}\phi^2
+C\int_{0}^{t}|d'(\tau)|^2\int_{s\tau+d(\tau)}^{+\infty}\langle\xi-\xi_\star\rangle^{\beta+1}U^2 \notag \\
&\leq\frac{\bar{ c}\beta}{8}\int_{0}^{t}\int_{s\tau+d(\tau)}^{+\infty}\langle\xi-\xi_\star\rangle^{\beta-1}\phi^2+C\int_{0}^{t}|d'(\tau)|^2,
	\end{align}
where we have used the fact $U^2\sim\langle\xi-\xi_\star\rangle^{-\frac{2}{ 1-m}} $ as $\xi\rightarrow+\infty$, $\langle\xi-\xi_\star\rangle^{\beta+1}\leq\langle\xi-\xi_\star\rangle^{\frac{2m}{1-m}}$ and 	
\begin{equation}\nonumber
\int_{s\tau+d(\tau)}^{+\infty}\langle\xi-\xi_\star\rangle^{\beta+1}U^2\leq	C \int_{s\tau+d(\tau)}^{+\infty}\langle\xi-\xi_\star\rangle^{-2}\leq C.
\end{equation}	
For $\xi<0$, it holds
	 \begin{align} \label{9}
&\int_{0}^{t}|d'(\tau)|\int_{0}^{s\tau+d(\tau)} \langle\xi-\xi_\star\rangle^\beta U|\phi|	\notag \\
&\leq \int_{0}^{t}|d'(\tau)|\left(\sup_{x\in [0,s\tau+d(\tau)]}\langle\xi-\xi_\star\rangle^\frac{\beta-\frac{1}{2}}{2} |\phi| \right)\int_{0}^{s\tau+d(\tau)} \langle\xi-\xi_\star\rangle^\frac{\beta+\frac{1}{2}}{2}  U,
	 \end{align}
where, in view of \eqref{st+d(t)}, a direct calculation gives	
\begin{equation}\nonumber
\int_{0}^{s\tau+d(\tau)} \langle\xi-\xi_\star\rangle^\frac{\beta+\frac{1}{2}}{2}  U\leq 	C\left(s\tau+d(\tau) \right) ^\frac{\beta+\frac{5}{2}}{2}\leq 	C\left(d_0+\tau \right) ^\frac{\beta+\frac{5}{2}}{2},
\end{equation}	
and
\begin{equation}\nonumber
	\begin{aligned}
\langle\xi-\xi_\star\rangle^{\beta-\frac{1}{2}} \phi^2&=-\int_{x}^{0}\left[ \frac{\partial}{\partial x}\left(\langle\xi-\xi_\star\rangle^{\beta-\frac{1}{2}} \right)\phi^2+2\langle\xi-\xi_\star\rangle^{\beta-\frac{1}{2}}\phi\phi_{x}\right] \\&\leq C\int_{0}^{s\tau+d(\tau)}\left(\langle\xi-\xi_\star\rangle^{\beta-1}\phi^2+\langle\xi-\xi_\star\rangle^{\beta}\phi_x^2 \right).
	\end{aligned}
\end{equation}	
Then, \eqref{9} can be estimated as
\begin{eqnarray}\label{12}
&&\int_{0}^{t}|d'(\tau)|\int_{0}^{s\tau+d(\tau)} \langle\xi-\xi_\star\rangle^\beta U|\phi| \nonumber \\
&&\leq C\int_{0}^{t}|d'(\tau)|\left[\int_{0}^{s\tau+d(\tau)}\left(\langle\xi-\xi_\star\rangle^{\beta-1}\phi^2+\langle\xi-\xi_\star\rangle^{\beta}\phi_x^2 \right)\right]^\frac{1}{2}\int_{0}^{s\tau+d(\tau)} \langle\xi-\xi_\star\rangle^\frac{\beta+\frac{1}{2}}{2}  U \notag \\
&&\leq C\int_{0}^{t}|d'(\tau)|\left[\left(\int_{0}^{s\tau+d(\tau)}\langle\xi-\xi_\star\rangle^{\beta-1}\phi^2\right)^\frac{1}{2}\!\!\!+\!\!\left(\int_{0}^{s\tau+d(\tau)}\langle\xi-\xi_\star\rangle^{\beta}\phi_x^2\right)^\frac{1}{2} \right]\int_{0}^{s\tau+d(\tau)} \langle\xi-\xi_\star\rangle^\frac{\beta+\frac{1}{2}}{2}  U	 \notag \\ &&\leq\int_{0}^{t}\int_{0}^{s\tau+d(\tau)}\left(\frac{c\beta}{8}\langle\xi-\xi_\star\rangle^{\beta-1}\phi^2+\frac{1}{4}\langle\xi-\xi_\star\rangle^{\beta}\phi_x^2 \right)+C\int_{0}^{t}|d'(\tau)|^2\left(d_0+\tau \right) ^{\beta+\frac{5}{2}}.		
\end{eqnarray}
Combining \eqref{11} and \eqref{12}, using \eqref{e1}, we have
	\begin{align}\label{14}
	&\int_{0}^{t}|d'(\tau)|\int\langle\xi-\xi_\star\rangle^\beta U|\phi| \notag \\
&\leq\int_{0}^{t}\int\left(\frac{c\beta}{4}\langle\xi-\xi_\star\rangle^{\beta-1}\phi^2+\frac{1}{4}\langle\xi-\xi_\star\rangle^{\beta}\phi_x^2 \right)+C\int_{0}^{t}|d'(\tau)|^2\left(d_0+\tau \right) ^{\beta+\frac{5}{2}} \notag \\
&\leq\int_{0}^{t}\int\left(\frac{c\beta}{4}\langle\xi-\xi_\star\rangle^{\beta-1}\phi^2+\frac{1}{4}\langle\xi-\xi_\star\rangle^{\beta}\phi_x^2 \right) \notag \\
&\ \ \ +C\left(\mathrm{e}^{-2\gamma d_0} +\int_{0}^{t}\left(d_0+\tau \right) ^{\beta+\frac{5}{2}}|\phi_{x x}^2(0,\tau)|\right).
	\end{align}
Furthermore, by \eqref{F1}, \eqref{10} and the fact that $\left\|\phi\left(\cdot,t\right)\right\|_{L^{\infty}}\leq CN(T)$ and $\left\|\phi_x(\cdot,t)/U\right\|_{L^{\infty}}\leq CN(T)$, one can see that
\begin{equation}\nonumber
	\int_{0}^{t}\int\bigg|\langle\xi-\xi_\star\rangle^{\beta} F\phi\bigg|\leq CN(T)\int_{0}^{t}\int\langle\xi-\xi_\star\rangle^{\beta}\phi_x^2 ,
\end{equation}	
\begin{equation}\nonumber
	\int_{0}^{t}\int\bigg|\langle\xi-\xi_\star\rangle^{\beta} G\phi_x\bigg|\leq C\int_0^t\int \langle\xi-\xi_\star\rangle^{\beta}\frac{|\phi_x|^3}{U^{2-m}}\leq C N(T)\int_0^t\int\langle\xi-\xi_\star\rangle^{\beta} \frac{\phi_x^2}{U^{1-m}},
\end{equation}
and
	\begin{align}\label{15}
	&\int_{0}^{t}\int\bigg|\langle\xi-\xi_\star\rangle^{\beta-1} G\phi\bigg| \notag \\
&\leq CN(T)\int_{0}^{t}\int\langle\xi-\xi_\star\rangle^{\beta-1}\frac{ |\phi\phi_x|}{U^{1-m}} \notag\\
&\leq  CN(T)\beta \int_{0}^{t}\!\!\int\langle\xi-\xi_\star\rangle^{\beta-1}\phi^2\!\!+\!CN(T) \int_{0}^{t} \!\!\int\frac{\langle\xi-\xi_\star\rangle^{\beta} }{U^{1-m}}\phi_x^2\!+\!C\int_{0}^{t}\!\! \int\frac{\phi_x^2 }{U^{2m}}.
		\end{align}
Combining \eqref{13}, \eqref{10} and \eqref{14}-\eqref{15} with \eqref{5},  we deduce
\begin{equation}\nonumber
	\begin{aligned}
&\frac{1}{2}\!\int\langle\xi-\xi_\star\rangle^\beta\phi^2\!+\!c\beta\left(\frac{1}{8}\!-\!C(N(T)\!+\!d_0^{-1})\right)\!\int_{0}^{t}\!\int\langle\xi-\xi_\star\rangle^{\beta-1}\phi^2\!+\!\left(\frac{1}{2}\!-\!CN(T)\right)\int_{0}^{t}\!\int \frac{\langle\xi-\xi_\star\rangle^\beta }{U^{1-m}}\phi_x^2\\&\leq \frac{1}{2}\int\langle \xi_0-\xi_\star\rangle^\beta\phi_0^2+	C\left(\mathrm{e}^{-2\gamma d_0} +\int_{0}^{t} \int\frac{\phi_x^2 }{U^{2m}}+\int_{0}^{t}\left(d_0+\tau \right) ^{\beta+\frac{5}{2}}|\phi_{x x}^2(0,\tau)|\right)\\&\leq C\left( \int\langle \xi_0-\xi_\star\rangle^\beta\phi_0^2+\int\frac{\phi_0^2 }{U^{3m-1}}+d_0^{-1}\right)+C\int_{0}^{t}\left(\left(d_0+\tau \right) ^{\beta+\frac{5}{2}}+\left(d_0+\tau \right) ^3 \right)|\phi_{x x}^2(0,\tau)|.	
	\end{aligned}
\end{equation}	
Therefore, when $C(N(T)+d_0^{-1})\leq 1/16$, we prove \eqref{L2-2 estimate}.	
\end{proof}	
	
\subsubsection{Estimates for first order derivative}

\begin{lemma}\label{H1}
	Let the assumptions of Proposition \ref{proposition priori estimate} hold. If $N(T)+d_0^{-\frac{1}{2}}\ll1$, then there exists a constant $C>0$ independent of $T$ such that
	\begin{equation}\label{H1 estimate}
		\begin{aligned}
		&\|\phi_x(\cdot,t)\|_{w_2}^2+\int_0^t\left\|\phi_{xx}(\cdot,\tau)\right\|_{w_5}^2 \mathrm{d} \tau   \\&\leq C\left(\|\phi_0\|_{w_{1,0}}^2+\|\phi_{0x}\|_{w_{2,0}}^2+d_0^{-1} \right)+C\int_{0}^{t}(d_0+\tau)^{3}|\phi_{xx}(0,\tau)|^2\mathrm{d} \tau,	
	\end{aligned}
	\end{equation}	
where $w_1=U^{1-3m}$,  $w_2=U^{-(1+m)}$, $w_5=U^{-2}$ for $U=U(x-st-d(t))$ and $w_{i,0}(U)~(i=1,2)$ for $U=U(x-d_0)$.
\end{lemma}	
	
\begin{proof}
Differentiating \eqref{phiequ} in $x$ gives
\begin{equation}\label{phizequ}
	\phi_{xt}+f''(U)U_x \phi_x+f'(U)\phi_{x x}-\left(\frac{\phi_x}{U^{1-m}}\right)_{xx}=d'(t)U_x+F_x+\frac{1}{m}G_{xx}.
\end{equation}	
From \eqref{decay u_-} and \eqref{decompose}, we have
\begin{equation}\label{phix 0}
	\phi_x(0,t)=u_--U(-st-d(t))\sim \mathrm{e}^{-\lambda_-(st+d(t))},
\end{equation}
and thus
\begin{equation}\label{3.92}
	\int_{0}^{t}\phi_{x}(0,\tau)\leq C\int_{0}^{t}\mathrm{e}^{-\lambda_-(s\tau+d(\tau))}\leq C\mathrm{e}^{-d_0}\leq Cd_0^{-1}
\end{equation}
due to \eqref{3.37}. Multiplying \eqref{phizequ} by $\frac{\phi_{x}}{U^{1+m}}$ and integrating the result over $(0,+\infty)\times(0,t)$, noting \eqref{3.92},
%\begin{equation}\nonumber
%	\frac{\phi_x\phi_{xt}}{U^{1+m}}=\frac{1}{2}\left(\frac{\phi_x^2}{U^{1+m}} \right)_t+\frac{1+m}{2}\frac{U_x}{U^{2+m}}(-s-d'(t))\phi_x^2,
%\end{equation}
%and
%\begin{equation}\nonumber
%	\begin{aligned}
%	-\left(\frac{\phi_x}{U^{1-m}}\right)_{xx}\frac{\phi_x}{U^{1+m}}&=-\left( \frac{\phi_{xx}}{U^{1-m}}-(1-m)\frac{U_x}{U^{2-m} }\phi_{x}\right)_x\frac{\phi_x}{U^{1+m}}\\&=-\left[\left( \frac{\phi_{xx}}{U^{1-m}}-(1-m)\frac{U_x}{U^{2-m} }\phi_{x}\right)\frac{\phi_x}{U^{1+m}}\right]_x\\&\quad+\left( \frac{\phi_{xx}}{U^{1-m}}-(1-m)\frac{U_x}{U^{2-m} }\phi_{x}\right)\left(\frac{\phi_{xx}}{U^{1+m}}-(1+m)\frac{U_x}{U^{2+m}}\phi_x \right),
%	\end{aligned}
%\end{equation}
%we get
%\begin{equation}\nonumber
%	 \begin{aligned}
%&\frac{1}{2}\left(\frac{\phi_x^2}{U^{1+m}} \right)_t+\left[-\left( \frac{\phi_{xx}}{U^{1-m}}-(1-m)\frac{U_x}{U^{2-m} }\phi_{x}\right)\frac{\phi_x}{U^{1+m}}-\frac{1}{m}G_x\frac{ \phi_x}{U^{1+m}}	 \right]_x+\frac{\phi_{xx}^2}{U^2}\\&=-f''(U)\frac{U_x }{U^{1+m}}\phi_x^2-\frac{1+m}{2}\frac{U_x}{U^{2+m}}\left[ \left( -s-d'+2(1-m)\frac{U_x}{U^{2-m}}\right)\phi_x^2\right]-f'(U)\frac{\phi_x\phi_{xx}}{U^{1+m}}\\&\quad+2\frac{U_x}{U^{3}}\phi_x\phi_{xx}+d'\frac{ U_x}{U^{1+m}}\phi_x	+F_x \frac{ \phi_x}{U^{1+m}}-\frac{1}{m}G_x\frac{ \phi_{xx}}{ U^{1+m}}+\frac{1+m}{m}\frac{U_x }{ U^{2+m}}G_x\phi_{x}.
%	 \end{aligned}
%\end{equation}
one obtains
	\begin{align}\label{eq32}
		&\frac{1}{2}\int\frac{\phi_x^2}{U^{1+m}} +\int_{0}^{t}\int\frac{\phi_{xx}^2}{U^2}\notag\\
&\leq\frac{1}{2} \int\frac{\phi_{0x}^2}{U^{1+m}}
 +Cd_0^{-1} \notag \notag\\
  & \ \ \ +C\left[  \int_{0}^{t}\int f''(U)\frac{|U_x| }{U^{1+m}}\phi_x^2
 +\int_{0}^{t}\int\frac{|U_x|}{U^{2+m}} \bigg|\!\left( \!-\!s\!-\!d'+2(1\!-\!m)\frac{U_x}{U^{2-m}}\!\right)\!\bigg|\phi_x^2\right.\notag\\
 &\left.\quad+\int_{0}^{t}\int \left(\frac{|f'(U)|}{U^{1+m}}+\frac{|U_x|}{U^{3}} \right)|\phi_x\phi_{xx}|+\int_{0}^{t} |d'(\tau)|\int\frac{ |U_x|}{U^{1+m}}|\phi_x|	 \right.\notag\\
 &\left.\quad+\int_{0}^{t}\int\left(\big| F_x \frac{ \phi_x}{U^{1+m}}\big|+\big|G_x\frac{ \phi_{xx}}{ U^{1+m}}\big|+\big|\frac{U_x }{ U^{2+m}}G_x\phi_{x}\big| \right)\right].
	\end{align}

We next estimate the terms on the RHS of \eqref{eq32}. By $|U_x|\leq CU^{2-m}$ and the Cauchy-Schwarz inequality, it holds that	
\begin{equation}\label{eq33}
\int_{0}^{t}\int f''(U)\frac{|U_x| }{U^{1+m}}\phi_x^2\leq	 C\int_{0}^{t}\int\frac{\phi_x^2}{U^{2m-1}},
\end{equation}	
\begin{equation}
	\int_{0}^{t}\int\frac{|U_x|}{U^{2+m}} \bigg|\!\left( \!-\!s\!-\!d'+2(1\!-\!m)\frac{U_x}{U^{2-m}}\!\right)\!\bigg|\phi_x^2\leq	 C\int_{0}^{t}\int\frac{\phi_x^2}{U^{2m}},
\end{equation}
and	
\begin{equation}\label{eq34}
	\int_{0}^{t}\int \left(\frac{|f'(U)|}{U^{1+m}}+\frac{|U_x|}{U^{3}} \right)|\phi_x\phi_{xx}|\leq \frac{1}{4}\int_{0}^{t}\int\frac{\phi_{xx}^2}{U^{2}}+C\int_{0}^{t}\int\frac{\phi_x^2}{U^{2m}}.
\end{equation}
To estimate the sixth term on RHS of \eqref{eq32}, we decompose it as
\begin{equation}\nonumber
	 \int_{0}^{t}|d'(\tau)|\int \frac{|U_x|}{U^{1+m}}|\phi_x|=\int_{0}^{t}|d'(\tau)|\left( \int_{0}^{s\tau+d(\tau)}+\int_{s\tau+d(\tau)}^{+\infty}\right)\frac{|U_x|}{U^{1+m}}|\phi_x|.
\end{equation}
For $\xi>0$, utilizing \eqref{f''>0}, \eqref{e1}, $|U_x|\leq CU^{2-m}$, we get from H\"{o}lder's inequality and Cauchy-Schwarz inequality that
	\begin{align}\label{eq35}
	&\int_{0}^{t}|d'(\tau)|\int_{s\tau+d(\tau)}^{+\infty}\frac{|U_x|}{U^{1+m}}|\phi_x|\notag \\
&\leq C\int_{0}^{t}|d'(\tau)|\int_{s\tau+d(\tau)}^{+\infty}\frac{|\phi_x|}{U^{2m-1}}\notag \\
&\leq C\int_{0}^{t}|d'(\tau)|\left(\int_{s\tau+d(\tau)}^{+\infty}\frac{\phi_x^2}{U^{2m}} \right)^{\frac{1}{2}}\left(\int_{s\tau+d(\tau)}^{+\infty}\frac{1}{U^{2m-2} }\right)^{\frac{1}{2}}	 \notag \\
&\leq C\left(\int_{0}^{t}\int\frac{\phi_x^2}{U^{2m}}+\int_{0}^{t}|d'(\tau)|^2\int_{s\tau+d(\tau)}^{+\infty}\frac{1 }{U^{2m-2} } \right)\notag \\
&\leq C\left(\mathrm{e}^{-2\gamma d_0}+ \int_{0}^{t}\int\frac{\phi_x^2}{U^{2m}}+\int_{0}^{t}|\phi_{xx}(0,\tau)|^2\right).
	\end{align}
For $\xi<0$, as in the proof of \eqref{eq16}, one can see that
\begin{equation}\nonumber
		\int_{0}^{s\tau+d(\tau)}|\phi_x(x,\tau)|\mathrm{d}x	 \leq|s\tau+d(\tau)|^{\frac{3}{2}}\left(\int_{0}^{s\tau+d(\tau)}|\phi_{yy}(y,\tau)|^2\mathrm{d}y\right)^{\frac{1}{2}}+|s\tau+d(\tau)||\phi_y(0,\tau)|,
\end{equation}
which along with \eqref{st+d(t)}, \eqref{e1} and \eqref{phix 0} implies that
	\begin{align}\label{eq36}
		&\int_{0}^{t}|d'(\tau)|\int_{0}^{s\tau+d(\tau)}\frac{|U_x|}{U^{1+m}}|\phi_x| \notag \\
		&\leq C\int_{0}^{t}(d_0+\tau)^{\frac{3}{2}}|d'(\tau)|\left(\int_{0}^{s\tau+d(\tau)}\phi_{xx}^2\right)^{\frac{1}{2}}
+C\int_{0}^{t}|d'(\tau)|(d_0+\tau)|\phi_x(0,\tau)| \notag\\
&\leq\frac{1}{4u_-^2}\int_{0}^{t}\int\phi_{xx}^2+C\left(\mathrm{e}^{-2\gamma d_0}+ \int_{0}^{t}(d_0+\tau)^{3}|d'(\tau)|^2\right)\notag \\
&\leq\frac{1}{4}\int_{0}^{t}\int\frac{\phi_{xx}^2 }{U^2}+C\left(\mathrm{e}^{-2\gamma d_0}+\int_{0}^{t}(d_0+\tau)^{3}|\phi_{xx}(0,\tau)|^2 \right).
	\end{align}
Therefore, combining \eqref{eq35} and \eqref{eq36}, the sixth term can be estimated as
\begin{align}\label{eq37}
	&\int_{0}^{t}|d'(\tau)|\int \frac{|U_x|}{U^{1+m}}|\phi_x| \notag \\
&\leq\frac{1}{4}\int_{0}^{t}\int\frac{\phi_{xx}^2}{U^{2} } +C\left(\mathrm{e}^{-2\gamma d_0}+\int_{0}^{t}\int\frac{\phi_x^2 }{U^{2m}}+\int_{0}^{t}(d_0+\tau)^{3}|\phi_{xx}(0,\tau)|^2 \right).
\end{align}
By virtue of $|U_x|\leq CU^{2-m}$ and the fact that $\left\|\phi_x(\cdot,t)/U\right\|_{L^{\infty}}\leq CN(T)$, it follows from the Cauchy-Schwarz inequality, \eqref{F1} and \eqref{Gz} that
\begin{equation}\nonumber
\begin{aligned}
	\int_{0}^{t}\int\big| F_x \frac{ \phi_x}{U^{1+m}}\big|
&\leq\int_{0}^{t}\int\left(\frac{|U_x| }{U^{1+m}}|\phi_x|^3+\frac{\phi_x^2|\phi_{xx}| }{U^{1+m}} \right) \\
&\leq CN(T)\int_{0}^{t}\int\frac{\phi_{xx}^2 }{U^{2}}+C\int_{0}^{t}\int\frac{\phi_x^2 }{U^{2m-2}},
\end{aligned}
\end{equation}	 	
\begin{equation}\nonumber
\begin{aligned}
\int_{0}^{t}\int \big|G_x\frac{ \phi_{xx}}{ U^{1+m}}\big|
&\leq\int_{0}^{t}\int\left(\frac{|U_x| }{U^{4}}\phi_x^2|\phi_{x x}|+\frac{|\phi_{x}|\phi_{xx}^2 }{U^{3}} \right) \\
&\leq CN(T)\int_{0}^{t}\int\frac{\phi_{xx}^2 }{U^{2}}+C\int_{0}^{t}\int\frac{\phi_x^2 }{U^{2m}},
\end{aligned}	
\end{equation}	
and	
\begin{align}\label{eq38}
\int_{0}^{t}\int \big|\frac{U_x }{ U^{2+m}}G_x\phi_{x}\big| &\leq\int_{0}^{t}
\int\left(\frac{U_x^2 }{U^{5}}|\phi_x|^3+\frac{|U_x| }{U^{4}} \phi_{x}^2|\phi_{xx}|\right) \notag \\
&\leq CN(T)\int_{0}^{t}\int\frac{\phi_{xx}^2 }{U^{2}}+C\int_{0}^{t}\int\frac{\phi_x^2 }{U^{2m}}.	
\end{align}	
Combining \eqref{eq32}, \eqref{eq33}-\eqref{eq34} with \eqref{eq37}-\eqref{eq38}, and noting that $\frac{1}{U^{2m-2}}\leq\frac{C}{U^{2m-1}}\leq\frac{C}{U^{2m}}$, yields
\begin{equation}\nonumber
	\begin{aligned}
		&\frac{1}{2}\int\frac{\phi_x^2}{U^{1+m}} +\left(\frac{1}{2}-CN(T)\right)\int_{0}^{t}\int\frac{\phi_{xx}^2}{U^2}\\&\leq C\left(\int\frac{\phi_{0x}^2}{U^{1+m}}+d_0^{-1} \right)+C\left(\mathrm{e}^{-2\gamma d_0}+\int_{0}^{t}\int\frac{\phi_x^2 }{U^{2m}}+\int_{0}^{t}(d_0+\tau)^{3}|\phi_{xx}(0,\tau)|^2 \right).
	\end{aligned}
\end{equation}		
Taking $N(T)$ sufficiently small such that $CN(T)\leq 1/4$, utilizing \eqref{L2-1 estimate}, we then get \eqref{H1 estimate}.
\end{proof}

\begin{lemma}\label{H1-2}
	Let the assumptions of Proposition \ref{proposition priori estimate} hold. Assume that $0<\beta \leq \frac{3m-1}{1-m}$, if $N(T)+d_0^{-\frac{1}{2}}\ll1$, then there exists a constant $C>0$ independent of $T$ such that
		\begin{align}\label{H1-2 estimate}
			&\|\phi_x(\cdot,t)\|_{\langle x-st-d(t)\rangle ^{\beta+1}}^2+\int_0^t \left\|\phi_{xx}(\cdot,\tau)/U^{\frac{1-m}{2}}\right\|_{\langle x-st-d(t)\rangle ^{\beta+1}}^2 \mathrm{d} \tau \notag \\
&  \leq C\left(\|\phi_0\|_{1,\langle x-d_0\rangle ^{\beta}}^2+\|\phi_0\|_{w_{1,0}}^2+d_0^{-1} \right) \notag \\
& \ \ \ +C\int_{0}^{t}\left( (d_0+\tau)^{\beta+\frac{5}{2}}+(d_0+\tau)^3\right)|\phi_{xx}(0,\tau)|^2\mathrm{d} \tau,
		\end{align}
where   $w_{1,0}=U^{1-3m}(x-d_0)$.
\end{lemma}		
	
\begin{proof}
Multiplying \eqref{phizequ} by $\langle\xi-\xi_\star\rangle ^{\beta+1}\phi_{x}$ and integrating the result over $(0,+\infty)\times(0,t)$, noting that \eqref{phix 0}, $|\xi-\xi_\star|\leq\langle\xi-\xi_\star\rangle$,
we get
\begin{eqnarray}\label{16}
		&&\frac{1}{2}\int\langle\xi-\xi_\star\rangle ^{\beta+1} \phi_x^2+\int_{0}^{t}\int\langle\xi-\xi_\star\rangle ^{\beta+1}\frac{\phi_{xx}^2}{U^{1-m}}
\notag \\
&&\leq\frac{1}{2} \int\langle\xi_0-\xi_\star\rangle ^{\beta+1} \phi_{0x}^2+Cd_0^{-1} +C\left\{\int_{0}^{t}\int \langle\xi-\xi_\star\rangle ^{\beta+1}f''(U)|U_x|\phi_x^2+\int_{0}^{t}\int\langle\xi-\xi_\star\rangle ^{\beta} \right. \notag \\
&&\left.\quad\times\bigg|\left( 
	-s-d'+2(1-m)\frac{U_x}{U^{2-m}}\right)\bigg|\phi_x^2+\int_{0}^{t}\int\left( \langle\xi-\xi_\star\rangle ^{\beta+1}\frac{|U_x|}{U^{2-m}}+\frac{\langle\xi-\xi_\star\rangle ^{\beta} }{U^{1-m}}\right)|\phi_x\phi_{xx}|\right. \notag \\
&&\quad\!+\!\int_{0}^{t}|d'(\tau)|\int \langle\xi\!-\!\xi_\star\rangle ^{\beta+1}|U_x||\phi_x| \notag \\
&& \ \ \  +\!\int_{0}^{t}\int\!\left[\langle\xi-\xi_\star\rangle ^{\beta+1}\left(\big|F_x\phi_x\big| \!+\!\big|G_x\phi_{xx}\big| \right)\!+\!\langle\xi\!-\!\xi_\star\rangle ^{\beta}\big|G_x\phi_{x}\big| \right]\Big\}.
\end{eqnarray}	
As $\xi\rightarrow-\infty$, by \eqref{decay u_-}, one has
\begin{equation}\label{e5}
	 |U_x(\xi)|\sim |U-u_-|\sim \mathrm{e}^{-\lambda_-|\xi|},
\end{equation}
and as $\xi\rightarrow+\infty$,
\begin{equation}\label{e6}
|U_x(\xi)|\sim U^{2-m}\sim |\xi|^{-1-\frac{1}{1-m}},	
\end{equation}
then
\begin{equation}\label{20}
\int_{0}^{t}\int \langle\xi-\xi_\star\rangle ^{\beta+1}f''(U)|U_x|\phi_x^2\leq C\int_{0}^{t}\int \langle\xi-\xi_\star\rangle ^{\beta}\phi_x^2.
\end{equation}	
In view of $|U_x|\leq CU^{2-m}$, we have
\begin{equation}\nonumber
\int_{0}^{t}\int\langle\xi-\xi_\star\rangle ^{\beta} \bigg|\left( -s-d'+2(1-m)\frac{U_x}{U^{2-m}}\right)\bigg|\phi_x^2	\leq C \int_{0}^{t}\int\langle\xi-\xi_\star\rangle ^{\beta} \phi_x^2.	
\end{equation}	
Moreover, utilizing \eqref{e5} and \eqref{e6}, we get by the Cauchy-Schwarz inequality that	
	 \begin{align} \label{21}
&\int_{0}^{t}\int\left( \langle\xi-\xi_\star\rangle ^{\beta+1}\frac{|U_x|}{U^{2-m}}+\frac{\langle\xi-\xi_\star\rangle ^{\beta} }{U^{1-m}}\right)|\phi_x\phi_{xx}| \notag \\
&\leq \frac{1}{4}\int_{0}^{t}\int\langle\xi-\xi_\star\rangle ^{\beta+1}\frac{\phi_{xx}^2}{U^{1-m}}+C\int_{0}^{t}\int\left(\langle\xi-\xi_\star\rangle ^{\beta+1}\frac{ |U_x|}{U}+\frac{\langle\xi-\xi_\star\rangle ^{\beta-1}}{U^{1-m}} \right)\phi_{x}^2 \notag \\
&\leq \frac{1}{4}\int_{0}^{t}\int\langle\xi-\xi_\star\rangle ^{\beta+1}\frac{\phi_{xx}^2}{U^{1-m}}+C\int_{0}^{t}\int\langle\xi-\xi_\star\rangle ^{\beta}\phi_{x}^2.
	 \end{align}
Decompose the sixth term on RHS of \eqref{16} as	
\begin{equation}\nonumber
\int_{0}^{t}|d'(\tau)|\int \langle\xi-\xi_\star\rangle ^{\beta+1}|U_x||\phi_x|=\int_{0}^{t}|d'(\tau)|\left(\int_{0}^{s\tau+d(\tau)}+\int_{s\tau+d(\tau)}^{+\infty} \right)\langle\xi-\xi_\star\rangle ^{\beta+1}|U_x||\phi_x|. 	
\end{equation}	
For $\xi>0$, it holds
	\begin{align} \label{18}
		&\int_{0}^{t}|d'(\tau)|\int_{s\tau+d(\tau)}^{+\infty} \langle\xi-\xi_\star\rangle ^{\beta+1}|U_x||\phi_x| \notag \\
&\leq\frac{1}{2}\int_{0}^{t}\int_{s\tau+d(\tau)}^{+\infty}\langle\xi-\xi_\star\rangle^{\beta}\frac{\phi_x^2 }{U^{1-m}}+\frac{1}{2}\int_{0}^{t}|d'(\tau)|^2\int_{s\tau+d(\tau)}^{+\infty}\langle\xi-\xi_\star\rangle^{\beta+2}U_x^2U^{1-m} \notag \\
&\leq C\left( \int_{0}^{t}\int\langle\xi-\xi_\star\rangle^{\beta}\frac{\phi_x^2 }{U^{1-m}}+\int_{0}^{t}|d'(\tau)|^2\right),
	\end{align}
where we have used the fact $U_x^2U^{1-m}\sim\langle\xi-\xi_\star\rangle^{-\frac{5-3m}{ 1-m}} $ as $\xi\rightarrow+\infty$ and 	
\begin{equation}\nonumber
	\int_{s\tau+d(\tau)}^{+\infty}\langle\xi-\xi_\star\rangle^{\beta+2}U_x^2U^{1-m}\leq	C \int_{s\tau+d(\tau)}^{+\infty}\langle\xi-\xi_\star\rangle^{-4}\leq C,
\end{equation}	
due to $\beta\leq \frac{3m-1}{1-m}$. For $\xi<0$,
	\begin{align} \label{17}
		&\int_{0}^{t}|d'(\tau)|\int_{0}^{s\tau+d(\tau)} \langle\xi-\xi_\star\rangle ^{\beta+1}|U_x||\phi_x|	\notag \\
&\leq \int_{0}^{t}|d'(\tau)|\left(\sup_{x\in [0,s\tau+d(\tau)]}\langle\xi-\xi_\star\rangle^\frac{\beta+\frac{1}{2}}{2} |\phi_x| \right)\int_{0}^{s\tau+d(\tau)} \langle\xi-\xi_\star\rangle^\frac{\beta+\frac{1}{2}}{2}  |U_x|,
	\end{align}	
where, in view of \eqref{st+d(t)}, a direct calculation gives	
\begin{equation}\nonumber
	\int_{0}^{s\tau+d(\tau)} \langle\xi-\xi_\star\rangle^\frac{\beta+\frac{1}{2}}{2}  |U_x|\leq C\left(d_0+\tau \right) ^\frac{\beta+\frac{5}{2}}{2},
\end{equation}	
and
\begin{equation}\nonumber
	\begin{aligned}
		&\langle\xi-\xi_\star\rangle^{\beta+\frac{1}{2}} \phi_x^2 \\
&=-\int_{x}^{0}\left[ \frac{\partial}{\partial x}\left(\langle\xi-\xi_\star\rangle^{\beta+\frac{1}{2}} \right)\phi_x^2+2\langle\xi-\xi_\star\rangle^{\beta+\frac{1}{2}}\phi_x\phi_{xx}\right]+ \langle\xi-\xi_\star\rangle^{\beta+\frac{1}{2}}\big|_{x=0} \phi_x^2(0,\tau)\\&\leq C\int_{0}^{s\tau+d(\tau)}\left(\langle\xi-\xi_\star\rangle^{\beta}\phi_x^2+\langle\xi-\xi_\star\rangle^{\beta+1}\phi_{xx}^2 \right)+ \langle\xi-\xi_\star\rangle^{\beta+\frac{1}{2}}\big|_{x=0} \phi_x^2(0,\tau).
	\end{aligned}
\end{equation}	
Then, by \eqref{phix 0} and $\langle\xi-\xi_\star\rangle^{\beta+\frac{1}{2}}\big|_{x=0}\sim (d_0+\tau)^{\beta+\frac{1}{2}}$ due to \eqref{st+d(t)}, \eqref{17} is transformed into
	\begin{align}
		&\int_{0}^{t}|d'(\tau)|\int_{0}^{s\tau+d(\tau)} \langle\xi-\xi_\star\rangle^{\beta+1} |U_x||\phi_x|\nonumber\\&\leq C\!\int_{0}^{t}\!|d'(\tau)|\!\left[\int_{0}^{s\tau+d(\tau)}\!\left(\!\langle\xi\!-\!\xi_\star\rangle^{\beta}\phi_x^2\!+\!\langle\xi\!-\!\xi_\star\rangle^{\beta+1}\phi_{xx}^2 \right)\!\!+\! \langle\xi\!-\!\xi_\star\rangle^{\beta+\frac{1}{2}}\big|_{x=0} \phi_x^2(0,\tau)\right]^\frac{1}{2}\!\left(d_0\!+\!\tau \right) ^\frac{\beta+\frac{5}{2}}{2}\nonumber\\&%\leq C\int_{0}^{t}|d'(\tau)|\left[\left(\int_{0}^{s\tau+d(\tau)}\langle\xi-\xi_\star\rangle^{\beta}\phi_x^2\right)^\frac{1}{2}+\left(\int_{0}^{s\tau+d(\tau)}\langle\xi-\xi_\star\rangle^{\beta+1}\phi_{xx}^2\right)^\frac{1}{2}\right]\left(d_0+\tau \right) ^\frac{\beta+\frac{5}{2}}{2}\nonumber\\&\quad+C\int_{0}^{t}|d'(\tau)|\langle\xi-\xi_\star\rangle^{\frac{\beta+\frac{1}{2} }{2}}\big|_{x=0} |\phi_x(0,\tau)| \left(d_0+\tau \right) ^\frac{\beta+\frac{5}{2}}{2}	\nonumber\\&
		\leq\!\frac{1}{4}\int_{0}^{t}\!\int_{0}^{s\tau+d(\tau)}\!\langle\xi\!-\!\xi_\star\rangle^{\beta+1}\phi_{xx}^2 \!+\!C\left(\!\mathrm{e}^{-2\gamma d_0}\!+\!\int_{0}^{t}\!\int_{0}^{s\tau+d(\tau)}\!\langle\xi\!-\!\xi_\star\rangle^{\beta}\phi_x^2\!+\!\int_{0}^{t}\!|d'(\tau)|^2\left(d_0\!+\!\tau \right) ^{\beta+\frac{5}{2}}\! \right)	.\label{19} \nonumber
	\end{align}
This along with \eqref{18}, using \eqref{e1}, we have
\begin{eqnarray}\label{22}
&&\int_{0}^{t}|d'(\tau)|\int\langle\xi-\xi_\star\rangle^\beta |U_x||\phi_x| \notag \\
&&\leq\frac{1}{4}\int_{0}^{t}\int\langle\xi-\xi_\star\rangle^{\beta+1}\phi_{xx}^2 \notag \\
&&\ \ \ +C\left(\mathrm{e}^{-2\gamma d_0}+\int_{0}^{t}\int\langle\xi-\xi_\star\rangle^{\beta}\frac{\phi_x^2 }{U^{1-m}}
+\int_{0}^{t}|d'(\tau)|^2\left(d_0+\tau \right) ^{\beta+\frac{5}{2}} \right).
\end{eqnarray}
Furthermore, by virtue of \eqref{e5}, \eqref{e6} and $\left\|\langle\xi-\xi_\star\rangle ^{\frac{\beta+1 }{2}}\phi_x(\cdot,t)\right\|_{L^{\infty}}\leq CN(T)$ due to \eqref{N(t)}, it follows from the Cauchy-Schwarz inequality and \eqref{F1} that
\begin{equation}\nonumber
	\begin{aligned}
	&\int_{0}^{t}\int\bigg|\langle\xi-\xi_\star\rangle ^{\beta+1}F_x\phi_x\bigg| \\&%\leq\int_{0}^{t}\int\langle\xi-\xi_\star\rangle ^{\beta+1}|\phi_x|\left(|U_x| \phi_x^2+|\phi_x\phi_{xx}|\right) \\&
	\leq CN(T)\int_{0}^{t}\int\langle\xi-\xi_\star\rangle ^{\beta+1}\frac{\phi_{xx}^2 }{U^{1-m}}+C\int_{0}^{t}\int\left(U^{1-m}\phi_x^2+\langle\xi-\xi_\star\rangle ^{\frac{\beta+1 }{2}}|U_x|\phi_x^2\right)\\&\leq CN(T)\int_{0}^{t}\int\langle\xi-\xi_\star\rangle ^{\beta+1}\frac{\phi_{xx}^2 }{U^{1-m}}+C\int_{0}^{t}\int\langle\xi-\xi_\star\rangle ^{\beta}\phi_x^2.
		\end{aligned}
\end{equation}	
Similarly, we deduce from \eqref{Gz}, \eqref{e5} and $\left\|\phi_x(\cdot,t)/U\right\|_{L^{\infty}}\leq CN(T)$ that	
	\begin{align}
	&\int_{0}^{t}\int\bigg|\langle\xi-\xi_\star\rangle ^{\beta+1}G_x\phi_{xx}\bigg|\nonumber\\&%\leq\int_{0}^{t}\int\langle\xi-\xi_\star\rangle ^{\beta+1}|\phi_{xx}|\left(\frac{|U_x| }{U^{3-m}}\phi_x^2+\frac{|\phi_{x}\phi_{xx}| }{U^{2-m}} \right)\\&
	\leq CN(T)\int_{0}^{t}\int\langle\xi-\xi_\star\rangle ^{\beta+1}\frac{\phi_{xx}^2 }{U^{1-m}}+C\int_{0}^{t}\int\langle\xi-\xi_\star\rangle ^{\beta}\frac{\phi_x^2 }{U^{1-m}}\cdot\langle\xi-\xi_\star\rangle\frac{U_x^2 }{U^{2}}\nonumber\\& \leq CN(T)\int_{0}^{t}\int\langle\xi-\xi_\star\rangle ^{\beta+1}\frac{\phi_{xx}^2 }{U^{1-m}}+C\int_{0}^{t}\int\langle\xi-\xi_\star\rangle ^{\beta}\frac{\phi_x^2 }{U^{1-m}},	\nonumber
		\end{align} 	
where we have used \eqref{e5} and $\langle\xi-\xi_\star\rangle\sim U^{1-m}$ as $\xi\rightarrow+\infty$	in the last inequality, and by $|U_x|\leq CU^{2-m}$,	and
\begin{eqnarray}\label{23}
&&\int_{0}^{t}\int\bigg|\langle\xi-\xi_\star\rangle ^{\beta}G_x\phi_{x}\bigg|\notag \\
&&	\leq CN(T)\int_{0}^{t}\int\langle\xi-\xi_\star\rangle ^{\beta}\frac{\phi_{xx}^2 }{U^{1-m}}+C\int_{0}^{t}\int\left(\langle\xi-\xi_\star\rangle ^{\beta}\frac{\phi_x^2 }{U^{1-m}} +\langle\xi-\xi_\star\rangle ^{\beta}\phi_x^2\right).
\end{eqnarray}	
Therefore, combining \eqref{20}-\eqref{21}, \eqref{22}-\eqref{23} and \eqref{16}, one has
\begin{equation}\nonumber
	\begin{aligned}
		&\frac{1}{2}\int\langle\xi-\xi_\star\rangle ^{\beta+1} \phi_x^2+\left(\frac{1}{4}-CN(T) \right)\int_{0}^{t}\int\langle\xi-\xi_\star\rangle ^{\beta+1}\frac{\phi_{xx}^2}{U^{1-m}}\\&\text{\small$\leq\frac{1}{2}\int\langle \xi_0-\xi_\star\rangle ^{\beta+1} \phi_{0x}^2+C\int_{0}^{t}\int\left(d_0^{-1}+\langle\xi-\xi_\star\rangle ^{\beta-1}\phi^2 +\langle\xi-\xi_\star\rangle ^{\beta}\frac{\phi_x^2 }{U^{1-m}}\right) +\int_{0}^{t}\left(d_0+\tau \right) ^{\beta+\frac{5}{2}}|\phi_{x x}^2(0,\tau)|$}\\&\text{\small$\leq C\left( \int\langle \xi_0-\xi_\star\rangle^\beta\phi_0^2\!+\!\int\langle \xi_0-\xi_\star\rangle ^{\beta+1} \phi_{0x}^2\!+\!\int\frac{\phi_0^2 }{U^{3m-1}}+d_0^{-1}\right)\!+\!	C\int_{0}^{t}\left( (d_0\!+\!\tau)^{\beta+\frac{5}{2}}\!+\!(d_0\!+\!\tau)^3\right)|\phi_{x x}^2(0,\tau)|$}.
	\end{aligned}
\end{equation}	
Taking $N(T)$ sufficiently small such that $CN(T)\leq 1/8$, we then get \eqref{H1-2 estimate}.	
\end{proof}	

\subsubsection{Estimates for second order derivative}	
\begin{lemma}\label{H2}
	Let the assumptions of Proposition \ref{proposition priori estimate} hold. If $N(T)+d_0^{-\frac{1}{2}}\ll1$, then there exists a constant $C>0$ independent of $T$ such that
	\begin{equation}\label{H2 estimate}
		\begin{aligned}
			&\|\phi_{xx}(\cdot,t)\|_{w_3}^2+\int_0^t\left\|\phi_{xxx}(\cdot,\tau)\right\|_{w_6}^2 \mathrm{d} \tau   \\&\leq C\left(\|\phi_0\|_{w_{1,0}}^2+\|\phi_{0x}\|_{w_{2,0}}^2+\|\phi_{0xx}\|_{w_{3,0}}^2+d_0^{-1} \right)+C\int_{0}^{t}(d_0+\tau)^{\beta+3}|\phi_{xx}(0,\tau)|^2\mathrm{d} \tau,
		\end{aligned}
	\end{equation}	
where  $w_1=U^{1-3m}$, $w_2=U^{-(1+m)}$, $w_3=U^{m-3}$, $w_6=U^{2m-4}$ for $U=U(x-st-d(t))$ and $w_{i,0}~(i=1,2,3)$ for $U=U(x-d_0)$.
\end{lemma}	

\begin{proof}
	Differentiate \eqref{phizequ} in $x$,
	\begin{eqnarray}\label{phizzequ}
		&&\phi_{xxt}+f'(U)\phi_{x xx}-\left(\frac{\phi_{xxx}}{U^{1-m}}\right)_{x} \notag \\
&&=-2f''(U)U_x \phi_{xx}-f'''(U)U_x^2\phi_x-f''(U)U_{xx}\phi_x \notag \\
&&\ \ \ +\left[-2(1-m)\frac{U_x}{U^{2-m}}\phi_{xx}-(1-m)\frac{U_{xx}}{U^{2-m}}\phi_{x}+(1-m)(2-m)\frac{U_x^2}{U^{3-m}}\phi_{x}\right]_x \notag \\
&&\ \ \ +d'(t)U_{xx}+F_{xx}+\frac{1}{m}G_{xxx}.	
	\end{eqnarray}	
Set $U_\epsilon\triangleq U+\epsilon$, where $\epsilon>0$ is a small constant. It is easy to see that $U_{\epsilon x}=U_x$.
Multiplying \eqref{phizzequ} by $\frac{\phi_{xx}}{U_\epsilon^{3-m}}$, integrating the result over $(0,+\infty)\times (0,t)$, we then get
	\begin{align}
		&\frac{1}{2}\int\frac{\phi_{xx}^2}{U_\epsilon^{3-m}} +\int_{0}^{t}\int\frac{\phi_{xxx}^2}{U^{1-m}U_\epsilon^{3-m}}+\int_{0}^{t}\left(\frac{\phi_{xx}\phi_{xxx}}{U^{1-m}U_\epsilon^{3-m}}
+\frac{1}{m}G_{xx}\frac{\phi_{xx}}{U_\epsilon^{3-m}}\right)\bigg|_{x=0}\nonumber\\
&\leq\frac{1}{2} \int\frac{\phi_{0xx}^2}{U_\epsilon^{3-m}}+C\left[ \int_{0}^{t}\int\left( \frac{|U_x|}{U_\epsilon^{4-m}}\big|\left( -s-d'\right)\big|+f''(U)\frac{|U_{x}| }{U_\epsilon^{3-m}}\right)\phi_{xx}^2\right.\nonumber\\&\left.\quad+\int_{0}^{t}\!\int \left(|f'''(U)|\frac{U_x^2 }{U_\epsilon^{3-m}}\!+\! f''(U)\frac{ |U_{xx}| }{U_\epsilon^{3-m}}\right)|\phi_x\phi_{xx}|\!+\!\int_{0}^{t}\int\left(\frac{|U_x|}{U^{1-m}U_\epsilon^{4-m} }\!+\!\frac{|f'(U)| }{ U_\epsilon^{3-m}} \right)|\phi_{xx}\phi_{xxx}| \right.\nonumber\\
&\left.\quad+\int_{0}^{t}\int\bigg|\left[-2(1-m)\frac{U_x}{U^{2-m}}\phi_{xx}-(1-m)\frac{U_{xx}}{U^{2-m}}\phi_{x}
+(1-m)(2-m)\frac{U_x^2}{U^{3-m}}\phi_{x}\right]\frac{\phi_{xxx}}{U_\epsilon^{3-m}}\bigg|\right.\nonumber\\
&\left.\quad+\int_{0}^{t}\int\bigg|\left[-2(1-m)\frac{U_x}{U^{2-m}}\phi_{xx}
-(1-m)\frac{U_{xx}}{U^{2-m}}\phi_{x}+(1-m)(2-m)\frac{U_x^2}{U^{3-m}}\phi_{x}\right]\frac{U_x}{U_\epsilon^{4-m}}\phi_{xx}\bigg|\right.\nonumber\\
&\left.\quad-\int_{0}^{t}\left[\left( -2(1-m)\frac{U_x}{U^{2-m}}\phi_{xx}-(1-m)\frac{U_{xx}}{U^{2-m}}\phi_{x}
+(1-m)(2-m)\frac{U_x^2}{U^{3-m}}\phi_{x}\right)\frac{\phi_{xx}}{U_\epsilon^{3-m}}\right]\bigg|_{x=0}\right.\nonumber\\
&\quad+\int_{0}^{t}|d'(\tau)|\int\frac{ |U_{xx}|}{U_\epsilon^{3-m}}|\phi_{xx}| \notag \\
&\quad +\int_{0}^{t}\int \left( \bigg|F_{xx} \frac{ \phi_{xx}}{U_\epsilon^{3-m}}\bigg|
+\bigg|G_{xx}\frac{ \phi_{xxx}}{ U_\epsilon^{3-m}}\bigg|+\bigg|\frac{U_x }{ U_\epsilon^{4-m}}G_{xx}\phi_{xx}\bigg|\right)\Big].\label{eq47}	
	\end{align}
In view of \eqref{Gzz}, \eqref{boundary phixxx}, $\left\|\phi_x/U\left(\cdot,t\right)\right\|_{L^{\infty}}\leq CN(T)$,  $|\phi_{xx}(0,t)|\leq CN(T)$ and $U_\epsilon(-st-d(t))<u_-+1$, the value from the boundary $x=0$ is equal to
	\begin{align}\label{eq41}
	&\left(\frac{\phi_{xx}\phi_{xxx}}{U^{1-m}U_\epsilon^{3-m}}+\frac{1}{m}G_{xx}\frac{\phi_{xx}}{U_\epsilon^{3-m}}	 \right) \bigg|_{x=0} \notag \\
&\geq \left(1-CN(T)\right)\frac{\phi_{xx}\phi_{xxx}}{U^{1-m}U_\epsilon^{3-m}}\bigg|_{x=0}-CN(T)\left( \phi_x^2+\phi_{xx}^2\right)\big|_{x=0} \notag \\
&\geq c_2\phi_{xx}^2(0,t)-C\left( U_x^2+U_{xx}^2+\phi_{x}^2\right)\big|_{x=0},		
	\end{align}
for a constant $c_2>0$, provided $N(T)$ is sufficiently small.

Next, we estimate the terms on RHS of \eqref{eq47}. By virtue of $d\in C^1$, $|U_x|\leq C U^{2-m}$ and $|U_{xx}|\leq C U^{3-2m}$, noting $\frac{1}{U_\epsilon}\leq \frac{1}{U}\leq\frac{C}{U^2} $, we get by Cauchy-Schwarz inequality that
\begin{equation}\label{eq52}\nonumber
\int_{0}^{t}\int\left( \frac{|U_x|}{U_\epsilon^{4-m}}\big|\left( -s-d'\right)\big|+f''(U)\frac{|U_{x}| }{U_\epsilon^{3-m}}\right)\phi_{xx}^2\leq C	\int_{0}^{t}\int \frac{\phi_{xx}^2 }{U^{2}},
\end{equation}
\begin{equation}\nonumber
\int_{0}^{t}\!\int \left(|f'''(U)|\frac{U_x^2 }{U_\epsilon^{3-m}}\!+\! f''(U)\frac{ |U_{xx}| }{U_\epsilon^{3-m}}\right)|\phi_x\phi_{xx}|\leq C\int_{0}^{t}\int\frac{|\phi_x\phi_{xx}| }{U^{m}}	\leq C\int_{0}^{t}\int\left( \frac{\phi_x^2 }{U^{2m}}+\frac{\phi_{xx}^2 }{U^{2m}} \right),
\end{equation}
and that
\begin{equation}\nonumber
\int_{0}^{t}\int\left(\frac{|U_x|}{U^{1-m}U_\epsilon^{4-m} }\!+\!\frac{|f'(U)| }{ U_\epsilon^{3-m}} \right)|\phi_{xx}\phi_{xxx}|	 %\leq \frac{1}{4}\int_{0}^{t}\int\frac{\phi_{xxx}^2}{U^{1-m}U_\epsilon^{3-m}}+C\int_{0}^{t}\int\left(\frac{U_x^2}{U^{1-m}U_\epsilon^{5-m}}+\frac{U^{1-m}}{U_\epsilon^{3-m}} \right)\phi_{xx}^2\\&
\leq \frac{1}{4}\int_{0}^{t}\int\frac{\phi_{xxx}^2}{U^{1-m}U_\epsilon^{3-m}}+C\int_{0}^{t}\int\frac{\phi_{xx}^2 }{U^{2}}.
\end{equation}
Similarly, it holds
\begin{equation}\nonumber
\text{the fifth term on RHS of \eqref{eq47}}
%\leq C\int_{0}^{t}\int	\left(  \frac{|\phi_{xx}\phi_{xxx} |}{ U_\epsilon^{3-m}}+\frac{|\phi_{x}\phi_{xxx} |}{U^{m-1} U_\epsilon^{3-m}}	\right)\nonumber\\&
\leq \frac{1}{4}\int_{0}^{t}\int\frac{\phi_{xxx}^2}{U^{1-m}U_\epsilon^{3-m}}+C\int_{0}^{t}\int\left( \frac{\phi_{x}^2 }{U^{2m}}+\frac{\phi_{xx}^2 }{U^{2}}\right),	
\end{equation}
and noting $U^2\leq \frac{C}{U^2}$, it has
	\begin{equation}
\text{the sixth term on RHS of \eqref{eq47}}\leq C\int_{0}^{t}\int\left(\frac{\phi_{xx}^2 }{U^{2}} +\frac{|\phi_x\phi_{xx}| }{U^{m-1}} \right) \leq C\int_{0}^{t}\int\left(\frac{\phi_{x}^2 }{U^{2m}} +\frac{\phi_{xx}^2 }{U^{2}} \right)\nonumber
	\end{equation}
and
\begin{equation}\label{eq46}
		\text{the seventh term on RHS of \eqref{eq47}}%\!=\!-\! 2(1\!-\!m)\frac{|U_x|}{U^{2-m}U_\epsilon^{3-m}}\phi_{xx}^2\bigg|_{x=0}\!\!+\!	\left[\left( (1-m)\frac{U_{xx}}{U^{2-m}}\phi_{x}\!-\!(1-m)(2-m)\frac{U_x^2}{U^{3-m}}\phi_{x}\right)\frac{\phi_{xx}}{U_\epsilon^{3-m}}\right]\bigg|_{x=0}\\&
		\leq \frac{c_2}{2}\int_{0}^{t}\phi_{xx}^2(0,t)+C\int_{0}^{t} \phi_{x}^2(0,t),
\end{equation}
where we have ignored the negative term $- 2(1-m)\frac{|U_x|}{U^{2-m}U_\epsilon^{3-m}}\phi_{xx}^2\big|_{x=0}$.
We decompose the eighth term on RHS of \eqref{eq47} as
\begin{equation}\nonumber
\int_{0}^{t}|d'(\tau)|\int \frac{|U_{xx}|}{U_\epsilon^{3-m}}|\phi_{xx}|=\int_{0}^{t}|d'(\tau)|\left( \int_{0}^{s\tau+d(\tau)}+\int_{s\tau+d(\tau)}^{+\infty}\right)\frac{|U_{xx}|}{U_\epsilon^{3-m}}|\phi_{xx}|.
\end{equation}
For $\xi>0$, by \eqref{f''>0}, \eqref{e1}, $|U_{xx}|\leq CU^{3-2m}$, H\"{o}lder's inequality and the Cauchy-Schwarz inequality, we have
	\begin{align}\label{eq48}
		&\int_{0}^{t}|d'(\tau)|\int_{s\tau+d(\tau)}^{+\infty}\frac{|U_{xx}|}{U_\epsilon^{3-m}}|\phi_{xx}| \notag \\
&\leq C\int_{0}^{t}|d'(\tau)|\int_{s\tau+d(\tau)}^{+\infty}\frac{|\phi_{xx}|}{U^{m}} \notag \\
&\leq C\int_{0}^{t}|d'(\tau)|\left(\int_{s\tau+d(\tau)}^{+\infty}\frac{\phi_{xx}^2}{U^{2}} \right)^{\frac{1}{2}}\left(\int_{s\tau+d(\tau)}^{+\infty}\frac{1}{U^{2m-2} }\right)^{\frac{1}{2}}	 \notag \\
&\leq C\left(\int_{0}^{t}\int\frac{\phi_{xx}^2}{U^{2}}+\int_{0}^{t}|d'(\tau)|^2\int_{s\tau+d(\tau)}^{+\infty}\frac{1 }{U^{2m-2} } \right) \notag \\
&\leq C\left(\mathrm{e}^{-2\gamma d_0}+ \int_{0}^{t}\int\frac{\phi_{xx}^2}{U^{2}}+\int_{0}^{t}|\phi_{xx}(0,\tau)|^2\right).
	\end{align}
For $\xi<0$, following the proof of \eqref{eq16}, one can see that
\begin{equation}\nonumber
	\int_{0}^{s\tau+d(\tau)}|\phi_{xx}(x,\tau)|\mathrm{d}x	 \leq|s\tau+d(\tau)|^{\frac{3}{2}}\left(\int_{0}^{s\tau+d(\tau)}|\phi_{yyy}(y,\tau)|^2\mathrm{d}y\right)^{\frac{1}{2}}+|s\tau+d(\tau)||\phi_{yy}(0,\tau)|,
\end{equation}
which in combination with $\frac{1}{U_\epsilon}\leq\frac{C}{U}$, $U(-st-d(t))>U(0)$, \eqref{e1} and \eqref{st+d(t)} implies that
	\begin{align}\label{eq49}
		&\int_{0}^{t}|d'(\tau)|\int_{0}^{s\tau+d(\tau)}\frac{|U_{xx}|}{U_\epsilon^{3-m}}|\phi_{xx}| \notag 
		\\
&\leq C\int_{0}^{t}(d_0+\tau)^{\frac{3}{2}}|d'(\tau)|\left(\int_{0}^{s\tau+d(\tau)}\phi_{xxx}^2\right)^{\frac{1}{2}}+C\int_{0}^{t}|d'(\tau)|(d_0+\tau)|\phi_{xx}(0,\tau)| \notag \\
&\leq\frac{1}{8u_-^{4-2m}}\int_{0}^{t}\int\phi_{xxx}^2+C\left(  \int_{0}^{t}(d_0+\tau)^{3}|d'(\tau)|^2\!+\!\int_{0}^{t}(d_0+\tau)^2|\phi_{xx}(0,\tau)|^2\right) \notag \\
&\leq\frac{1}{8}\int_{0}^{t}\int\frac{\phi_{xxx}^2}{U^{4-2m}}\!\!+\!C\left(\mathrm{e}^{-2\gamma d_0}\!\!+\!\!\int_{0}^{t}(d_0+\tau)^{3}|\phi_{xx}(0,\tau)|^2 \right).
	\end{align}
Therefore, combining \eqref{eq48} and \eqref{eq49}, we arrive at
\begin{align}\label{eq54}
&\int_{0}^{t}|d'(\tau)|\int \frac{|U_{xx}|}{U_\epsilon^{3-m}}|\phi_{xx}| \notag \\
&\leq\frac{1}{8}\int_{0}^{t}\int\frac{\phi_{xxx}^2}{U^{4-2m}}+C\left(\mathrm{e}^{-2\gamma d_0}\!+\!\int_{0}^{t}\int\frac{\phi_{xx}^2 }{U^{2}}+\int_{0}^{t}(d_0+\tau)^{3}|\phi_{xx}(0,\tau)|^2 \right).
\end{align}
Furthermore, by $|U_x|\leq CU^{2-m}$ and $|U_{xx}|\leq CU^{3-2m}$, it follows from \eqref{Fzz} and \eqref{Gzz} that
\begin{equation}\nonumber
	|F_{xx}|\leq C\left[ \left(U^{4-2m}+U^{3-2m}\right) \phi_{x}^2+\phi_{xx}^2+U^{2-m}|\phi_{x}||\phi_{xx}|+|\phi_{x}||\phi_{xxx}|\right],
\end{equation}				
and
\begin{equation}\nonumber
	|G_{xx}| \leq C\left(\frac{\phi_{x}^2 }{U^m}+ \frac{\phi_{xx}^2 }{U^{2-m}}+\frac{|\phi_{x}||\phi_{xx}| }{U}+\frac{ |\phi_{x}||\phi_{xxx}|}{U^{2-m}}\right),
\end{equation}
Then, together with $\left\|\phi_x(\cdot,t)/U\right\|_{L^{\infty}}\leq CN(T)$, the last term on the right hand side of \eqref{eq47} can be estimated as
	\begin{align}
		&\int_{0}^{t}\int \left( \bigg|F_{xx} \frac{ \phi_{xx}}{U_\epsilon^{3-m}}\bigg|+\bigg|G_{xx}\frac{ \phi_{xxx}}{ U_\epsilon^{3-m}}\bigg|+\bigg|\frac{U_x }{ U_\epsilon^{4-m}}G_{xx}\phi_{xx}\bigg|\right)\nonumber\\
		&\leq C\int_{0}^{t}\int\left(\frac{|\phi_{xx}|^3 }{U_\epsilon^{3-m}}+\frac{|U_{x}||\phi_{xx}|^3}{U^{2-m}U_\epsilon^{4-m}}  \right)+C\int_{0}^{t}	 \int\frac{\phi_{xx}^2|\phi_{xxx}|}{U^{2-m}U_\epsilon^{3-m}}\nonumber\\&\quad+CN(T)\int_{0}^{t}\int \left[ \frac{\phi_{xxx}^2}{U^{1-m}U_\epsilon^{3-m}}+\left(1+\frac{1}{U^2} \right)\phi_{xx}^2 +\left(\frac{1}{U^{m}U_\epsilon} +\frac{1}{U^2}\right)|\phi_{x}||\phi_{xxx}| \right.\nonumber\\&\left.\quad+\left(U^{2-m}+ \frac{1 }{U^{1+m}}\right)|\phi_{x}||\phi_{xx}|+\left(\frac{1}{U_\epsilon^{2-m}}+\frac{ 1}{U_\epsilon^{3-m}} \right)|\phi_{xx}||\phi_{xxx}|\right] \nonumber\\&\leq C\int_{0}^{t}\int\frac{|\phi_{xx}|^3}{U_\epsilon^{4-m}}  \!+\!C\int_{0}^{t}	 \int\frac{\phi_{xx}^2|\phi_{xxx}|}{U^{2-m}U_\epsilon^{3-m}}\!+\! CN(T)\int_{0}^{t}\int \frac{\phi_{xxx}^2}{U^{1-m}U_{\epsilon}^{3-m}}+C\int_{0}^{t}\int\left(\frac{\phi_{x}^2}{U^{2m}}+  \frac{\phi_{xx}^2}{U^2} \right).\nonumber
	\end{align}
Utilizing $\|\phi_{xx}(\cdot,t)/U^{\frac{3-m}{2}}\|_{L^2}\leq N(T)$, we get from the H\"{o}lder's inequality that
\begin{equation}\label{II2}
	\begin{aligned}
 \int_{0}^{t}\int\frac{|\phi_{xx}|^3 }{U_\epsilon^{4-m}}&\leq C\int_{0}^{t}\left\|\frac{\phi_{xx} }{U_\epsilon^{\frac{3-m}{2}} }\right\|_{L^{\infty}}\left\|\frac{\phi_{xx} }{U_\epsilon^{\frac{3-m}{2}} }\right\|_{L^{2}}\left\|\frac{\phi_{xx} }{U_\epsilon} \right\|_{L^{2}}\\&%\leq CN(T)\int_{0}^{t}\left\|\frac{\phi_{xx} }{U_\epsilon^{\frac{3-m}{2}} }\right\|_{L^{\infty}}\left\|\frac{\phi_{xx} }{U} \right\|_{L^{2}}\\&
		\leq CN(T)\int_{0}^{t}\int\frac{\phi_{xxx}^2}{U^{1-m}U_\epsilon^{3-m}}+C \int_{0}^{t}\int\frac{\phi_{xx}^2 }{U^{2}},\nonumber
	\end{aligned}	
\end{equation}
where we have used
\begin{equation}\nonumber
		\frac{\phi_{xx}^2 }{U_\epsilon^{3-m} }\!=\!-\int_{x}^{+\infty}\!\!\left( \frac{\phi_{xx}^2 }{U_\epsilon^{3-m} }\right)_x\!=\!\int \frac{2\phi_{xx}\phi_{xxx} }{U_\epsilon^{3-m} }-(3-m)\int \frac{U_x}{U_\epsilon^{4-m} }\phi_{xx}^2\!\leq\! C\left(\! \int\frac{\phi_{xxx}^2 }{U^{1-m}U_\epsilon^{3-m}}\!+\!\int\frac{\phi_{xx}^2 }{U^{2}}\right).
\end{equation}
Similarly, it holds
	\begin{align} \label{eq55}
		 \int_{0}^{t}	 \int\frac{\phi_{xx}^2|\phi_{xxx}|}{U^{2-m}U_\epsilon^{3-m}}
		 &\leq C\int_{0}^{t} \left\|\frac{ \phi_{xx}}{U_{\epsilon}^{\frac{3-m}{2}}}
            \right\|_{L^{\infty}}\left\|\frac{\phi_{xx} }{U^{\frac{3-m}{2}} }\right\|_{L^{2}}
            \left\|\frac{\phi_{xxx} }{U^{\frac{1-m}{2}}U_\epsilon^{\frac{3-m}{2}}} \right\|_{L^{2}} \notag \\
		 &
		 \leq CN(T)\int_{0}^{t}\int\frac{\phi_{xxx}^2}{U^{1-m}U_\epsilon^{3-m}}+C \int_{0}^{t}\int\frac{\phi_{xx}^2 }{U^{2}}.
	\end{align}
Therefore, combining \eqref{eq41}-\eqref{eq46} with \eqref{eq54}-\eqref{eq55}, noting that \eqref{st+d(t)}, \eqref{phix 0}, $1\leq\frac{C}{U^{2m}}\leq \frac{C}{U^{2}}$ and
\begin{equation}\label{e7}
	\int_{0}^{t}\left( U_x^2+U_{xx}^2\right)\big|_{x=0}\leq C\int_{0}^{t}\mathrm{e}^{-s\tau-d(\tau)}\leq C\int_{0}^{t}\mathrm{e}^{-(d_0+\tau)}\leq Cd_0^{-1},
\end{equation}
if $d_0\gg1$, then
\begin{equation}\nonumber
	 \begin{aligned}
&\frac{1}{2}\int\frac{\phi_{xx}^2}{U_\epsilon^{3-m}} +\left(\frac{1}{2}-CN(T)\right)\int_{0}^{t}\int\frac{\phi_{xxx}^2}{U^{1-m}U_\epsilon^{3-m}}\\&\leq\frac{1}{8} \int_{0}^{t}\!\int\frac{\phi_{xxx}^2}{U^{4-2m}}\!+\!\frac{1}{2}\int\frac{\phi_{0xx}^2}{U_\epsilon^{3-m}} \!+\!C\left(d_0^{-1}\!+\!\int_{0}^{t}\!\!\int\left(\frac{\phi_x^2 }{U^{2m}}\!+\!\frac{\phi_{xx}^2 }{U^{2}} \right)\!+\!\!\int_{0}^{t}(d_0\!+\!\tau)^{3}|\phi_{xx}(0,\tau)|^2 \right).	 	
	 \end{aligned}
\end{equation}	
Taking $N(T)$ sufficiently small such that $CN(T)\leq 1/4$, utilizing \eqref{L2-1 estimate} and \eqref{H1 estimate}, we get
\begin{equation}\nonumber
	\begin{aligned}
		&\frac{1}{2}\int\frac{\phi_{xx}^2}{U_\epsilon^{3-m}} +\frac{1}{4}\int_{0}^{t}\int\frac{\phi_{xxx}^2}{U^{1-m}U_\epsilon^{3-m}}\\&\leq\frac{1}{8} \int_{0}^{t}\int\frac{\phi_{xxx}^2}{U^{4-2m}}\!+\!C\left(\int\frac{\phi_{0xx}^2}{U^{3-m}}\!+\!\int\frac{\phi_{0x}^2}{U^{1+m}}\!+\!\int\frac{\phi_{0}^2}{U^{3m-1}}+d_0^{-1} +\int_{0}^{t}(d_0+\tau)^{3}|\phi_{xx}(0,\tau)|^2 \right),	 	
	\end{aligned}
\end{equation} 	
where $C$ is independent of $\epsilon$. Now letting $\epsilon\rightarrow0^+$, it follows from the Fatou's Lemma that
	\begin{align}
		&\frac{1}{2}\int\frac{\phi_{xx}^2}{U^{3-m}} +\frac{1}{4}\int_{0}^{t}\int\frac{\phi_{xxx}^2}{U^{4-2m}}\leq\varliminf_{\epsilon\rightarrow0^+ }\left(\frac{1}{2}\int\frac{\phi_{xx}^2}{U_\epsilon^{3-m}} +\frac{1}{4}\int_{0}^{t}\int\frac{\phi_{xxx}^2}{U^{1-m}U_\epsilon^{3-m}}\right)\nonumber\\&\leq\frac{1}{8} \int_{0}^{t}\int\frac{\phi_{xxx}^2}{U^{4-2m}}\!+\!C\left(\int\frac{\phi_{0xx}^2}{U^{3-m}}\!+\!\int\frac{\phi_{0x}^2}{U^{1+m}}\!+\!\int\frac{\phi_{0}^2}{U^{3m-1}}+d_0^{-1} +\int_{0}^{t}(d_0+\tau)^{3}|\phi_{xx}(0,\tau)|^2 \right)\nonumber,	 	
	\end{align}
which implies that \eqref{H2 estimate}.
\end{proof}

\begin{lemma}\label{H2-2}
	Let the assumptions of Proposition \ref{proposition priori estimate} hold. Assume that $0<\beta \leq \frac{3m-1}{1-m}$, if $N(T)+d_0^{-\frac{1}{2}}\ll1$, then there exists a constant $C>0$ independent of $T$ such that
		\begin{align}\label{H2-2 estimate}
			&\|\phi_{xx}(\cdot,t)\|_{\langle x-st-d(t)\rangle ^{\beta+2}}^2+\int_{0}^{t}(d_0+\tau)^{\beta+2}|\phi_{xx}(0,\tau)|^2+\int_0^t \left\|\phi_{xxx}(\cdot,\tau)/U^{\frac{1-m}{2}}\right\|_{\langle x-st-d(t)\rangle ^{\beta+2}}^2 \mathrm{d} \tau \notag \\
&  \leq C\left(\|\phi_0\|_{2,\langle x-d_0\rangle ^{\beta}}^2+\|\phi_0\|_{w_{1,0}}^2+\|\phi_{0x}\|_{w_{2,0}}^2+d_0^{-1} \right) \notag \\
&\ \ \ +C\int_{0}^{t}\left( (d_0+\tau)^{\beta+\frac{5}{2}}+(d_0+\tau)^3\right)|\phi_{xx}(0,\tau)|^2\mathrm{d} \tau,
		\end{align}	
where  $w_{1,0}=U^{1-3m}$, $ w_{2,0}=U^{-(1+m)}$  for $U=U(x-d_0)$.
\end{lemma}		

\begin{proof}
Multiplying \eqref{phizzequ} by $\langle\xi-\xi_\star\rangle ^{\beta+2}\phi_{xx}$, integrating the result over $(0,+\infty)\times (0,t)$, one obtains
	\begin{align}
		&\frac{1}{2}\!\int\langle\xi\!-\!\xi_\star\rangle ^{\beta+2} \phi_{xx}^2\!-\!\int_{0}^{t}\!\left[\left(\frac{f'(U)\phi_{xx}}{2}\!-\! \frac{\phi_{xxx}}{U^{1-m}}\!-\!\frac{G_{xx}}{m}\right)\langle\xi\!-\!\xi_\star\rangle ^{\beta+2}\phi_{xx} \!\right]\bigg|_{x=0}\!\!+\!\int_{0}^{t}\!\int\langle\xi\!-\!\xi_\star\rangle ^{\beta+2}\frac{\phi_{xxx}^2}{U^{1-m}}\nonumber\\&\leq\frac{1}{2}\int\!\langle\xi_0\!-\!\xi_\star\rangle ^{\beta+2} \phi_{0xx}^2\!+\!C\left\{\int_{0}^{t}\!\int\!\left(\langle\xi\!-\!\xi_\star\rangle ^{\beta+1} \big|\!\left( g'\!-\!d'\right)\!\big|+\langle\xi\!-\!\xi_\star\rangle ^{\beta+2}f''(U)|U_x| \right)\phi_{xx}^2\right.\nonumber\\&\left.\!\quad+\!\!\int_{0}^{t}\!\!\int\!\langle\xi\!-\!\xi_\star\rangle ^{\beta+2}|f'''|U_x^2|\phi_x\phi_{xx}|\!+\!\!\int_{0}^{t}\!\!\int\!\langle\xi\!-\!\xi_\star\rangle ^{\beta+2}f''(U)|U_{xx}||\phi_x\phi_{xx}|\!+\!\!\int_{0}^{t}\!\!\int\langle\xi\!-\!\xi_\star\rangle ^{\beta+1}\frac{|\phi_{xx}\phi_{xxx}| }{U^{1-m}}\right.\nonumber\\&\left.\quad+\int_{0}^{t}\int\bigg|\left[\!-2(1\!-\!m)\frac{U_x}{U^{2-m}}\phi_{xx}\!-\!(1\!-\!m)\frac{U_{xx}}{U^{2-m}}\phi_{x}\!+\!(1\!-\!m)(2\!-\!m)\frac{U_x^2}{U^{3-m}}\phi_{x}\right]\langle\xi\!-\!\xi_\star\rangle ^{\beta+2}\phi_{xxx}\bigg|\right.\nonumber\\&\left.\quad+\int_{0}^{t}\int\bigg|\left[\!-2(1\!-\!m)\frac{U_x}{U^{2-m}}\phi_{xx}\!-\!(1\!-\!m)\frac{U_{xx}}{U^{2-m}}\phi_{x}\!+\!(1\!-\!m)(2\!-\!m)\frac{U_x^2}{U^{3-m}}\phi_{x}\right]\langle\xi\!-\!\xi_\star\rangle ^{\beta+1}\phi_{xx}\bigg|\right.\nonumber\\&\left.\quad\!-\!\int_{0}^{t}\!\left[\left( \!-2(1\!-\!m)\frac{U_x}{U^{2-m}}\phi_{xx}\!-\!(1\!-\!m)\frac{U_{xx}}{U^{2-m}}\phi_{x}\!+\!(1\!-\!m)(2\!-\!m)\frac{U_x^2}{U^{3-m}}\phi_{x}\right)\langle\xi\!-\!\xi_\star\rangle ^{\beta+2}\phi_{xx}	\right]\!\bigg|_{x=0}\right.\nonumber\\&\left.\quad+\int_{0}^{t} |d'(\tau)|\int\langle\xi-\xi_\star\rangle ^{\beta+2}|U_{xx}||\phi_{xx}|	 +\int_{0}^{t}\int\left( \langle\xi-\xi_\star\rangle ^{\beta+2}|F_{xx}\phi_{xx}|+\langle\xi-\xi_\star\rangle ^{\beta+2}|G_{xx}\phi_{xxx}|\right.\right.\nonumber\\&\left.\left.\quad+\langle\xi-\xi_\star\rangle ^{\beta+1}|G_{xx}\phi_{xx}|\right)\right\}.\label{32}
	\end{align}
Noting that $f''>0$, $U(-st-d(t))<u_-$, in view of \eqref{Gzz}, \eqref{st+d(t)}, $\left\|\phi_x/U\left(\cdot,t\right)\right\|_{L^{\infty}}\leq CN(T)$ and $|\phi_{xx}(0,t)|\leq CN(T)$, we have
\begin{equation}\nonumber
	\begin{aligned}
		&-\left[\left(\frac{f'(U)\phi_{xx}}{2}- \frac{\phi_{xxx}}{U^{1-m}}-\frac{G_{xx}}{m}\right)\langle\xi-\xi_\star\rangle ^{\beta+2}\phi_{xx} \!\right]\bigg|_{x=0}\\&\geq\langle\xi-\xi_\star\rangle ^{\beta+2}\big|_{x=0}\left[ \left(-\frac{f'(u_-)\phi^2_{xx}}{2}+(1-CN(T)) \frac{\phi_{xx}\phi_{xxx}}{U^{1-m}}\right)\bigg|_{x=0}-CN(T)\left(\phi^2_{x}+ \phi^2_{xx}\right)\big|_{x=0}\right]\\&\geq \langle\xi-\xi_\star\rangle ^{\beta+2}\big|_{x=0}\left[ \left(\frac{1}{4}-CN(T)\right)f'(u_-)\phi_{xx}^2(0,t)-C\left( U_x^2+U_{xx}^2+\phi_{x}^2\right)\big|_{x=0}\right]\\&\geq (d_0+t)^{\beta+2}\left[  c_3\phi_{xx}^2(0,t)-C\left( U_x^2+U_{xx}^2+\phi_{x}^2\right)\big|_{x=0}	\right].
	\end{aligned}	
\end{equation}
We deduce from \eqref{e5}-\eqref{e6} and the Cauchy-Schwarz inequality that
\begin{equation}\label{33}
\int_{0}^{t}\!\int\!\left(\langle\xi\!-\!\xi_\star\rangle ^{\beta+1} \big|\!\left( g'\!-\!d'\right)\!\big|+\langle\xi\!-\!\xi_\star\rangle ^{\beta+2}f''(U)|U_x| \right)\phi_{xx}^2\leq C	 \int_{0}^{t}\int\langle\xi-\xi_\star\rangle ^{\beta+1}\phi_{xx}^2,
\end{equation}
and
\begin{equation}\label{34}
	\int_{0}^{t}\int\langle\xi-\xi_\star\rangle ^{\beta+2}|f'''(U)|U_x^2|\phi_x\phi_{xx}| %&\leq C\left(\int_{0}^{t}\int\langle\xi-\xi_\star\rangle ^{\beta+1}\phi_{xx}^2+\int_{0}^{t}\int\langle\xi-\xi_\star\rangle ^{\beta+3}U_x^4\phi_{x}^2\right)\\&
	\leq C\left(\int_{0}^{t}\int\langle\xi-\xi_\star\rangle ^{\beta+1}\phi_{xx}^2+\int_{0}^{t}\int\langle\xi-\xi_\star\rangle ^{\beta}\phi_{x}^2\right).
\end{equation}
In view of \eqref{direct Uzz}, it has
\begin{equation}\label{e8}
	|U_{xx}(\xi)|\sim \mathrm{e}^{-\lambda_-|\xi|},\quad \text{as }\xi\rightarrow-\infty,
\end{equation}
and
\begin{equation}\label{e9}
	|U_{xx}(\xi)|\sim U^{3-2m}\sim |\xi|^{-2-\frac{1}{1-m}},\quad \text{as }\xi\rightarrow+\infty,
\end{equation}
then we have
	\begin{align} \label{35}
	 &\int_{0}^{t}\!\!\int\!\langle\xi\!-\!\xi_\star\rangle ^{\beta+2}f''(U)|U_{xx}||\phi_x\phi_{xx}| \notag \\
&\leq C\left(\int_{0}^{t}\int\langle\xi-\xi_\star\rangle ^{\beta+1}\phi_{xx}^2+\int_{0}^{t}\int\langle\xi-\xi_\star\rangle ^{\beta+3}U_{xx}^2\phi_{x}^2\right) \notag \\
&\leq C\left(\int_{0}^{t}\int\langle\xi-\xi_\star\rangle ^{\beta+1}\phi_{xx}^2+\int_{0}^{t}\int\langle\xi-\xi_\star\rangle ^{\beta}\phi_{x}^2\right).
 \end{align}	
Similarly, it holds
\begin{equation}\nonumber
	\int_{0}^{t}\int\langle\xi\!-\!\xi_\star\rangle ^{\beta+1}\frac{|\phi_{xx}\phi_{xxx}| }{U^{1-m}}\leq\frac{1}{8} \int_{0}^{t}\int\langle\xi-\xi_\star\rangle ^{\beta+2}\frac{\phi_{xxx}^2}{U^{1-m}}+C\int_{0}^{t}\int\langle\xi-\xi_\star\rangle ^{\beta}\frac{\phi_{xx}^2 }{U^{1-m}}.
\end{equation}
Moreover, it follows from \eqref{e5}-\eqref{e6}, \eqref{e8}-\eqref{e9} and the Cauchy-Schwarz inequality that
\begin{equation}\nonumber
	\begin{aligned}
&\text{the sixth term on RHS of \eqref{32}}\\&%\leq\frac{1}{8} \int_{0}^{t}\int\langle\xi-\xi_\star\rangle ^{\beta+2}\frac{\phi_{xxx}^2}{U^{1-m}}+C\int_{0}^{t}\int\langle\xi-\xi_\star\rangle ^{\beta+2}\left(\frac{U_x^2}{U^{3-m}}\phi_{xx}^2+\frac{U_{xx}^2}{U^{3-m}}\phi_{x}^2+\frac{U_{x}^4}{U^{5-m}}\phi_{x}^2 \right)\\&
\leq\frac{1}{8} \int_{0}^{t}\int\langle\xi-\xi_\star\rangle ^{\beta+2}\frac{\phi_{xxx}^2}{U^{1-m}}+C\int_{0}^{t}\int\left(\langle\xi-\xi_\star\rangle ^{\beta+1}\phi_{xx}^2+\langle\xi-\xi_\star\rangle ^{\beta}\phi_{x}^2\right),
	\end{aligned}
\end{equation}
\begin{equation}\nonumber
\text{the seventh term on RHS of \eqref{32}}%\leq C \int_{0}^{t}\int\left(\langle\xi-\xi_\star\rangle ^{\beta+1}\frac{\phi_{xx}^2}{U^{1-m}}+\langle\xi-\xi_\star\rangle ^{\beta+1}\frac{U_{xx}^2}{U^{3-m}}\phi_{x}^2+\langle\xi-\xi_\star\rangle ^{\beta+1} \frac{U_{x}^4}{U^{5-m}}\phi_{x}^2\right)\\&
\leq C \int_{0}^{t}\int\left(\langle\xi-\xi_\star\rangle ^{\beta+1}\frac{\phi_{xx}^2}{U^{1-m}}+\langle\xi-\xi_\star\rangle ^{\beta}\phi_{x}^2\right),	 	
\end{equation}
and
\begin{align}
	\text{the eighth term on RHS of \eqref{32}}\leq\frac{c_3}{2}\int_{0}^{t}(d_0\!+\!t)^{\beta+2} \phi_{xx}^2(0,t)\!+\!C\int_{0}^{t}(d_0\!+\!t)^{\beta+2} \phi_{x}^2(0,t),\label{36}
\end{align}
where we have ignored the negative term $- 2(1-m)\frac{|U_x|}{U^{2-m}}\langle\xi\!-\!\xi_\star\rangle ^{\beta+2}\phi_{xx}^2\bigg|_{x=0}$. Let us rewrite	
\begin{equation}\nonumber
\int_{0}^{t}|d'(\tau)|\int \langle\xi-\xi_\star\rangle ^{\beta+2}|U_{xx}||\phi_{xx}|=\int_{0}^{t}|d'(\tau)|\left(\int_{0}^{s\tau+d(\tau)}+\int_{s\tau+d(\tau)}^{+\infty} \right)\langle\xi-\xi_\star\rangle ^{\beta+2}|U_{xx}||\phi_{xx}|.
\end{equation}	
For $\xi>0$, since $U_{xx}^2U^{1-m}\sim\langle\xi-\xi_\star\rangle^{-\frac{7-5m}{ 1-m}} $ as $\xi\rightarrow+\infty$ and $\int_{s\tau+d(\tau)}^{+\infty}\langle\xi-\xi_\star\rangle^{\beta+3}U_{xx}^2U^{1-m}\leq	C$, we get from \eqref{e1} that
	\begin{align} \label{29}
		&\int_{0}^{t}|d'(\tau)|\int_{s\tau+d(\tau)}^{+\infty} \langle\xi-\xi_\star\rangle ^{\beta+2}|U_{xx}||\phi_{xx}| \notag \\
&\leq\frac{1}{2}\int_{0}^{t}\int_{s\tau+d(\tau)}^{+\infty}\langle\xi-\xi_\star\rangle^{\beta+1}\frac{\phi_{xx}^2 }{U^{1-m}}+\frac{1}{2}\int_{0}^{t}|d'(\tau)|^2\int_{s\tau+d(\tau)}^{+\infty}\langle\xi-\xi_\star\rangle^{\beta+3}U_{xx}^2U^{1-m} \notag \\
&\leq C\left(\mathrm{e}^{-2\gamma d_0}+\int_{0}^{t}\int\langle\xi-\xi_\star\rangle^{\beta+1}\frac{\phi_{xx}^2 }{U^{1-m}}+\int_{0}^{t}|\phi_{xx}^2(0,\tau)| \right).
	\end{align}
For $\xi<0$, it holds
\begin{equation}\nonumber
	\begin{aligned}
		&\int_{0}^{t}|d'(\tau)|\int_{0}^{s\tau+d(\tau)} \langle\xi-\xi_\star\rangle ^{\beta+2}|U_{xx}||\phi_{xx}|	\\&\leq \int_{0}^{t}|d'(\tau)|\left(\sup_{x\in [0,s\tau+d(\tau)]}\langle\xi-\xi_\star\rangle^\frac{\beta+\frac{3}{2}}{2} |\phi_{xx}| \right)\int_{0}^{s\tau+d(\tau)} \langle\xi-\xi_\star\rangle^\frac{\beta+\frac{1}{2}}{2}  |U_{xx}|,
	\end{aligned}
\end{equation}	
where, in view of \eqref{st+d(t)},	
\begin{equation}\nonumber
	\int_{0}^{s\tau+d(\tau)} \langle\xi-\xi_\star\rangle^\frac{\beta+\frac{1}{2}}{2}  |U_{xx}|\leq C\left(d_0+\tau \right) ^\frac{\beta+\frac{5}{2}}{2},
\end{equation}	
and
\begin{equation}\nonumber
	\begin{aligned}
		\langle\xi-\xi_\star\rangle^{\beta+\frac{3}{2}} \phi_{xx}^2&=\!-\!\int_{x}^{0}\!\left[ \frac{\partial}{\partial x}\left(\langle\xi\!-\!\xi_\star\rangle^{\beta+\frac{3}{2}} \right)\phi_{xx}^2\!+\!2\langle\xi-\xi_\star\rangle^{\beta+\frac{3}{2}}\phi_{xx}\phi_{xxx}\right]\!+\! \langle\xi\!-\!\xi_\star\rangle^{\beta+\frac{3}{2}}\big|_{x=0} \phi_{xx}^2(0,\tau)\\&\leq C\int_{0}^{s\tau+d(\tau)}\left(\langle\xi-\xi_\star\rangle^{\beta+1}\phi_{xx}^2+\langle\xi-\xi_\star\rangle^{\beta+2}\phi_{xxx}^2 \right)+ \langle\xi-\xi_\star\rangle^{\beta+\frac{3}{2}}\big|_{x=0} \phi_{xx}^2(0,\tau).
	\end{aligned}
\end{equation}	
Then, by $\langle\xi-\xi_\star\rangle^{\beta+\frac{3}{2}}\big|_{x=0}\sim (d_0+\tau)^{\beta+\frac{3}{2}}$ due to \eqref{st+d(t)}, it has
	\begin{align}
		&\int_{0}^{t}|d'(\tau)|\int_{0}^{s\tau+d(\tau)} \langle\xi-\xi_\star\rangle^{\beta+2} |U_{xx}||\phi_{xx}|%\\&\leq C\!\int_{0}^{t}\!|d'(\tau)|\!\left[\int_{0}^{s\tau+d(\tau)}\!\left(\!\langle\xi\!-\!\xi_\star\rangle^{\beta+1}\phi_{xx}^2\!+\!\langle\xi\!-\!\xi_\star\rangle^{\beta+2}\phi_{xxx}^2 \right)\!\!+\!\! \langle\xi\!-\!\xi_\star\rangle^{\beta+\frac{3}{2}}\big|_{x=0} \phi_{xx}^2(0,\tau)\right]^\frac{1}{2}\!\left(1\!+\!d_0\!+\!\tau \right) ^\frac{\beta+\frac{5}{2}}{2}
	\nonumber	\\&\leq C\int_{0}^{t}|d'(\tau)|\left[\left(\int_{0}^{s\tau+d(\tau)}\langle\xi-\xi_\star\rangle^{\beta+1}\phi_{xx}^2\right)^\frac{1}{2}\!+\!\left(\int_{0}^{s\tau+d(\tau)}\langle\xi-\xi_\star\rangle^{\beta+2}\phi_{xxx}^2\right)^\frac{1}{2}\right]\left(d_0+\tau \right) ^\frac{\beta+\frac{5}{2}}{2}\nonumber\\&\quad+C\int_{0}^{t}|d'(\tau)|\langle\xi-\xi_\star\rangle^{\frac{\beta+\frac{3}{2} }{2}}\big|_{x=0} |\phi_{xx}(0,\tau)| \left(d_0+\tau \right) ^\frac{\beta+\frac{5}{2}}{2}\nonumber	\\&\leq\frac{1}{4}\int_{0}^{t}\int_{0}^{s\tau+d(\tau)}\langle\xi-\xi_\star\rangle^{\beta+2}\phi_{xxx}^2 +C\left(\mathrm{e}^{-2\gamma d_0}+\int_{0}^{t}\int_{0}^{s\tau+d(\tau)}\langle\xi-\xi_\star\rangle^{\beta+1}\phi_{xx}^2\right.\nonumber\\&\quad\left.+\int_{0}^{t}\left(d_0+\tau \right)^{\beta+\frac{5}{2}}|\phi_{xx}^2(0,\tau)|  \right)\nonumber.
	\end{align}
Therefore, combining with \eqref{29}, we get
	\begin{align}\label{37}
		&\int_{0}^{t}|d'(\tau)|\int\langle\xi-\xi_\star\rangle^{\beta+2} |U_{xx}||\phi_{xx}| \notag \\		 &\leq\frac{1}{4}\!\int_{0}^{t}\int_{0}^{s\tau+d(\tau)}\!\!\langle\xi-\xi_\star\rangle^{\beta+2}\phi_{xxx}^2\!+\!C\left(\int_{0}^{t}\int\langle\xi\!-\!\xi_\star\rangle^{\beta+1}\frac{\phi_{xx}^2 }{U^{1-m}} \right. \notag \\
&\left.\quad+\mathrm{e}^{-2\gamma d_0}+\int_{0}^{t}\left(d_0+\tau \right) ^{\beta+\frac{5}{2}}|\phi_{x x}^2(0,\tau)|\right).
	\end{align}
In view of \eqref{Gzz} and \eqref{Fzz}, the last three terms on RHS of \eqref{eq32} can be estimated as
	\begin{align}
		&\int_{0}^{t}\int\left( \langle\xi-\xi_\star\rangle^{\beta+2}|F_{xx}\phi_{xx} |+\langle\xi-\xi_\star\rangle^{\beta+2}|G_{xx}\phi_{xxx}|+\langle\xi-\xi_\star\rangle^{\beta+1}|G_{xx}\phi_{xx}|\right) \nonumber\\
		&\leq C\int_{0}^{t}	 \int\langle\xi-\xi_\star\rangle^{\beta+2}\frac{\phi_{xx}^2|\phi_{xxx}|}{U^{2-m}}+C\int_{0}^{t}	 \int\langle\xi-\xi_\star\rangle^{\beta+1}\frac{|\phi_{xx}|^3}{U^{2-m}} +C\int_{0}^{t}	\int \bigg[\langle\xi-\xi_\star\rangle^{\beta+2} \frac{|\phi_{x}|\phi_{xxx}^2}{U^{2-m}}\nonumber\\&\quad+\langle\xi-\xi_\star\rangle^{\beta+2}|U_x||\phi_{x}|\phi_{xx}^2\!+\!\langle\xi-\xi_\star\rangle^{\beta+1} \frac{U_x}{U^{3-m}}|\phi_{x}|\phi_{xx}^2\!+\!\langle\xi-\xi_\star\rangle^{\beta+1}\left(\frac{U_{x}^2}{U^{4-m}}+\frac{|U_{xx}|}{U^{3-m}} \right)\phi_{x}^2|\phi_{xx}|\nonumber\\&\quad+\langle\xi-\xi_\star\rangle^{\beta+2}\left(U_x^2+|U_{xx}|\right)\phi_{x}^2|\phi_{xx}|+\langle\xi-\xi_\star\rangle^{\beta+2}\left(\frac{U_x^2}{U^{4-m}}+\frac{|U_{xx}|}{U^{3-m}} \right)\phi_{x}^2|\phi_{xxx}|\nonumber\\&\quad+\left(\langle\xi-\xi_\star\rangle^{\beta+2}\frac{|U_x|}{U^{3-m}}+ \frac{\langle\xi-\xi_\star\rangle^{\beta+1}}{U^{2-m}}+\langle\xi-\xi_\star\rangle^{\beta+2}\right) |\phi_{x}||\phi_{xx}||\phi_{xxx}|+ \langle\xi-\xi_\star\rangle^{\beta+2}|\phi_{xx}|^3\bigg]\nonumber\\&\triangleq I_1+I_2+I_3\nonumber.
	\end{align}
Let us express
\begin{equation}\nonumber
I_1=\int_{0}^{t}\left(\int_{0}^{s\tau+d(\tau)}+\int_{s\tau+d(\tau)}^{+\infty} \right)	 \langle\xi-\xi_\star\rangle^{\beta+2}\frac{\phi_{xx}^2|\phi_{xxx}|}{U^{2-m}},
\end{equation}
where, for $\xi<0$, by virtue of $\left\|\langle\xi-\xi_\star\rangle^{\frac{\beta+2 }{2}}\phi_{xx}(\cdot,t)\right\|_{L^{\infty}}\leq CN(T)$ and $U(0)<U(\xi)<u_-$, we have
	\begin{align}
\int_{0}^{t}\int_{0}^{s\tau+d(\tau)}\langle\xi-\xi_\star\rangle^{\beta+2}\frac{\phi_{xx}^2|\phi_{xxx}|}{U^{2-m}}&\leq CN(T)\int_{0}^{t}\int_{0}^{s\tau+d(\tau)}\langle\xi-\xi_\star\rangle^{\frac{\beta+2 }{2}}\frac{|\phi_{xx}||\phi_{xxx}|}{U^{2-m}}\nonumber\\&\leq N(T)\int_{0}^{t}\int\langle\xi-\xi_\star\rangle^{\beta+2}\frac{\phi_{xxx}^2}{U^{1-m}}+C\int_{0}^{t}\int\phi_{xx}^2,\nonumber
	\end{align}	
and since $\langle\xi-\xi_\star\rangle/U^2\sim U^{-(3-m)}$ as $\xi\rightarrow+\infty$, we deduce from $\|\phi_{xx}(\cdot,t)/U^{\frac{3-m}{2}}\|_{L^2}\leq N(T)$ and H\"{o}lder's inequality that
	 \begin{align}
&\int_{0}^{t}	 \int_{s\tau+d(\tau)}^{+\infty}\langle\xi-\xi_\star\rangle^{\beta+2}\frac{\phi_{xx}^2|\phi_{xxx}|}{U^{2-m}}\nonumber\\&\leq C\int_{0}^{t}	 \left\|\langle\xi-\xi_\star\rangle^{\frac{\beta+1}{2}}\frac{ \phi_{xx}}{U^{\frac{1-m}{2}}}\right\|_{L^{\infty}}\left(\int_{s\tau+d(\tau)}^{+\infty}\langle\xi-\xi_\star\rangle\frac{\phi_{xx} }{U^2} \right)^{\frac{1}{2}}\left( \int\langle\xi-\xi_\star\rangle^{\beta+2}\frac{\phi_{xxx}^2}{U^{1-m}}\right)^{\frac{1}{2}}\nonumber\\&\leq CN(T)\int_{0}^{t}\left(\int\langle\xi-\xi_\star\rangle^{\beta+1}\frac{\phi_{xx}^2}{U^{1-m}}+\int\langle\xi-\xi_\star\rangle^{\beta+1}\frac{\phi_{xxx}^2}{U^{1-m}} \right)^{\frac{1}{2}}\left( \int\langle\xi-\xi_\star\rangle^{\beta+2}\frac{\phi_{xxx}^2}{U^{1-m}}\right)^{\frac{1}{2}}\nonumber\\&\leq CN(T)\int_{0}^{t}\int\langle\xi-\xi_\star\rangle^{\beta+2}\frac{\phi_{xxx}^2}{U^{1-m}}+	C\int_{0}^{t}\int\langle\xi-\xi_\star\rangle^{\beta+1}\frac{\phi_{xx}^2}{U^{1-m}}. \nonumber	 
	 \end{align}
Similarly, by $\left\|\langle\xi-\xi_\star\rangle^{\frac{\beta+2 }{2}}\phi_{xx}(\cdot,t)\right\|_{L^{\infty}}\leq CN(T)$ and $\langle\xi-\xi_\star\rangle^{\frac{\beta }{2}}U^{m}\leq U^{\frac{1-m}{2}}$ as $\xi\rightarrow+\infty$, one obtains
	\begin{align}
I_2&=\int_{0}^{t}\int_{0}^{s\tau+d(\tau)}\langle\xi-\xi_\star\rangle^{\beta+1}\frac{|\phi_{xx}|^3}{U^{2-m}}	 +\int_{0}^{t}\int_{s\tau+d(\tau)}^{+\infty}\langle\xi-\xi_\star\rangle^{\beta+1}\frac{|\phi_{xx}|^3}{U^{2-m}}\nonumber\\&\leq CN(T)\left(\int_{0}^{t}\int_{0}^{s\tau+d(\tau)} \langle\xi-\xi_\star\rangle^{\frac{\beta }{2}}\phi_{xx}^2+\int_{0}^{t}\int_{s\tau+d(\tau)}^{+\infty} \langle\xi-\xi_\star\rangle^{\frac{\beta }{2}}U^m\frac{\phi_{xx}^2 }{U^2}\right)\nonumber\\&\leq C\left(\int_{0}^{t}\int \langle\xi-\xi_\star\rangle^{\beta+1}\phi_{xx}^2+\int_{0}^{t}	\int\frac{\phi_{xx}^2}{U^{2}} \right)\nonumber.
	\end{align}
Furthermore, utilizing $\left\|\langle\xi-\xi_\star\rangle^{\frac{\beta+1 }{2}}\phi_x(\cdot,t)\right\|_{L^{\infty}}\leq CN(T)$, $\left\|\langle\xi-\xi_\star\rangle^{\frac{\beta+2 }{2}}\phi_{xx}(\cdot,t)\right\|_{L^{\infty}}\leq CN(T)$, $\|\phi_{x}(\cdot,t)/U\|_{L^\infty}\leq N(T)$ and $|U_x|\leq CU^{2-m}$, leads to
	\begin{align}
		 I_3&\leq CN(T)\int_{0}^{t}\int\left[\langle\xi-\xi_\star\rangle^{\beta+2} \frac{\phi_{xxx}^2}{U^{1-m}}+
		 \langle\xi-\xi_\star\rangle^{\frac{\beta+3 }{2}}|U_x|\phi_{xx}^2+\langle\xi-\xi_\star\rangle^{\beta+1}\phi_{xx}^2\right.\nonumber\\&\left.
		\quad+\langle\xi-\xi_\star\rangle^{\beta+1}\left(\frac{U_{x}^2}{U^{3-m}}+\frac{|U_{xx}|}{U^{2-m}} \right)|\phi_{x}||\phi_{xx}|+\langle\xi-\xi_\star\rangle^{\frac{\beta+3 }{2}}\left(U_x^2+|U_{xx}|\right)|\phi_{x}||\phi_{xx}|\right.\nonumber\\&\left.
		 \quad+\langle\xi-\xi_\star\rangle^{\beta+2}\left(\frac{U_x^2}{U^{3-m}}+\frac{|U_{xx}|}{U^{2-m}} \right)|\phi_{x}||\phi_{xxx}|+\langle\xi-\xi_\star\rangle^{\frac{\beta+2 }{2}}\phi_{xx}^2\right.\nonumber\\&\left.\quad+\left(\langle\xi-\xi_\star\rangle^{\beta+2}\frac{|U_x|}{U^{2-m}}+ \frac{\langle\xi-\xi_\star\rangle^{\beta+1}}{U^{1-m}}+\langle\xi-\xi_\star\rangle^{\frac{\beta+3 }{2}}\right) |\phi_{xx}||\phi_{xxx}| \right]\nonumber\\&\leq CN(T)\int_{0}^{t}\int\langle\xi-\xi_\star\rangle^{\beta+2} \frac{\phi_{xxx}^2}{U^{1-m}}+C\int_{0}^{t}\int\left[
		 \langle\xi-\xi_\star\rangle^{\frac{\beta+3 }{2}}|U_x|\phi_{xx}^2+\langle\xi-\xi_\star\rangle^{\beta+1}\phi_{xx}^2\right.\nonumber\\&\left.\quad+\langle\xi-\xi_\star\rangle^{\beta+1}\left(\frac{U_{x}^4}{U^{6-2m}}+\frac{U_{xx}^2}{U^{4-2m}} \right)\phi_{x}^2+\langle\xi-\xi_\star\rangle^{2} \left(U_x^2+|U_{xx}|\right)^2\phi_{x}^2\right.\nonumber\\&\left.\quad+ \langle\xi-\xi_\star\rangle^{\beta+2}\left(\frac{U_x^4}{U^{5-m}}+\frac{U_{xx}^2}{U^{3-m}} \right)\phi_{x}^2+\langle\xi-\xi_\star\rangle^{\beta+1}\phi_{xx}^2\right.\nonumber\\&\left.\quad+\left(\langle\xi-\xi_\star\rangle^{\beta+2}\frac{U_x^2}{U^{3-m}}+\frac{\langle\xi-\xi_\star\rangle^{\beta+1}}{U^{1-m}} \right) \phi_{xx}^2\right]\nonumber\\
&\leq CN(T)\int_{0}^{t}\int\langle\xi-\xi_\star\rangle^{\beta+2} \frac{\phi_{xxx}^2}{U^{1-m}} \nonumber \\
&\ \ \ +C\left[\int_{0}^{t}\int \langle\xi-\xi_\star\rangle^{\beta}\phi_{x}^2+\int_{0}^{t}\int\langle\xi-\xi_\star\rangle^{\beta+1}\frac{\phi_{xx}^2}{U^{1-m}}\right],\label{40}
	\end{align}
where we have use \eqref{e5}-\eqref{e6} and \eqref{e8}-\eqref{e9} in the last inequality. Therefore, combining \eqref{33}-\eqref{34}, \eqref{35}-\eqref{36}, \eqref{37}-\eqref{40} and \eqref{32}, noting \eqref{phix 0} and \eqref{e7}, we get
\begin{equation}\nonumber
	\begin{aligned}
		&\frac{1}{2}\int\langle\xi-\xi_\star\rangle ^{\beta+2} \phi_{xx}^2+\frac{c_3}{2}\int_{0}^{t}(d_0+\tau)^{\beta+2}\phi^2_{xx}(0,\tau)+\left(\frac{1}{2}-CN(T) \right)\int_{0}^{t}\int\langle\xi-\xi_\star\rangle ^{\beta+2}\frac{\phi_{xxx}^2}{U^{1-m}}\\&\leq C\left(\int\langle\xi_0-\xi_\star\rangle ^{\beta+2} \phi_{0xx}^2+d_0^{-1} +\int_{0}^{t}\left(d_0+\tau \right) ^{\beta+\frac{5}{2}}|\phi_{x x}^2(0,\tau)|\right.\\&\left.\quad+ \int_{0}^{t}\int\langle\xi-\xi_\star\rangle ^{\beta+1}\frac{\phi_{xx}^2}{U^{1-m}}+\int_{0}^{t}\int\langle\xi-\xi_\star\rangle ^{\beta}\frac{\phi_{x}^2}{U^{1-m}}+\int_{0}^{t}\int\frac{\phi_{xx}^2}{U^2}\right)%\\&\leq C\left(\int\langle\xi-\xi_\star\rangle^\beta\phi_0^2+\int\langle\xi-\xi_\star\rangle ^{\beta+1} \phi_{0x}^2+\int\langle\xi-\xi_\star\rangle ^{\beta+2} \phi_{0xx}^2+\int\frac{\phi_0^2 }{U^{3m-1}}\right.\\&\left.\quad+\int\frac{\phi_{0x}^2}{U^{1+m}}+d_0^{-1}+C\int_{0}^{t}\left(d_0+\tau \right) ^{\beta+\frac{5}{2}}|\phi_{x x}^2(0,\tau)|\right),
\end{aligned}
\end{equation}
Taking $CN(T)\leq1/4$ and utilizing \eqref{L2-2 estimate}, \eqref{H1 estimate} and \eqref{H1-2 estimate}, we then get \eqref{H2-2 estimate}.
\end{proof}

\subsubsection{Estimates for third order derivative}

\begin{lemma}\label{H3}
	Let the assumptions of Proposition \ref{proposition priori estimate} hold. Assume that $0<\beta \leq \frac{3m-1}{1-m}$, if $N(T)+d_0^{-\frac{1}{2}}\ll1$, then there exists a constant $C>0$ independent of $T$ such that
	\begin{equation}\label{H3 estimate}
		\begin{aligned}
			&\|\phi_{xxx}(\cdot,t)\|_{\langle x-st-d(t)\rangle ^{\beta+3}}^2+(d_0+t)^{\beta+3}|\phi_{xx}(0,t)|^2\\&+\int_{0}^{t}\left[ (d_0+\tau)^{\beta+3}|\phi_{xxx}(0,\tau)|^2+ \left\|\phi_{xxxx}(\cdot,\tau)/U^{\frac{1-m}{2}}\right\|_{\langle x-st-d(t)\rangle ^{\beta+3}}^2 \right]\\& \leq C(1+d_0^{\beta+3})\left(\|\phi_0\|_{3,\langle x-d_0\rangle ^{\beta}}^2 +\|\phi_0\|_{w_{1,0}}^2+\|\phi_{0x}\|_{w_{2,0}}^2+\|\phi_{0xx}\|_{w_{3,0}}^2 \right)+Cd_0^{-1},
		\end{aligned}
	\end{equation}	
where  $w_{1,0}=U^{1-3m}$,  $w_{2,0}=U^{-(1+m)}$, $w_{3,0}=U^{m-3}$ for $U=U(x-d_0)$.
\end{lemma}	

\begin{proof}
We take a smooth cut-off function $\eta(x)$ satisfying
\begin{equation}\label{eta}
\eta(x)=1 \text{ for }x<L,~\eta(x)=0 \text{ for }x>2L,~0\leq \eta(x)\leq 1 \text{ and } |\eta_x(x)|\leq \frac{C}{L}\text{ for }x\in\mathbb{R_+}.
\end{equation}
Differentiating \eqref{phizzequ} in $x$ and multiplying the result by $\langle\xi-\xi_\star\rangle ^{\beta+3}\phi_{xxx}\eta^2(x)$,
%\begin{equation}\nonumber
%	\langle\xi-\xi_\star\rangle ^{\beta+3}\phi_{xxx}\phi_{xxxt}=\frac{1}{2}\left(\langle\xi-\xi_\star\rangle ^{\beta+3} \phi_{xxx}^2\right)_t-\frac{\beta+3}{2}\langle\xi-\xi_\star\rangle ^{\beta+1}\left(\xi-\xi_\star \right)(-s-d'(t))\phi_{xxx}^2,
%\end{equation}
%\begin{equation}\nonumber
%	\begin{aligned}
%		\langle\xi-\xi_\star\rangle ^{\beta+3}f'(U)\phi_{xxx}\phi_{xxxx}&=\left(\frac{\langle\xi-\xi_\star\rangle ^{\beta+3}}{2} f'(U)\phi_{xxx}^2\right)_x\\&\quad-\frac{\beta+3}{2}\langle\xi-\xi_\star\rangle ^{\beta+1}(\xi-\xi_\star)f'(U)\phi_{xxx}^2-\frac{\langle\xi-\xi_\star\rangle ^{\beta+3}}{2} f''(U)U_x\phi_{xxx}^2,
%	\end{aligned}
%\end{equation}
%and
%\begin{equation}\nonumber
%	\begin{aligned}
%		-\langle\xi-\xi_\star\rangle ^{\beta+3}\phi_{xxx}\left(\frac{\phi_{xxxx}}{U^{1-m}}\right)_{x}&=-\left[ \frac{\phi_{xxxx}}{U^{1-m}}\langle\xi-\xi_\star\rangle ^{\beta+3}\phi_{xxx}\right]_x\\&\quad+\langle\xi-\xi_\star\rangle ^{\beta+3}\frac{\phi_{xxxx}^2}{U^{1-m}}+(\beta+3)\langle\xi-\xi_\star\rangle ^{\beta+1}(\xi-\xi_\star)\frac{\phi_{xxx}\phi_{xxxx}}{U^{1-m}} ,
%	\end{aligned}
%\end{equation}
we have
	\begin{align}
		&\frac{1}{2}\int_{0}^{2L} \langle\xi-\xi_\star\rangle ^{\beta+3}\phi_{xxx}^2\eta^2+\int_{0}^{t}\int_{0}^{2L}\langle\xi-\xi_\star\rangle ^{\beta+3}\frac{\phi_{xxxx}^2}{U^{1-m}}\eta^2\nonumber\\&\quad-\int_{0}^{t}\left[\left(\frac{f'(U)\phi_{xxx}}{2}- \frac{\phi_{xxxx}}{U^{1-m}}-\frac{G_{xxx}}{m}\right)\langle\xi-\xi_\star\rangle ^{\beta+3}\phi_{xxx}\eta^2 \right]\bigg|_{x=0}\nonumber\\&=\frac{1}{2}\int_0^{2L} \langle\xi_0-\xi_\star\rangle ^{\beta+3}\phi_{0xxx}^2\eta^2-(\beta+3)\int_{0}^{t}\int_{0}^{2L}\langle\xi-\xi_\star\rangle ^{\beta+1}(\xi-\xi_\star)\frac{\phi_{xxx}\phi_{xxxx} }{U^{1-m}}\eta^2\nonumber\\&\quad\!-\!(1\!-\!m)\!\!\int_{0}^{t}\!\!\int_{0}^{2L}\!\!\langle\xi\!-\!\xi_\star\rangle ^{\beta+3}\frac{U_x}{U^{2-m} }\phi_{xxx}\phi_{xxxx}\eta^2\!+\!\frac{\beta\!+\!3}{2}\!\int_{0}^{t}\!\!\int_{0}^{2L}\!\langle\xi\!-\!\xi_\star\rangle ^{\beta+1}(\xi\!-\!\xi_\star) \left( g'\!-\!d'\right)\phi_{xxx}^2\eta^2\nonumber\\&\quad+\int_{0}^{t}\int_{0}^{2L}\langle\xi-\xi_\star\rangle ^{\beta+3}\left[ (1-m)\left((2-m)\frac{U_x^2}{U^{3-m} }-\frac{U_{xx}}{U^{2-m} } \right)-\frac{5}{2}f''(U)U_x\right]\phi_{xxx}^2\eta^2\nonumber\\&\quad+\int_{0}^{t}\int_{0}^{2L}\langle\xi-\xi_\star\rangle ^{\beta+3}\left[ -3\left(f'''U_x^2+f''U_{xx} \right)\phi_{xx}\phi_{xxx}-f^{(4)}(U)U_x^3\phi_{x}\phi_{xxx}\right]\eta^2\nonumber\\&\quad-\int_{0}^{t}\int_{0}^{2L}\langle\xi\!-\!\xi_\star\rangle ^{\beta+3}\left(3f'''U_xU_{xx}\!+\!f''U_{xxx} \right)\phi_{x}\phi_{xxx}\eta^2+(1-m)\int_{0}^{t}\int_{0}^{2L}\left[-\frac{2U_x}{U^{2-m}}\phi_{xxx}\right.\nonumber\\&\left.\quad\!-\!\frac{3U_{xx}}{U^{2-m}}\phi_{xx}\!-\!\frac{U_{xxx}}{U^{2-m}}\phi_{x}\!+\!(2\!-\!m)\left( \frac{3U_x^2}{U^{3-m}}\phi_{xx}\!+\!\frac{3U_xU_{xx}}{U^{3-m}}\phi_{x}\!-\!\frac{(3\!-\!m)U_{x}^3}{U^{4-m}}\phi_{x}\right)\right]_x\!\!\langle\xi\!-\!\xi_\star\rangle ^{\beta+3}\phi_{xxx}\eta^2\nonumber\\&\quad+\int_{0}^{t}d'\langle\xi-\xi_\star\rangle ^{\beta+3}\int_{0}^{2L}U_{xxx}\phi_{xxx}\eta^2	 +\int_{0}^{t}\int_{0}^{2L}\langle\xi-\xi_\star\rangle ^{\beta+3}F_{xxx}\phi_{xxx}\eta^2 \nonumber\\&\quad-\frac{1}{m}\int_{0}^{t}\int_{0}^{2L}\langle\xi-\xi_\star\rangle ^{\beta+3}G_{xxx}\phi_{xxxx}\eta^2-\frac{\beta+3}{m}\int_{0}^{t}\int_{0}^{2L}\langle\xi-\xi_\star\rangle ^{\beta}(\xi-\xi_\star)G_{xxx}\phi_{xxx}\eta^2\nonumber\\&\quad+2\int_{0}^{t}\int_{0}^{2L}\left(\frac{f'(U)\phi_{xxx}}{2}- \frac{\phi_{xxxx}}{U^{1-m}}-\frac{G_{xxx}}{m}\right)\langle\xi-\xi_\star\rangle ^{\beta+3}\phi_{xxx}\eta\eta_x.\label{41}
	\end{align}
Notice that
	 \begin{align}
&\langle\xi\!-\!\xi_\star\rangle ^{\beta+3}\big|_{x=0}
\frac{f'(u_-) }{2}\frac{\mathrm{d}}{\mathrm{d}t}\phi_{xx}^2(0,t)\nonumber\\&=\frac{f'(u_-) }{2}\frac{\mathrm{d}}{\mathrm{d}t}\left(\langle\xi-\xi_\star\rangle ^{\beta+3}\phi_{xx}^2\right)\bigg|_{x=0}-(\beta+3)\langle\xi-\xi_\star\rangle ^{\beta+1} (\xi-\xi_\star)(-s-d')
\frac{f'(u_-) }{2}\phi_{xx}^2\bigg|_{x=0}\nonumber\\&\geq \frac{f'(u_-) }{2}\frac{\mathrm{d}}{\mathrm{d}t}\left(\langle\xi-\xi_\star\rangle ^{\beta+3}\phi_{xx}^2\!\right)\!\bigg|_{x=0}-C  d_0^{-1}\langle\xi-\xi_\star\rangle ^{\beta+3}\big|_{x=0}\phi_{xx}^2(0,t),\nonumber
\end{align}
for $\langle\xi\!-\!\xi_\star\rangle^{-1}\big|_{x=0}\leq d_0^{-1}$ in the last inequality, and
	\begin{align}
&\frac{\langle\xi\!-\!\xi_\star\rangle ^{\beta+3}}{U^{1-m}}\frac{\mathrm{d}}{\mathrm{d}t}\left[\phi_{xx}\left( u_-^{1-m}f'(u_-)U_{x} \!-\!U_{xx}\!-\!\frac{ (m\!-\!1)}{u_-}\left(U_{x}\!+\!\phi_{xx}\right)^2\!\right)\!\right]\!\bigg|_{x=0}\nonumber\\&=\frac{\mathrm{d}}{\mathrm{d}t}\left[\frac{\langle\xi\!-\!\xi_\star\rangle ^{\beta+3}}{U^{1-m}}\phi_{xx}\left( u_-^{1-m}f'(u_-)U_{x} \!-\!U_{xx}\!-\!\frac{ (m\!-\!1)}{u_-}\left(U_{x}\!+\!\phi_{xx}\right)^2\!\right)\!\right]\!\bigg|_{x=0}\nonumber\\&\quad-\!(\beta+3)\langle\xi\!-\!\xi_\star\rangle ^{\beta+1} (\xi\!-\!\xi_\star)(\!-\!s\!-\!d')\frac{\phi_{xx} }{ U^{1-m}}\left( u_-^{1-m}f'(u_-)U_{x} \!-\!U_{xx}\!-\!\frac{ (m\!-\!1)}{u_-}\left(U_{x}\!+\!\phi_{xx}\right)^2\!\right)\bigg|_{x=0}\nonumber\\&\quad+(1-m)\langle\xi\!-\!\xi_\star\rangle ^{\beta+3}\frac{U_x}{U^{2-m}}\phi_{xx}\left( u_-^{1-m}f'(u_-)U_{x} \!-\!U_{xx}\!-\!\frac{ (m\!-\!1)}{u_-}\left(U_{x}\!+\!\phi_{xx}\right)^2\!\right)\!\bigg|_{x=0},\nonumber\\&
\leq \frac{\mathrm{d}}{\mathrm{d}t}\left[\frac{\langle\xi\!-\!\xi_\star\rangle ^{\beta+3}}{U^{1-m}}\phi_{xx}\left( u_-^{1-m}f'(u_-)U_{x} \!-\!U_{xx}\!-\!\frac{ (m\!-\!1)}{u_-}\left(U_{x}\!+\!\phi_{xx}\right)^2\!\right)\!\right]\!\bigg|_{x=0}\nonumber\\&\quad +C\left( d_0^{-1}+N(T)\right)\langle\xi\!-\!\xi_\star\rangle ^{\beta+3}\big|_{x=0}\phi_{xx}^2(0,t)+C\langle\xi\!-\!\xi_\star\rangle ^{\beta+3}\left(U_x^2+U_{xx}^2 \right)\big|_{x=0},	\nonumber	
	\end{align}
it then follows from \eqref{Gzz}, \eqref{boundary phixxxx}, $\left\|\phi_x/U\left(\cdot,t\right)\right\|_{L^{\infty}}\leq CN(T)$ and $|\phi_{xx}(0,t)|\leq CN(T)$ that
	\begin{align}
		&-\left[\left(\frac{f'(U)\phi_{xxx}}{2}- \frac{\phi_{xxxx}}{U^{1-m}}-\frac{G_{xxx}}{m}\right)\langle\xi-\xi_\star\rangle ^{\beta+3}\phi_{xxx}\eta^2 \!\right]\bigg|_{x=0}\nonumber\\
&\geq\!\langle\xi\!-\!\xi_\star\rangle ^{\beta+3}\big|_{x=0}\!\left[\! \left(\!-\frac{f'(u_-)\phi^2_{xxx}}{2}\!+\!(1\!-\!CN(T)) \frac{\phi_{xxx}\phi_{xxxx}}{U^{1-m}}\!\right)\!\bigg|_{x=0}\!\!-\!CN(T)\left(\phi^2_{x}\!+\! \phi^2_{xx}\!+\!\phi^2_{xxx}\right)\!\big|_{x=0}\right]\nonumber\\
&\geq\!\left(\!\frac{f'(u_-) }{8}\!\!-\!CN(T)\right)\!\frac{\mathrm{d}}{\mathrm{d}t}\!\left(\!\langle\xi\!-\!\xi_\star\rangle ^{\beta+3}\phi_{xx}^2\!\right)\!\bigg|_{x=0}\!+\!\langle\xi\!-\!\xi_\star\rangle ^{\beta+3}\big|_{x=0}\!\left(\!\frac{f'(u_-)}{4}\!-\!CN(T) \right)\phi^2_{xxx}(0,t)\nonumber\\
&\quad+\frac{\mathrm{d}}{\mathrm{d}t}\left[\frac{\langle\xi\!-\!\xi_\star\rangle ^{\beta+3}}{U^{1-m}}\phi_{xx}\left( u_-^{1-m}f'(u_-)U_{x} \!-\!U_{xx}\!-\!\frac{ (m\!-\!1)}{u_-}\left(U_{x}\!+\!\phi_{xx}\right)^2\!\right)\!\right]\!\bigg|_{x=0}\nonumber\\
&\quad-C\langle\xi\!-\!\xi_\star\rangle ^{\beta+3}\!\left(U_x^2+U_{xx}^2 +U_{xxx}^2+\phi^2_{x} \right)\big|_{x=0} \notag \\
&\quad -C(d_0^{-\frac{1}{2}}+N(T))\langle\xi\!-\!\xi_\star\rangle ^{\beta+3}\big|_{x=0} \phi^2_{xx}(0,t).\label{47}
	\end{align}	
%Furthermore,
%\begin{equation}
%	\begin{aligned}
%		&\quad\!-\!\left[\left( -2(1-m)\frac{U_x}{U^{2-m}}\phi_{xxx}\!-\!3(1-m)\frac{U_{xx}}{U^{2-m}}\phi_{xx}\!+\!3(1-m)(2\!-\!m)\frac{U_x^2}{U^{3-m}}\phi_{xx}\!-\!(1\!-\!m)\frac{U_{xxx}}{U^{2-m}}\phi_{x}\!\right.\right.\\&\left.\left.\quad+3(1-m)(2-m)\frac{U_xU_{xx}}{U^{3-m}}\phi_{x}-(1-m)(2-m)(3-m)\frac{U_{x}^3}{U^{4-m}}\phi_{x}\right)\langle\xi-\xi_\star\rangle ^{\beta+3}\phi_{xxx}\eta^2\right]\bigg|_{x=0}\\&\leq\langle\xi\!-\!\xi_\star\rangle ^{\beta+3}\big|_{x=0}\frac{\phi^2_{xxx}(0,t) }{8}+C\langle\xi\!-\!\xi_\star\rangle ^{\beta+3}\big|_{x=0}\!\left(U_x^2+U_{xx}^2 +U_{xxx}^2 \right),
%	\end{aligned}
%\end{equation}
%where we have ignored the negative term $ -2(1-m)\langle\xi\!-\!\xi_\star\rangle ^{\beta+3}\big|_{x=0}\frac{|U_x|}{U^{2-m}}\phi_{xxx}^2(0,t)$.
For the last term on RHS of \eqref{41}, a direct calculation gives
	\begin{align}
&\bigg|2\left(\frac{f'(U)\phi_{xxx}}{2}- \frac{\phi_{xxxx}}{U^{1-m}}-\frac{G_{xxx}}{m}\right)\bigg|\nonumber	\\&\leq C\left[|\phi_{xxx}|+\frac{|\phi_{xxxx}|}{U^{1-m}}+\left(\frac{|U_x|^3 }{U^{5-m}}\!+\!\frac{|U_x||U_{xx}| }{U^{4-m}}\!+\!\frac{|U_{xxx}| }{U^{3-m}} \right)\phi_{x}^2\!+\!\left(\frac{U_x^2 }{U^{4-m}}\!+\!\frac{|U_{xx}| }{U^{3-m}} \right)|\phi_{x}||\phi_{xx}|\right.\nonumber\\&\left.\quad+\frac{|U_x| }{U^{3-m}}|\phi_{x}||\phi_{xxx}|+\frac{|\phi_{x}||\phi_{xxxx}| }{U^{2-m}}+\frac{U_x }{U^{3-m}}\phi_{xx}^2+\frac{|\phi_{xx}|^3}{U^{3-m}}+\frac{|\phi_{xx}||\phi_{xxx}|}{U^{2-m}}\right]\nonumber.	
	\end{align}
Then, in view of \eqref{Gzzz} and $\|\phi_x(\cdot,t)/U\|_{L^\infty}\leq CN(T)\leq C$, one has
	\begin{align}
		&\bigg|2\int_{0}^{t}\int_{0}^{2L}\left(\frac{f'(U)\phi_{xxx}}{2}- \frac{\phi_{xxxx}}{U^{1-m}}-\frac{G_{xxx}}{m}\right)\langle\xi-\xi_\star\rangle ^{\beta+3}\phi_{xxx}\eta\eta_x\bigg|\nonumber \\&%\leq C\!\int_{0}^{t}\int_{0}^{2L}\!\langle\xi\!-\!\xi_\star\rangle ^{\beta+3}\phi_{xxx}^2\eta|\eta_x|+C\int_{0}^{t}\int_{0}^{2L}\langle\xi-\xi_\star\rangle ^{\beta+3}\frac{|\phi_{xxx}||\phi_{xxxx}|}{U^{1-m}}\eta|\eta_x|\nonumber\\&\quad+ C\!\int_{0}^{t}\int_{0}^{2L}\!\langle\xi\!-\!\xi_\star\rangle ^{\beta+3}|\phi_{xxx}|\!\left[\left(\frac{|U_x|^3 }{U^{4-m}}\!+\!\frac{|U_x||U_{xx}| }{U^{3-m}}\!+\!\frac{|U_{xxx}| }{U^{2-m}} \right)|\phi_{x}|\!+\!\left(\frac{U_x^2 }{U^{3-m}}\!+\!\frac{|U_{xx}| }{U^{2-m}} \right)|\phi_{xx}|\right.\nonumber\\&\left.\quad+\frac{|U_x| }{U^{2-m}}|\phi_{xxx}|+\frac{|\phi_{xxxx}| }{U^{1-m}}\right]|\eta||\eta_x|\cdot \frac{|\phi_x|}{U}+C\int_{0}^{t}\int_{0}^{2L}\langle\xi-\xi_\star\rangle ^{\beta+3}\frac{|U_x| }{U^{3-m}}\phi_{xx}^2|\phi_{xxx}||\eta||\eta_x|\nonumber\\&\quad+C\int_{0}^{t}\int_{0}^{2L}\langle\xi-\xi_\star\rangle ^{\beta+3}\frac{|\phi_{xx}|^3|\phi_{xxx}|}{U^{3-m}}|\eta||\eta_x|+C\int_{0}^{t}\int_{0}^{2L}\langle\xi-\xi_\star\rangle ^{\beta+3}\frac{|\phi_{xx}|\phi_{xxx}^2}{U^{2-m}}|\eta||\eta_x|\nonumber\\&
		\leq C\!\int_{0}^{t}\int_{0}^{2L}\!\langle\xi\!-\!\xi_\star\rangle ^{\beta+3}\phi_{xxx}^2\eta|\eta_x|+C\int_{0}^{t}\int_{0}^{2L}\langle\xi-\xi_\star\rangle ^{\beta+3}\frac{|\phi_{xxx}||\phi_{xxxx}|}{U^{1-m}}\eta|\eta_x|\nonumber\\&\quad+\! C\!\!\int_{0}^{t}\!\!\int_{0}^{2L}\!\!\langle\xi\!-\!\xi_\star\rangle ^{\beta+3}|\phi_{xxx}|\!\left[\!\left(\!\frac{|U_x|^3 }{U^{4-m}}\!\!+\!\frac{|U_x||U_{xx}| }{U^{3-m}}\!+\!\frac{|U_{xxx}| }{U^{2-m}}\! \right)\!|\phi_{x}|\!+\!\left(\!\frac{U_x^2 }{U^{3-m}}\!+\!\frac{|U_{xx}| }{U^{2-m}} \!\right)\!|\phi_{xx}|\right]\!|\eta||\eta_x|\nonumber\\&\quad+C\int_{0}^{t}\int_{0}^{2L}\langle\xi-\xi_\star\rangle ^{\beta+3}\frac{|U_x| }{U^{3-m}}\phi_{xx}^2|\phi_{xxx}||\eta||\eta_x|+C\int_{0}^{t}\int_{0}^{2L}\langle\xi-\xi_\star\rangle ^{\beta+3}\frac{|\phi_{xx}|^3|\phi_{xxx}|}{U^{3-m}}|\eta||\eta_x|\nonumber\\&\quad+C\int_{0}^{t}\int_{0}^{2L}\langle\xi-\xi_\star\rangle ^{\beta+3}\frac{|\phi_{xx}|\phi_{xxx}^2}{U^{2-m}}|\eta||\eta_x|\nonumber\\&\triangleq B_7+\cdots+B_{12}.\label{51}
	\end{align}
Since $\phi\in X(0,T)$, we know for any $t\in[0,T]$, that
\begin{equation}\nonumber
\int_{0}^{2L}\langle\xi-\xi_\star\rangle^{\beta+3}\phi_{xxx}^2\leq\int\langle\xi-\xi_\star\rangle^{\beta+3}\phi_{xxx}^2	 \leq N^2(T),
\end{equation}
and
\begin{equation}\nonumber
\int_{0}^{2L}\left(\langle\xi-\xi_\star\rangle ^{\beta+2}\phi_{xx}^2\!+\!\langle\xi\!-\!\xi_\star\rangle ^{\beta+1}\phi_{x}^2\right)	\leq\int\left(\langle\xi-\xi_\star\rangle ^{\beta+2}\phi_{xx}^2\!+\!\langle\xi\!-\!\xi_\star\rangle ^{\beta+1}\phi_{x}^2\right) \leq N^2(T).
\end{equation}
Thus, utilizing \eqref{eta}, we deduce that
\begin{equation}\label{48}
B_7\leq\frac{C}{L}\int_{0}^{t}\int_{0}^{2L}\langle\xi-\xi_\star\rangle ^{\beta+3}\phi_{xxx}^2\leq\frac{CTN^2(T)}{L},
\end{equation}
and
\begin{equation}\label{49}
B_8\leq	\int_{0}^{t}\int_{0}^{2L}\langle\xi-\xi_\star\rangle ^{\beta+3}\phi_{xxx}^2\frac{|\eta_x|^{\frac{3}{2}}}{U^{1-m}}+\int_{0}^{t}\int_{0}^{2L}\langle\xi-\xi_\star\rangle ^{\beta+3}\frac{\phi_{xxxx}^2}{U^{1-m}}\eta^2|\eta_x|^{\frac{1}{2}}.
\end{equation}
Thanks to $U^{1-m}(\xi)\sim |\xi|^{-1}$ as $\xi\rightarrow+\infty$ due to  \eqref{f''>0}, this implies $\frac{|\eta_x|^{\frac{3}{2}}}{U^{1-m}}\leq \frac{C}{\sqrt{L}}$. Thus,
\begin{equation}\nonumber
\int_{0}^{t}\int_{0}^{2L}\langle\xi-\xi_\star\rangle ^{\beta+3}\phi_{xxx}^2\frac{|\eta_x|^{\frac{3}{2}}}{U^{1-m}}\leq\frac{C}{\sqrt{L}}\int_{0}^{t}\int_{0}^{2L}\langle\xi-\xi_\star\rangle ^{\beta+3}\phi_{xxx}^2\leq\frac{CTN^2(T)}{\sqrt{L}},
\end{equation}
and
\begin{equation}\nonumber
\int_{0}^{t}\int_{0}^{2L}\langle\xi-\xi_\star\rangle ^{\beta+3}\frac{\phi_{xxxx}^2}{U^{1-m}}\eta^2|\eta_x|^{\frac{1}{2}}\leq\frac{C}{\sqrt{L}}\int_{0}^{t}\int_{0}^{2L}\langle\xi-\xi_\star\rangle ^{\beta+3}\frac{\phi_{xxxx}^2}{U^{1-m}}\eta^2.	
\end{equation}
Therefore, \eqref{49} is transformed into
\begin{equation}\label{50}
B_8\leq\frac{C}{\sqrt{L}}\int_{0}^{t}\int_{0}^{2L}\langle\xi-\xi_\star\rangle ^{\beta+3}\frac{\phi_{xxxx}^2}{U^{1-m}}\eta^2+\frac{CTN^2(T)}{\sqrt{L}}.
\end{equation}
Moreover, by \eqref{e5}-\eqref{e6}, \eqref{e8}-\eqref{e9} and Cauchy-Schwarz inequality, we get
	 \begin{align}
B_9&\leq C	 \int_{0}^{t}\int_{0}^{2L}\langle\xi\!-\!\xi_\star\rangle ^{\beta+3}\phi_{xxx}^2\frac{|\eta_x|^{\frac{3}{2}}}{U^{1-m}}\!+\!C\int_{0}^{t}\int_{0}^{2L}\langle\xi\!-\!\xi_\star\rangle ^{\beta+3}\left[\left(\frac{U_x^6}{U^{7-m}}\!+\!\frac{U_x^2U_{xx}^2 }{U^{5-m}}\!+\!\frac{U_{xxx}^2 }{U^{3-m}} \right)\phi_{x}^2\right.\nonumber\\&\left.\quad\!+\!\left(\frac{U_x^4 }{U^{5-m}}\!+\!\frac{U_{xx}^2 }{U^{3-m}} \right)\phi_{xx}^2\right]\eta^2|\eta_x|^{\frac{1}{2}}\nonumber\\&\leq\frac{C }{\sqrt{L}}\int_{0}^{t}\int_{0}^{2L}\left(\langle\xi-\xi_\star\rangle ^{\beta+3}\phi_{xxx}^2+\langle\xi-\xi_\star\rangle ^{\beta+2}\phi_{xx}^2+\langle\xi-\xi_\star\rangle ^{\beta+1}\phi_{x}^2 \right)\nonumber\\&\leq\frac{C TN^2(T)}{\sqrt{L}},\label{52}\nonumber
	 \end{align}
and
\begin{equation}\nonumber
	 \begin{aligned}
B_{10}&\leq\frac{C}{L} \int_{0}^{t}\left\|\langle\xi-\xi_\star\rangle ^{\frac{\beta+2 }{2}}\phi_{xx}\right\|_{L^{\infty}}\left\|\langle\xi-\xi_\star\rangle ^{\frac{1 }{2}}\frac{|U_x|}{U^{3-m}}\phi_{xx}\right\|_{L^{2}}\left\|\langle\xi-\xi_\star\rangle ^{\frac{\beta+3 }{2}}\phi_{xxx}\right\|_{L^{2}}\\&%\leq\frac{CN(T)}{L} \int_{0}^{t}\left\|\langle\xi-\xi_\star\rangle ^{\frac{\beta+2 }{2}}\phi_{xx}\right\|_{L^{\infty}}\left\|\langle\xi-\xi_\star\rangle ^{\frac{\beta+3 }{2}}\phi_{xxx}\right\|_{L^{2}}\\&%\leq\frac{CN(T)}{L} \int_{0}^{t}\left(\left\|\langle\xi-\xi_\star\rangle ^{\frac{\beta+2 }{2}}\phi_{xx}\right\|_{L^{2}}^2+\left\|\langle\xi-\xi_\star\rangle ^{\frac{\beta+2 }{2}}\phi_{xxx}\right\|_{L^{2}}^2 \right)^{\frac{1}{2}}\left\|\langle\xi-\xi_\star\rangle ^{\frac{\beta+3 }{2}}\phi_{xxx}\right\|_{L^{2}}\\&
\leq\frac{CN(T)}{L} \int_{0}^{t}\int\left(\langle\xi-\xi_\star\rangle ^{\beta+3}\phi_{xxx}^2+\langle\xi-\xi_\star\rangle ^{\beta+2}\phi_{xx}^2 \right)\\&\leq\frac{CTN^3(T)}{L},	
	 \end{aligned}
\end{equation}
where we have used $\langle\xi-\xi_\star
\rangle\frac{U_x^2}{U^{6-2m}}\leq CU^{-(3-m)}$ as $\xi\rightarrow+\infty$, $|U_x(\xi)|\sim \mathrm{e}^{-\lambda_-|\xi|}$ as $\xi\rightarrow-\infty$ and $\|\phi_{xx}(\cdot,t)/U^{\frac{3-m}{2}}\|_{L^2}\leq CN(T)$ in the second inequality. Noting that $\|\langle\xi-\xi_\star\rangle ^{\frac{\beta+2 }{2}}\phi_{xx} \|_{L^{\infty}}\leq N(T)$,
\begin{equation}\nonumber
	\begin{aligned}
B_{11}&%\leq CN(T)\int_{0}^{t}\int_{0}^{2L}\langle\xi-\xi_\star\rangle ^{\frac{ \beta+4}{2}}\frac{\phi_{xx}^2|\phi_{xxx}|}{U^{3-m}}|\eta_x|\\&
\leq CN(T)\int_{0}^{t}\left(\int_{0}^{s\tau+d(\tau)}+\int_{s\tau+d(\tau)}^{2L} \right)\langle\xi-\xi_\star\rangle ^{\frac{ \beta+4}{2}}\frac{\phi_{xx}^2|\phi_{xxx}|}{U^{3-m}}|\eta_x|.
	\end{aligned}
\end{equation}
For $\xi>0$, in view of $\langle\xi-\xi_\star\rangle\sim U^{-(1-m)}$, $\langle\xi-\xi_\star\rangle|\eta_x|\leq C$ and $\langle\xi-\xi_\star\rangle ^{\beta+2}\leq\langle\xi-\xi_\star\rangle ^{\frac{1+m}{1-m}}\leq U^{-(1+m)}\leq CU^{-(4-2m)}$ as $\xi\rightarrow+\infty$, it holds that
	\begin{align}
&\quad CN(T)\int_{0}^{t}\int_{s\tau+d(\tau)}^{2L}\langle\xi-\xi_\star\rangle ^{\frac{ \beta+4}{2}}\frac{\phi_{xx}^2|\phi_{xxx}|}{U^{3-m}}|\eta_x|		 \nonumber\\&\leq\frac{CN(T)}{\sqrt{L}} \int_{0}^{t}\left\|\langle\xi-\xi_\star\rangle ^{\frac{\beta+2 }{2}}\phi_{xx}\right\|_{L^{\infty}}\left(\int_{s\tau+d(\tau)}^{2L}\langle\xi-\xi_\star\rangle^2|\eta_x|\frac{\phi_{xx}^2 }{U^2} \right)^{\frac{1}{2}}\left(\int_{s\tau+d(\tau)}^{2L}\frac{\phi_{xxx}^2 }{U^{4-2m}} \right)^{\frac{1}{2}} \nonumber\\&%\leq\frac{CN^2(T)}{\sqrt{L}} \int_{0}^{t}\left(\left\|\langle\xi-\xi_\star\rangle ^{\frac{\beta+2 }{2}}\phi_{xx}\right\|_{L^{2}}^2+\left\|\langle\xi-\xi_\star\rangle ^{\frac{\beta+2 }{2}}\phi_{xxx}\right\|_{L^{2}}^2 \right)^{\frac{1}{2}}\left(\int\frac{\phi_{xxx}^2 }{U^{4-2m}} \right)^{\frac{1}{2}}\nonumber\\&
%\leq \frac{CN^2(T)}{\sqrt{L}}\left(\int_{0}^{t}\int \langle\xi-\xi_\star\rangle^{\beta+2} \phi_{xx}^2+\int_{0}^{t}\int\frac{\phi_{xxx}^2 }{U^{4-2m}} \right)\nonumber \\&	
\leq \frac{CTN^4(T)}{\sqrt{L}}+\frac{CN^2(T)}{\sqrt{L}}\int_{0}^{t}\int\frac{\phi_{xxx}^2 }{U^{4-2m}}.\nonumber
	\end{align}
For $\xi<0$, we have $0<U(0)<U(\xi)<u_-$, then
\begin{equation}\nonumber
\!CN(T)\!\!\int_{0}^{t}\!\int_0^{s\tau+d(\tau)}\!\langle\xi\!-\!\xi_\star\rangle ^{\frac{ \beta+4}{2}}\frac{\phi_{xx}^2|\phi_{xxx}|}{U^{3-m}}|\eta_x|\!\leq\!\frac{ CN^2(T)}{L}\! \int_{0}^{t}\!\!\int_0^{s\tau+d(\tau)}\langle\xi-\xi_\star\rangle |\phi_{xx}||\phi_{xxx}|\!	\!\leq \frac{CTN^4(T)}{\sqrt{L}}.
\end{equation}
Similarly, it holds
\begin{align}
B_{12}&%=C\int_{0}^{t}\left(\int_{0}^{s\tau+d(\tau)}+\int_{s\tau+d(\tau)}^{2L} \right)\langle\xi-\xi_\star\rangle ^{\beta+3}\frac{|\phi_{xx}|\phi_{xxx}^2}{U^{2-m}}|\eta||\eta_x|\\&
\leq\frac{CN(T)}{L}\!\int_{0}^{t}\!\int_{0}^{s\tau+d(\tau)}\!\langle\xi\!-\!\xi_\star\rangle ^{\frac{\beta+4 }{2}}\phi_{xxx}^2+\frac{CN(T)}{L}\!\int_{0}^{t}\!\int_{s\tau+d(\tau)}^{2L}\!\langle\xi\!-\!\xi_\star\rangle ^{\frac{\beta+4 }{2}}\frac{\phi_{xxx}^2}{U^{2-m}}\nonumber\\&\leq\frac{CTN^3(T)}{L}+\int_{0}^{t}\!\!\int\!\!\frac{\phi_{xxx}^2 }{U^{4\!-\!2m}},\label{53}
\end{align}
where we used the fact $\langle\xi\!-\!\xi_\star\rangle ^{\frac{\beta+4 }{2}}U^{-(2-m)}\leq\langle\xi\!-\!\xi_\star\rangle ^{\frac{3-m}{2(1-m)}}U^{-(2-m)}\sim U^{-\frac{7-3m}{2}}\leq CU^{-(4-2m)}$ as $\xi\rightarrow+\infty$. Adding \eqref{48}-\eqref{53} with \eqref{51}, one obtains
	\begin{align} \label{54}
&\bigg|2\int_{0}^{t}\int_{0}^{2L}\left(\frac{f'(U)\phi_{xxx}}{2}- \frac{\phi_{xxxx}}{U^{1-m}}-\frac{G_{xxx}}{m}\right)\langle\xi-\xi_\star\rangle ^{\beta+3}\phi_{xxx}\eta\eta_x\bigg| \notag \\
&\leq\frac{C }{\sqrt{L}}\int_{0}^{t}\int_{0}^{2L}\langle\xi-\xi_\star\rangle ^{\beta+3}\frac{\phi_{xxxx}^2 }{U^{1-m}}\eta^2+\frac{C TN^2(T)}{\sqrt{L}}+\frac{CN(T)}{\sqrt{L}}\int_{0}^{t}\int\frac{\phi_{xxx}^2 }{U^{4-2m}}.	
	\end{align}

Now letting $L\rightarrow+\infty$ in \eqref{41}, noting that \eqref{phix 0}, \eqref{47}, \eqref{54} and $|U_x(-st-d(t))|\sim \mathrm{e}^{-\lambda_-(st+d(t))}$, we have
\begin{align}\label{46}
	&\int\langle\xi-\xi_\star\rangle ^{\beta+3} \phi_{xxx}^2+(d_0+t) ^{\beta+3}\phi_{xx}^2(0,t)+\int_{0}^{t}(d_0+\tau) ^{\beta+3}\phi_{xxx}^2(0,\tau)\!+\!\int_{0}^{t}\!\int\!\langle\xi\!-\!\xi_\star\rangle ^{\beta+3}\frac{\phi_{xxxx}^2}{U^{1-m}}\nonumber\\&\leq C\left\{\!\int\!\langle\xi_0\!-\!\xi_\star\rangle ^{\beta+3}\phi_{0xxx}^2\!+\!d_0^{\beta+3}\int\left(\phi_{0xx}^2\!+\!\phi_{0x}^2\!+\! \phi_{0}^2\right)\!+\!d_0^{-1}\!+\!(d_0^{-\frac{1}{2}}\!+\!N(T))\int_{0}^{t}(d_0\!+\!\tau) ^{\beta+3}\phi_{xx}^2(0,\tau)\right.\nonumber\\&\left.\quad+\int_{0}^{t}\bigg|\frac{\langle\xi-\xi_\star\rangle ^{\beta+3}}{U^{1-m}}\phi_{xx}\left[  u_-^{1-m}f'(u_-)U_{x} -U_{xx}-\frac{ (m-1)}{u_-}\left(U_{x}+\phi_{xx}\right)^2\right]\bigg|_{x=0} \bigg|\right.\nonumber\\&\left.\quad+ \int_{0}^{t}\int\left(\langle\xi-\xi_\star\rangle ^{\beta+1}\frac{|\xi-\xi_\star| }{U^{1-m}}+\langle\xi-\xi_\star\rangle ^{\beta+3}\frac{|U_x|}{U^{2-m} } \right)|\phi_{xxx}\phi_{xxxx}|\right.\nonumber\\&\left.\quad+\int_{0}^{t}\!\!\int\langle\xi-\xi_\star\rangle ^{\beta+3}\left(\frac{|U_{xx}|}{U^{2-m} }\!+\!\frac{U_x^2}{U^{3-m} } \!+\!f''(U)|U_x|\right)\phi_{xxx}^2\!+\!\int_{0}^{t}\!\!\int\langle\xi-\xi_\star\rangle ^{\beta+1}|\xi-\xi_\star| |g'-d'|\phi_{xxx}^2\right.\nonumber\\&\left.\quad+\int_{0}^{t}\int\langle\xi-\xi_\star\rangle ^{\beta+3}\left[ \left(f'''(U)U_x^2\!+\!f''(U)|U_{xx}| \right)|\phi_{xx}\phi_{xxx}|+\big|f^{(4)}(U)U_x^3\phi_{x}\phi_{xxx}\big|\right]\right.\nonumber\\&\left.\quad+\!\int_{0}^{t}\int\langle\xi-\xi_\star\rangle ^{\beta+3}\bigg|\left(f'''(U)U_xU_{xx}\!+\!f''(U)U_{xxx} \right)\phi_{x}\phi_{xxx}\bigg|
	 +(1-m)\int_{0}^{t}\int\left[-\frac{2U_x}{U^{2-m}}\phi_{xxx}\right.\right.\nonumber\\&\left.\left.\quad\!-\!\frac{3U_{xx}}{U^{2-m}}\phi_{xx}\!-\!\frac{U_{xxx}}{U^{2-m}}\phi_{x}\!+\!(2\!-\!m)\left( \frac{3U_x^2}{U^{3-m}}\phi_{xx}\!+\!\frac{3U_xU_{xx}}{U^{3-m}}\phi_{x}\!-\!\frac{(3\!-\!m)U_{x}^3}{U^{4-m}}\phi_{x}\right)\right]_x\!\!\langle\xi\!-\!\xi_\star\rangle ^{\beta+3}\phi_{xxx}\right.\nonumber\\&\left.\quad+\int_{0}^{t}|d'(\tau)|\int\langle\xi-\xi_\star\rangle ^{\beta+3}|U_{xxx}||\phi_{xxx}|\!+\!\int_{0}^{t}\int\langle\xi\!-\!\xi_\star\rangle ^{\beta+3}|F_{xxx}\phi_{xxx}| \right.\nonumber\\&\left.\quad\!+\!\int_{0}^{t}\!\int\left(\langle\xi\!-\!\xi_\star\rangle ^{\beta+3}|G_{xxx}\phi_{xxxx}|+\langle\xi\!-\!\xi_\star\rangle ^{\beta+2}|G_{xxx}\phi_{xxx}| \right)\right\}.
\end{align}
We next estimate the terms on RHS of \eqref{46}. Utilizing $\langle\xi-\xi_\star\rangle ^{\beta+3}\sim (d_0+t)^{\beta+3}$, $|U_x(-st-d(t))|\sim \mathrm{e}^{-\lambda_-(st+d(t))}$ and $|\phi_{xx}(0,t)|\leq N(T)$, we deduce from the Cauchy-Schwarz inequality that
	\begin{align} \label{60}
&\int_{0}^{t}\bigg|\frac{\langle\xi\!-\!\xi_\star\rangle ^{\beta+3}}{U^{1-m}}\phi_{xx}\!
\left[  u_-^{1-m}f'(u_-)U_{x} \!-\!U_{xx}\!-\!\!\frac{ (m\!-\!1)}{u_-}\left(U_{x}\!+\!\phi_{xx}\right)^2\!\right]\!\bigg|_{x=0} \bigg| \notag \\
&
\leq\left(\frac{1}{4}+CN(T) \right)\int_{0}^{t}	(d_0+t) ^{\beta+3}\phi_{xx}^2(0,t)+Cd_0^{-1},
	\end{align}
and
	 \begin{align}
&\int_{0}^{t}\int\left(\langle\xi-\xi_\star\rangle ^{\beta+1}\frac{|\xi-\xi_\star| }{U^{1-m}}+\langle\xi-\xi_\star\rangle ^{\beta+3}\frac{|U_x|}{U^{2-m} } \right)|\phi_{xxx}\phi_{xxxx}|\nonumber\\&%\leq\frac{1}{8}\int_{0}^{t}\int\langle\xi-\xi_\star\rangle ^{\beta+3}\frac{\phi_{xxxx}^2 }{U^{1-m}}+C\int_{0}^{t}\int\left(\frac{\langle\xi-\xi_\star\rangle ^{\beta+1} }{U^{1-m}}+\langle\xi-\xi_\star\rangle ^{\beta+3}\frac{U_x^2 }{U^{3-m}} \right)	\phi_{xxx}^2\nonumber\\&
\leq\frac{1}{8}\int_{0}^{t}\int\langle\xi-\xi_\star\rangle ^{\beta+3}\frac{\phi_{xxxx}^2 }{U^{1-m}}+C\int_{0}^{t}\int\langle\xi-\xi_\star\rangle ^{\beta+2}\frac{\phi_{xxx}^2 }{U^{1-m}},	
	 \end{align}
where we have used \eqref{e5} and $\langle\xi-\xi_\star\rangle\frac{U_x^2 }{U^{3-m}}\leq \langle\xi-\xi_\star\rangle U^{1-m}\leq C\leq \frac{C}{U^{1-m}}$ in the last inequality. Similarly, one has
\begin{equation}\nonumber
\int_{0}^{t}\!\!\int\langle\xi-\xi_\star\rangle ^{\beta+3}\left(\frac{|U_{xx}|}{U^{2-m} }\!+\!\frac{U_x^2}{U^{3-m} } \!+\!f''(U)|U_x|\right)\phi_{xxx}^2\leq\! C\int_{0}^{t}\!\!\int\!\langle\xi\!-\!\xi_\star\rangle ^{\beta+2}\phi_{xxx}^2,	
\end{equation}
\begin{equation}\nonumber
\int_{0}^{t}\!\!\int\langle\xi-\xi_\star\rangle ^{\beta+1}|\xi-\xi_\star| |g'-d'|\phi_{xxx}^2\leq C\int_{0}^{t}\int\langle\xi-\xi_\star\rangle ^{\beta+2}\phi_{xxx}^2,
\end{equation}
\begin{equation}\nonumber
	 \begin{aligned}
&\int_{0}^{t}\int\langle\xi-\xi_\star\rangle ^{\beta+3}\left[ \left(f'''(U)U_x^2\!+\!f''(U)|U_{xx}| \right)|\phi_{xx}\phi_{xxx}|+\big|f^{(4)}(U)U_x^3\phi_{x}\phi_{xxx}\big|\right]\\&\leq C\int_{0}^{t}\int\langle\xi-\xi_\star\rangle ^{\beta+3}\left(U_x^2+|U_{xx}| \right)(\phi_{xx}^2+\phi_{xxx}^2)+C\int_{0}^{t}\int\langle\xi-\xi_\star\rangle ^{\beta+3}|U_x|^3(\phi_{x}^2+\phi_{xxx}^2)\\&\leq C\int_{0}^{t}\int\left( \langle\xi-\xi_\star\rangle ^{\beta+2}\phi_{xxx}^2+\langle\xi-\xi_\star\rangle ^{\beta+1}\phi_{xx}^2+\langle\xi-\xi_\star\rangle ^{\beta}\phi_{x}^2\right),
	 \end{aligned}
\end{equation}
and
\begin{equation}\nonumber
	\begin{aligned}
&\int_{0}^{t}\int\langle\xi-\xi_\star\rangle ^{\beta+3}\bigg|\left(f'''(U)U_xU_{xx}\!+\!f''(U)U_{xxx} \right)\phi_{x}\phi_{xxx}\bigg|\\&%\leq\int_{0}^{t}\int\langle\xi-\xi_\star\rangle ^{\beta+3}\left(|U_x||U_{xx}|+|U_{xxx}| \right)(\phi_{x}^2+\phi_{xxx}^2)\\&
\leq C\int_{0}^{t}\int\left( \langle\xi-\xi_\star\rangle ^{\beta+2}\phi_{xxx}^2+\langle\xi-\xi_\star\rangle ^{\beta}\phi_{x}^2\right).
	\end{aligned}
\end{equation}
By virtue of \eqref{q1}-\eqref{q4}, \eqref{e5} and $\langle\xi-\xi_\star\rangle\sim U^{-(1-m)}$ as $\xi\rightarrow+\infty$, a direct calculation yields
	\begin{align}
&\text{the fourth term from the end on RHS of \eqref{46}}\nonumber\\&\leq C\int_{0}^{t}\!\int\!\left[\frac{|U_x|}{U^{2-m}}|\phi_{xxxx}|+\left(\frac{U_x^2}{U^{3-m}}+\frac{|U_{xx}|}{U^{2-m}} \right)|\phi_{xxx}|+\left(\frac{|U_{xxx}|}{U^{2-m}} +\frac{|U_x||U_{xx}|}{U^{3-m}}+\frac{|U_x|^3}{U^{4-m}}\right)|\phi_{xx}|\right.\nonumber\\&\left.\quad+\left(\frac{|U_x||U_{xxx}|}{U^{3-m}}+\frac{|U_{xxxx}|}{U^{2-m}} +\frac{U_x^2|U_{xx}|}{U^{4-m}}+\frac{U_{xx}^2}{U^{3-m}}+\frac{U_x^4}{U^{5-m}}\right)|\phi_{x}|	\right]	\langle\xi-\xi_\star\rangle ^{\beta+3}|\phi_{xxx}| \nonumber\\&
%\leq \frac{1}{8}\int_{0}^{t}\int\langle\xi-\xi_\star\rangle ^{\beta+3}\frac{\phi_{xxxx}^2 }{U^{1-m}}+C\int_{0}^{t}\!\int\!\langle\xi-\xi_\star\rangle ^{\beta+3}\left(\frac{U_x^2}{U^{3-m}}+\frac{|U_{xx}|}{U^{2-m}} \right)\phi_{xxx}^2\nonumber\\&\quad+C\int_{0}^{t}\!\int\!\langle\xi-\xi_\star\rangle ^{\beta+3}\left(\frac{|U_{xxx}|}{U^{2-m}} +\frac{|U_x||U_{xx}|}{U^{3-m}}+\frac{|U_x|^3}{U^{4-m}}\right)\left(\phi_{xx}^2+\phi_{xxx}^2 \right)\nonumber\\&\quad+C\int_{0}^{t}\!\int\!\langle\xi-\xi_\star\rangle ^{\beta+3}\left(\frac{|U_x||U_{xxx}|}{U^{3-m}}+\frac{|U_{xxxx}|}{U^{2-m}} +\frac{U_x^2|U_{xx}|}{U^{4-m}}+\frac{U_{xx}^2}{U^{3-m}}+\frac{U_x^4}{U^{5-m}}\right)\left(\phi_{x}^2+\phi_{xxx}^2 \right)\nonumber\\&
\leq\frac{1}{8}\int_{0}^{t}\int\langle\xi-\xi_\star\rangle ^{\beta+3}\frac{\phi_{xxxx}^2 }{U^{1-m}} \notag \\
&\ \ \ +C\int_{0}^{t}\!\int\!\left(\langle\xi-\xi_\star\rangle ^{\beta+2}\phi_{xxx}^2+\langle\xi-\xi_\star\rangle ^{\beta+1}\phi_{xx}^2+\langle\xi-\xi_\star\rangle ^{\beta}\phi_{x}^2 \right).\label{61}
	\end{align}
We decompose third term from the end of \eqref{46} as
\begin{equation}\nonumber
\int_{0}^{t}|d'(\tau)|\int \langle\xi-\xi_\star\rangle ^{\beta+3}|U_{xxx}||\phi_{xxx}|=\int_{0}^{t}|d'(\tau)|\left(\int_{0}^{s\tau+d(\tau)}+\int_{s\tau+d(\tau)}^{+\infty} \right)\langle\xi-\xi_\star\rangle ^{\beta+3}|U_{xxx}||\phi_{xxx}|.
\end{equation}	
For $\xi>0$, since $U_{xxx}^2U^{1-m}\sim\langle\xi-\xi_\star\rangle^{-\frac{9-7m}{ 1-m}} $ as $\xi\rightarrow+\infty$ and $\int_{s\tau+d(\tau)}^{+\infty}\langle\xi-\xi_\star\rangle^{\beta+4}U_{xxx}^2U^{1-m}\leq\int_{s\tau+d(\tau)}^{+\infty}\langle\xi-\xi_\star\rangle^{-6}\leq	C$, we get from \eqref{e1} that
	\begin{align}
		&\int_{0}^{t}|d'(\tau)|\int_{s\tau+d(\tau)}^{+\infty} \langle\xi-\xi_\star\rangle ^{\beta+3}|U_{xxx}||\phi_{xxx}|\nonumber\\&\leq\frac{1}{2}\int_{0}^{t}\int_{s\tau+d(\tau)}^{+\infty}\langle\xi-\xi_\star\rangle^{\beta+2}\frac{\phi_{xxx}^2 }{U^{1-m}}+\frac{1}{2}\int_{0}^{t}|d'(\tau)|^2\int_{s\tau+d(\tau)}^{+\infty}\langle\xi-\xi_\star\rangle^{\beta+4}U_{xxx}^2U^{1-m}\nonumber\\&\leq C\left(\mathrm{e}^{-2\gamma d_0}+\int_{0}^{t}\int\langle\xi-\xi_\star\rangle^{\beta+2}\frac{\phi_{xxx}^2 }{U^{1-m}}+\int_{0}^{t}|\phi_{xx}^2(0,\tau)| \right).\label{55}
	\end{align}	
For $\xi<0$, it holds
	\begin{align}
		&\int_{0}^{t}|d'(\tau)|\int_{0}^{s\tau+d(\tau)} \langle\xi-\xi_\star\rangle ^{\beta+3}|U_{xxx}||\phi_{xxx}|\nonumber	\\&\leq \int_{0}^{t}|d'(\tau)|\left(\sup_{x\in [0,s\tau+d(\tau)]}\langle\xi-\xi_\star\rangle^\frac{\beta+\frac{5}{2}}{2} |\phi_{xxx}| \right)\int_{0}^{s\tau+d(\tau)} \langle\xi-\xi_\star\rangle^\frac{\beta+\frac{1}{2}}{2}  |U_{xxx}|,\nonumber
	\end{align}
where, in view of \eqref{st+d(t)},	
\begin{equation}\nonumber
	\int_{0}^{s\tau+d(\tau)} \langle\xi-\xi_\star\rangle^\frac{\beta+\frac{1}{2}}{2}  |U_{xxx}|\leq C\left(d_0+\tau \right) ^\frac{\beta+\frac{5}{2}}{2},
\end{equation}	
and
\begin{equation}\nonumber
	\begin{aligned}
		\langle\xi\!-\!\xi_\star\rangle^{\beta+\frac{5}{2}} \phi_{xxx}^2&\!\!=\!\!-\!\!\int_{x}^{0}\!\!\left[ \frac{\partial}{\partial x}\!\left(\!\langle\xi\!-\!\xi_\star\rangle^{\beta+\frac{5}{2}} \right)\!\phi_{xxx}^2\!\!+\!2\langle\xi-\xi_\star\rangle^{\beta+\frac{5}{2}}\phi_{xxx}\phi_{xxxx}\!\right]\!\!+\! \langle\xi\!-\!\xi_\star\rangle^{\beta+\frac{5}{2}}\big|_{x=0} \phi_{xxx}^2(0,\tau)\\&\!\leq\! C\!\int_{0}^{s\tau+d(\tau)}\!\!\left(\!\langle\xi\!-\!\xi_\star\rangle^{\beta+2}\phi_{xxx}^2+\langle\xi-\xi_\star\rangle^{\beta+3}\phi_{xxxx}^2 \right)\!+\! \langle\xi\!-\!\xi_\star\rangle^{\beta+\frac{5}{2}}\big|_{x=0} \phi_{xxx}^2(0,\tau).
	\end{aligned}
\end{equation}	
Then, by $\langle\xi-\xi_\star\rangle^{\beta+\frac{5}{2}}\big|_{x=0}\sim (d_0+\tau)^{\beta+\frac{5}{2}}$ due to \eqref{st+d(t)}, it has
	\begin{align}
		&\int_{0}^{t}|d'(\tau)|\int_{0}^{s\tau+d(\tau)} \langle\xi-\xi_\star\rangle^{\beta+3} |U_{xxx}||\phi_{xxx}|%\\&\leq C\!\int_{0}^{t}\!|d'(\tau)|\!\left[\int_{0}^{s\tau+d(\tau)}\!\left(\!\langle\xi\!-\!\xi_\star\rangle^{\beta+1}\phi_{xx}^2\!+\!\langle\xi\!-\!\xi_\star\rangle^{\beta+2}\phi_{xxx}^2 \right)\!\!+\!\! \langle\xi\!-\!\xi_\star\rangle^{\beta+\frac{3}{2}}\big|_{x=0} \phi_{xx}^2(0,\tau)\right]^\frac{1}{2}\!\left(1\!+\!d_0\!+\!\tau \right) ^\frac{\beta+\frac{5}{2}}{2}
		\nonumber\\&\!\leq\! C\!\int_{0}^{t}\!|d'(\tau)|\!\left[\!\left(\!\int_{0}^{s\tau+d(\tau)}\!\langle\xi\!-\!\xi_\star\rangle^{\beta+2}\phi_{xxx}^2\!\right)^\frac{1}{2}\!\!+\!\left(\!\int_{0}^{s\tau+d(\tau)}\!\langle\xi\!-\!\xi_\star\rangle^{\beta+3}\phi_{xxxx}^2\!\right)^\frac{1}{2}\right]\left(d_0\!+\!\tau \right) ^\frac{\beta+\frac{5}{2}}{2}\nonumber\\&\quad+C\int_{0}^{t}|d'(\tau)|\langle\xi-\xi_\star\rangle^{\frac{\beta+\frac{5}{2} }{2}}\big|_{x=0} |\phi_{xxx}(0,\tau)| \left(d_0+\tau \right) ^\frac{\beta+\frac{5}{2}}{2}	\nonumber\\&\leq\frac{1}{8}\int_{0}^{t}\int_{0}^{s\tau+d(\tau)}\langle\xi-\xi_\star\rangle^{\beta+3}\phi_{xxxx}^2 +C\left(\mathrm{e}^{-2\gamma d_0}+\int_{0}^{t}\int_{0}^{s\tau+d(\tau)}\langle\xi-\xi_\star\rangle^{\beta+2}\phi_{xxx}^2\right.\nonumber\\&\quad\left.+\int_{0}^{t}\left(d_0+\tau \right)^{\beta+\frac{5}{2}}|\phi_{xx}^2(0,\tau)|  +\int_{0}^{t}\left(d_0+\tau \right)^{\beta+\frac{5}{2}}|\phi_{xxx}^2(0,\tau)|  \right)	. \label{56}
	\end{align}
Therefore, combining \eqref{55} and \eqref{56}, we get
	\begin{align} \label{62}
		&\int_{0}^{t}|d'(\tau)|\int\langle\xi-\xi_\star\rangle^{\beta+3} |U_{xxx}||\phi_{xxx}| \notag \\
		&\leq\frac{1}{8}\int_{0}^{t}\int\langle\xi-\xi_\star\rangle^{\beta+3}\phi_{xxxx}^2\!+\!C\left(\int_{0}^{t}\int\langle\xi-\xi_\star\rangle^{\beta+2}\frac{\phi_{xxx}^2 }{U^{1-m}}+\mathrm{e}^{-2\gamma d_0} \right. \notag \\
&\left.\quad+d_0^{-\frac{1}{2}}\int_{0}^{t}\left(d_0+\tau \right) ^{\beta+3}|\phi_{x x}^2(0,\tau)|+d_0^{-\frac{1}{2}}\int_{0}^{t}\left(d_0+\tau \right) ^{\beta+3}|\phi_{xx x}^2(0,\tau)|\right).
	\end{align}	
Furthermore, by virtue of \eqref{Fzz},  $\left\|\langle\xi-\xi_\star\rangle^{\frac{\beta+1 }{2}}\phi_x(\cdot,t)\right\|_{L^{\infty}}\leq CN(T)$ and $\left\|\langle\xi-\xi_\star\rangle^{\frac{\beta+2 }{2}}\phi_{xx}(\cdot,t)\right\|_{L^{\infty}}\\\leq CN(T)$, as in the proof of \eqref{40}, we have
	\begin{align}\label{63}
		&\int_{0}^{t}\int \langle\xi-\xi_\star\rangle^{\beta+3}|F_{xxx}\phi_{xxx} |\nonumber\\&%\leq C\int_{0}^{t}\int\langle\xi-\xi_\star\rangle^{\beta+3}|\phi_{xxx}|\left[ \left(|U_x|^3+|U_x||U_{xx}|+|U_{xxx}|\right) \phi_{x}^2+\left(U_x^2+|U_{xx}| \right)|\phi_{x}||\phi_{xx}|\right.\nonumber\\&\left.\quad+|U_x||\phi_{x}||\phi_{xxx}|+|\phi_{x}||\phi_{xxxx}|+|U_x|\phi_{xx}^2+|\phi_{x x}|^3+|\phi_{xx}||\phi_{xxx}|\right]\nonumber\\&
		\leq CN(T)\int_{0}^{t}\int\langle\xi-\xi_\star\rangle^{\frac{\beta+5 }{2}}|\phi_{xxx}|\left[\left(|U_x|^3+|U_x||U_{xx}|+|U_{xxx}|\right)|\phi_{x}|+\left(U_x^2+|U_{xx}| \right)|\phi_{xx}|\right.\nonumber\\&\left.\quad+|U_x||\phi_{xxx}|+|\phi_{xxxx}|\right]+CN(T)\int_{0}^{t}\int\langle\xi-\xi_\star\rangle^{\frac{\beta+4 }{2}}|\phi_{xxx}|\left(|U_x||\phi_{xx}|+\phi_{x x}^2+|\phi_{xxx}|\right)\nonumber\\
		&%\leq  CN(T)\int_{0}^{t}\int\langle\xi-\xi_\star\rangle^{\frac{\beta+5 }{2}}\left(|U_x|^3+|U_x||U_{xx}|+|U_{xxx}|\right)\left(\phi_{xxx}^2+\phi_{x}^2 \right)\nonumber\\&\quad+CN(T)\int_{0}^{t}\int\langle\xi-\xi_\star\rangle^{\frac{\beta+5 }{2}}\left(U_x^2+|U_{xx}| \right)\left(\phi_{xxx}^2+\phi_{xx}^2 \right)+CN(T)\int_{0}^{t}\int\langle\xi-\xi_\star\rangle^{\frac{\beta+5 }{2}}|U_x|\phi_{xxx}^2\nonumber\\&\quad+N(T)\int_{0}^{t}\int\langle\xi-\xi_\star\rangle^{\beta+3}\phi_{xxxx}^2+CN(T)\int_{0}^{t}\int\langle\xi-\xi_\star\rangle^{\beta+2}\phi_{xxx}^2\nonumber\\&\quad+CN(T)\int_{0}^{t}\int\langle\xi-\xi_\star\rangle^{\frac{\beta+4 }{2}}|U_x|\left(\phi_{xxx}^2+\phi_{xx}^2 \right)+CN^2(T)\int_{0}^{t}\int\langle\xi-\xi_\star\rangle\left(\phi_{xxx}^2+\phi_{xx}^2 \right)\nonumber\\&
		\!\leq\! N(T)\!\!\int_{0}^{t}\!\!\int\!\langle\xi\!-\!\xi_\star\rangle^{\beta+3}\phi_{xxxx}^2\!+\!CN(T)\!\!\int_{0}^{t}\!\!\int\!\left(\!\langle\xi\!-\!\xi_\star\rangle^{\beta+2}\phi_{xxx}^2\!+\! \langle\xi\!-\!\xi_\star\rangle^{\beta+1}\phi_{xx}^2\!+\!\langle\xi\!-\!\xi_\star\rangle^{\beta}\phi_{x}^2 \right).\nonumber
	\end{align}
%where we have used  \eqref{q1}-\eqref{q3}, \eqref{e5} and $\langle\xi-\xi_\star\rangle\sim U^{-(1-m)}$ as $\xi\rightarrow+\infty$ in the last inequality.
From \eqref{Gzzz}, noting that $\left\|\phi_x(\cdot,t)/U\right\|_{L^{\infty}}\leq CN(T)$, the last two term on RHS of \eqref{46} can be estimated as
	\begin{align}
&\int_{0}^{t}\!\int\left(\langle\xi\!-\!\xi_\star\rangle ^{\beta+3}|G_{xxx}\phi_{xxxx}|+\langle\xi\!-\!\xi_\star\rangle ^{\beta+2}|G_{xxx}\phi_{xxx}| \right)\nonumber\\&%\leq C\int_{0}^{t}\int \langle\xi-\xi_\star\rangle^{\beta+3}|\phi_{xxxx}|	\left[\left(\frac{|U_x|^3 }{U^{5-m}}+\frac{|U_x||U_{xx}| }{U^{4-m}}+\frac{|U_{xxx}| }{U^{3-m}} \right)\phi_{x}^2+\left(\frac{U_x^2 }{U^{4-m}}+\frac{|U_{xx}| }{U^{3-m}} \right)|\phi_{x}||\phi_{xx}|\right.\nonumber\\&\left.\quad+\frac{|U_x| }{U^{3-m}}|\phi_{x}||\phi_{xxx}|+\frac{|\phi_{x}||\phi_{xxxx}| }{U^{2-m}}+\frac{|U_x| }{U^{3-m}}\phi_{xx}^2+\frac{|\phi_{xx}|^3}{U^{3-m}}+\frac{|\phi_{xx}||\phi_{xxx}|}{U^{2-m}}\right]\nonumber\\&
\leq CN(T)\int_{0}^{t}\int\left(\langle\xi-\xi_\star\rangle^{\beta+3}|\phi_{xxxx}|+\langle\xi-\xi_\star\rangle^{\beta+2}|\phi_{xxx}| \right)\left[\left(\frac{|U_x|^3 }{U^{4-m}}\!+\!\frac{|U_x||U_{xx}| }{U^{3-m}}\!+\!\frac{|U_{xxx}| }{U^{2-m}} \right)|\phi_{x}|\right.\nonumber\\&\left.\quad+\left(\frac{U_x^2 }{U^{3-m}}+\frac{|U_{xx}| }{U^{2-m}} \right)|\phi_{xx}|+\frac{|U_x| }{U^{2-m}}|\phi_{xxx}|+\frac{|\phi_{xxxx}| }{U^{1-m}}\right]+C\int_{0}^{t}\int \langle\xi-\xi_\star\rangle^{\beta+2}\left(\frac{|U_x| }{U^{3-m}}\phi_{xx}^2|\phi_{xxx}|\right.\nonumber\\&\left.\quad+\frac{|\phi_{xx}|^3|\phi_{xxx}|}{U^{3-m}} \right)
+C\int_{0}^{t}\!\!\int\! \langle\xi\!-\!\xi_\star\rangle^{\beta+3}\left(\frac{|U_x| }{U^{3-m}}\phi_{xx}^2|\phi_{xxxx}|+ \frac{|\phi_{xx}|^3|\phi_{xxxx}|}{U^{3-m}}\right)\nonumber\\&\quad+C\int_{0}^{t}\int\left(\langle\xi-\xi_\star\rangle^{\beta+3}|\phi_{xxxx}|+\langle\xi-\xi_\star\rangle^{\beta+2}|\phi_{xxx}| \right) \frac{|\phi_{xx}||\phi_{xxx}|}{U^{2-m}}\nonumber\\&\triangleq I_4+\cdots+I_{7},\nonumber
	\end{align}
where, by virtue of \eqref{q1}-\eqref{q3} and \eqref{e5}, it holds that
	\begin{align}
I_4&\leq\! CN(T)\!\int_{0}^{t}\int\left(\langle\xi-\xi_\star\rangle^{\beta+3}+\langle\xi-\xi_\star\rangle^{\beta+2} \right)\frac{\phi_{xxxx}^2 }{U^{1-m}}+C\int_{0}^{t}\int\langle\xi-\xi_\star\rangle^{\beta+3}\frac{U_x^2 }{U^{3-m}}\phi_{xxx}^2\nonumber\\&\quad+\!C\!\int_{0}^{t}\!\int\langle\xi\!-\!\xi_\star\rangle^{\beta+2}\frac{\phi_{xxx}^2 }{U^{1-m}}\!\!+\!C\int_{0}^{t}\!\int\!\langle\xi-\xi_\star\rangle^{\beta+2}\phi_{xxx}^2\!+\!C\int_{0}^{t}\!\int\left(\langle\xi-\xi_\star\rangle^{\beta+3}\!+\!\langle\xi-\xi_\star\rangle^{\beta+2} \right)\nonumber\\&\quad\times\left[\left(\frac{U_x^6 }{U^{7-m}}\!+\!\frac{U_x^2U_{xx}^2 }{U^{5-m}}\!+\!\frac{U_{xxx}^2 }{U^{3-m}} \right)\phi_{x}^2+\left(\frac{U_x^4 }{U^{5-m}}\!+\!\frac{U_{xx}^2 }{U^{3-m}} \right)\phi_{xx}^2\right]\nonumber\\&\leq \!CN(T)\!\int_{0}^{t}\!\!\int\langle\xi\!-\!\xi_\star\rangle^{\beta+3}\frac{\phi_{xxxx}^2 }{U^{1-m}}\!+\!C\int_{0}^{t}\!\!\int\left( \langle\xi\!-\!\xi_\star\rangle^{\beta+2}\frac{\phi_{xxx}^2 }{U^{1-m}}\!+\!\langle\xi\!-\!\xi_\star\rangle^{\beta+1}\phi_{xx}^2\!+\!\langle\xi-\xi_\star\rangle^{\beta}\phi_{x}^2\right)\nonumber.
	\end{align}
By $\left\|\langle\xi-\xi_\star\rangle^{\frac{\beta+2 }{2}}\phi_{xx}(\cdot,t)\right\|_{L^{\infty}}\leq CN(T)$ and $|U_x|\leq CU^{2-m}$, it holds that
\begin{align}
	I_{5}&\leq CN(T)\int_{0}^{t}\int\langle\xi\!-\!\xi_\star\rangle^{\frac{\beta+2 }{2}}\frac{|\phi_{xx}||\phi_{xxx}|}{U}+CN^2(T)\int_{0}^{t}\int\frac{|\phi_{xx}||\phi_{xxx}|}{U^{3-m}}\nonumber\\&\leq C\int_{0}^{t}\int\left(\frac{\phi_{xxx}^2 }{U^{4-2m}}+\langle\xi\!-\!\xi_\star\rangle^{\beta+2}\phi_{xxx}^2+\frac{\phi_{xx}^2 }{U^2} \right),\nonumber
\end{align}
Note that
\begin{align}
I_6= C\int_{0}^{t}\left(\int_{0}^{s\tau+d(\tau)}+\int_{s\tau+d(\tau)}^{+\infty} \right)\langle\xi-\xi_\star\rangle^{\beta+3}\left(\frac{|U_x| }{U^{3-m}}\phi_{xx}^2|\phi_{xxxx}|+\frac{|\phi_{xx}|^3|\phi_{xxxx}| }{U^{3-m}}\right).\nonumber	
\end{align}
where, for $\xi<0$, noting that $U(\xi)>U(0)>0$ and $\left\|\langle\xi-\xi_\star\rangle^{\frac{\beta+2 }{2}}\phi_{xx}(\cdot,t)\right\|_{L^{\infty}}\leq CN(T)$,
\begin{equation}\nonumber
	\begin{aligned}
&\int_{0}^{t}\!\!\int_{0}^{s\tau+d(\tau)}\langle\xi-\xi_\star\rangle^{\beta+3}\left(\frac{|U_x| }{U^{3-m}}\phi_{xx}^2|\phi_{xxxx}|+\frac{|\phi_{xx}|^3|\phi_{xxxx}| }{U^{3-m}}\right)\\&\leq CN(T)\int_{0}^{t}\!\!\int_{0}^{s\tau+d(\tau)}\left(\langle\xi\!-\!\xi_\star\rangle^{\frac{\beta+4 }{2}}+\langle\xi-\xi_\star\rangle \right)|\phi_{xx}||\phi_{xxxx}|\\&\leq N(T)\int_{0}^{t}\!\!\int\langle\xi\!-\!\xi_\star\rangle^{\beta+3}\phi_{xxxx}^2+C\int_{0}^{t}\!\!\int\langle\xi\!-\!\xi_\star\rangle\phi_{xx}^2,		 
	\end{aligned}
\end{equation}
and for $\xi>0$, noting $\langle\xi-\xi_\star
\rangle^2\frac{U_x^2}{U^{5-m}}\leq CU^{-(3-m)}$ as $\xi\rightarrow+\infty$, $\|\phi_{xx}(\cdot,t)/U^{\frac{3-m}{2}}\|_{L^2}\leq CN(T)$ and
\begin{align}
	\langle\xi-\xi_\star\rangle \frac{\phi_{xx}^2}{U^2}=-\int_{x}^{+\infty}\frac{\partial}{\partial x}\left( \langle\xi-\xi_\star\rangle \frac{\phi_{xx}^2}{U^2}\right)&%=-2\int_{x}^{+\infty}\langle\xi-\xi_\star\rangle\frac{\phi_{xx}\phi_{xxx} }{U^2}-\int_{x}^{+\infty}\frac{\xi-\xi_\star}{\langle\xi-\xi_\star\rangle }\frac{\phi_{xx}^2}{U^2}+2\int_{x}^{+\infty}\langle\xi-\xi_\star\rangle\frac{U_x }{U^3}\phi_{xx}^2\nonumber\\&
	\leq C\left( \int_{x}^{+\infty}\frac{\phi_{xx}^2}{U^2}+\int_{x}^{+\infty}\langle\xi-\xi_\star\rangle^2\frac{\phi_{xxx}^2}{U^2}\right)	\nonumber\\&\leq C\left( \int\frac{\phi_{xx}^2}{U^2}+\int\frac{\phi_{xxx}^2}{U^{4-2m}}\right),	\nonumber
\end{align}
one has,
\begin{align}
		&\int_{0}^{t}\int_{s\tau+d(\tau)}^{+\infty} \langle\xi-\xi_\star\rangle^{\beta+3}\left(\frac{|U_x| }{U^{3-m}}\phi_{xx}^2|\phi_{xxxx}|+\frac{|\phi_{xx}|^3|\phi_{xxxx}| }{U^{3-m}}\right)\nonumber\\&\leq C \int_{0}^{t}\left\|\langle\xi-\xi_\star\rangle ^{\frac{\beta+1 }{2}}\phi_{xx}\right\|_{L^{\infty}}\left(\int_{s\tau+d(\tau)}^{+\infty} \langle\xi-\xi_\star\rangle^{2}\frac{U_x^2 }{U^{5-m}}\phi_{xx}^2 \right)^\frac{1}{2}\left(\int_{s\tau+d(\tau)}^{+\infty} \langle\xi-\xi_\star\rangle^{\beta+3}\frac{\phi_{xxxx}^2 }{U^{1-m}} \right)^\frac{1}{2}\nonumber\\&\quad+CN(T) \!\int_{0}^{t}\!\left\|\langle\xi-\xi_\star\rangle ^{\frac{1 }{2}}\frac{\phi_{xx}}{U}\right\|_{L^{\infty}(s\tau+d(\tau),+\infty)}\left(\int_{s\tau+d(\tau)}^{+\infty} \frac{\phi_{xx}^2}{U^{3-m}} \right)^\frac{1}{2}\!\!\left(\int_{s\tau+d(\tau)}^{+\infty} \langle\xi\!-\!\xi_\star\rangle^{\beta+3}\frac{\phi_{xxxx}^2 }{U^{1-m}} \right)^\frac{1}{2}\nonumber\\&
		%\leq CN(T)\int_{0}^{t}\left(\int \langle\xi-\xi_\star\rangle^{\beta+1}\phi_{xx}^2 +\int \langle\xi-\xi_\star\rangle^{\beta+1}\phi_{xxx}^2  \right)^\frac{1}{2}\left(\int_{s\tau+d(\tau)}^{+\infty} \langle\xi-\xi_\star\rangle^{\beta+3}\frac{\phi_{xxxx}^2 }{U^{1-m}} \right)^\frac{1}{2}\nonumber\\&
		\leq\! CN(T)\! \int_{0}^{t}\!\int\langle\xi-\xi_\star\rangle ^{\beta+3}\frac{\phi_{xxxx}^2 }{U^{1-m}}\!+\!C \int_{0}^{t}\!\int\!\left(\!\langle\xi-\xi_\star\rangle ^{\beta+2}\phi_{xxx}^2\!+\!\langle\xi-\xi_\star\rangle ^{\beta+1}\phi_{xx}^2\! +\! \frac{\phi_{xx}^2}{U^2}\! +\! \frac{\phi_{xxx}^2}{U^{4-2m}} \right)\nonumber.	
	\end{align}
Moreover, it holds
\begin{align}
	I_{7}&= \int_{0}^{t}\left(\int_{0}^{s\tau+d(\tau)}+\int_{s\tau+d(\tau)}^{+\infty}\right)\left(\langle\xi-\xi_\star\rangle^{\beta+3}|\phi_{xxxx}|+\langle\xi-\xi_\star\rangle^{\beta+2}|\phi_{xxx}| \right) \frac{|\phi_{xx}||\phi_{xxx}|}{U^{2-m}}\nonumber\\&\leq CN(T)\!\int_{0}^{t}\!\int_{0}^{s\tau+d(\tau)}\langle\xi-\xi_\star\rangle^{\frac{\beta+4}{2}}|\phi_{xxx}||\phi_{xxxx}|\!+\!CN(T)\int_{0}^{t}\!\int_{s\tau+d(\tau)}^{+\infty}\langle\xi\!-\!\xi_\star\rangle^{\frac{\beta+4}{2}}\frac{|\phi_{xxx}||\phi_{xxxx}| }{U^{2-m}}\nonumber\\&\quad+CN(T)\int_{0}^{t}\int_{0}^{s\tau+d(\tau)}\langle\xi-\xi_\star\rangle^{\frac{\beta+2}{2}}\phi_{xxx}^2	 \!+\!CN(T)\int_{0}^{t}\int_{s\tau+d(\tau)}^{+\infty}\langle\xi-\xi_\star\rangle^{\frac{\beta+2}{2}}\frac{\phi_{xxx}^2}{U^{2-m}}\nonumber
	\\&%\leq CN(T)\int_{0}^{t}\int_{0}^{s\tau+d(\tau)}\langle\xi-\xi_\star\rangle^{\beta+3}\phi_{xxxx}^2+C\int_{0}^{t}\int_{0}^{s\tau+d(\tau)}\left(\langle\xi-\xi_\star\rangle+\langle\xi-\xi_\star\rangle^{\beta+2} \right)\phi_{xxx}^2\nonumber\\&\quad+CN(T) \int_{0}^{t}\int_{s\tau+d(\tau)}^{+\infty}\langle\xi-\xi_\star\rangle ^{\beta+3}\frac{\phi_{xxxx}^2 }{U^{1-m}}+	C\int_{0}^{t}\int_{s\tau+d(\tau)}^{+\infty}\left(\langle\xi-\xi_\star\rangle \frac{\phi_{xxx}^2 }{U^{3-m}}+\frac{\phi_{xxx}^2 }{U^{4-2m}} \right)\nonumber\\&
	\leq CN(T) \int_{0}^{t}\int\langle\xi-\xi_\star\rangle ^{\beta+3}\frac{\phi_{xxxx}^2 }{U^{1-m}}+	C\int_{0}^{t}\int\frac{\phi_{xxx}^2 }{U^{4-2m}}+C\int_{0}^{t}\int\langle\xi-\xi_\star\rangle^{\beta+2}\phi_{xxx}^2,\label{58}
\end{align}
where we have used the fact $\langle\xi-\xi_\star\rangle\sim U^{-(1-m)}$ and $\langle\xi-\xi_\star\rangle^{\frac{\beta+2}{2}}\leq\langle\xi-\xi_\star\rangle^{\frac{1+m}{2(1-m)}} \sim U^{-\frac{1+m}{2}}\leq CU^{-(2-m)}$ as $\xi\rightarrow+\infty$ due to $\beta\leq \frac{3m-1}{1-m}$ and $m<1$ in the last inequality.
%Therefore, adding \eqref{57}-\eqref{58} with \eqref{59}, we get
%	\begin{align}
%		\int_{0}^{t}\int \bigg|\langle\xi-\xi_\star\rangle^{\beta+3}G_{xxx}\phi_{xxxx} \bigg|&\leq CN(T)\!\int_{0}^{t}\!\!\int\langle\xi\!-\!\xi_\star\rangle^{\beta+3}\frac{\phi_{xxxx}^2 }{U^{1-m}}+	C\int_{0}^{t}\int\left(\frac{\phi_{xxx}^2 }{U^{4-2m}}+\frac{\phi_{xx}^2 }{U^2} \right)\nonumber\\&\quad+C\int_{0}^{t}\!\!\int\left( \langle\xi\!-\!\xi_\star\rangle^{\beta+2}\frac{\phi_{xxx}^2 }{U^{1-m}}\!+\!\langle\xi\!-\!\xi_\star\rangle^{\beta+1}\phi_{xx}^2\!+\!\langle\xi-\xi_\star\rangle^{\beta}\phi_{x}^2\right)\label{65}.
%	\end{align}
Combining \eqref{60}-\eqref{61}, \eqref{62}-\eqref{58}, we have
\begin{align}
	&\int\langle\xi-\xi_\star\rangle ^{\beta+3} \phi_{xxx}^2+\left(\frac{3}{4}-CN(T) \right)(d_0+t) ^{\beta+3}\phi_{xx}^2(0,t)\nonumber\\&+(1-d_0^{-\frac{1}{2}})\int_{0}^{t}(d_0+\tau) ^{\beta+3}\phi_{xxx}^2(0,\tau)+\left(\frac{3}{8}-CN(T) \right)\int_{0}^{t}\int\langle\xi-\xi_\star\rangle ^{\beta+3}\frac{\phi_{xxxx}^2}{U^{1-m}}\nonumber\\&\leq C\!\left(\int\!\langle\xi_0\!-\!\xi_\star\rangle ^{\beta+3}\phi_{0xxx}^2\!+\!d_0^{\beta+3}\!\!\int\!\left( \phi_{0xx}^2\!+\!\phi_{0x}^2\!+\!\phi_{0}^2\right)\!+\!d_0^{-1}\!+\!(d_0^{-\frac{1}{2}}\!+\!N(T))\int_{0}^{t}(d_0\!+\!\tau) ^{\beta+3}\phi_{xx}^2(0,\tau)\!\right)\nonumber\\&\quad+C\int_{0}^{t}\int\left( \langle\xi-\xi_\star\rangle^{\beta+2}\frac{\phi_{xxx}^2 }{U^{1-m}}+\langle\xi-\xi_\star\rangle^{\beta+1}\phi_{xx}^2+\langle\xi-\xi_\star\rangle^{\beta}\phi_{x}^2+\frac{\phi_{xxx}^2 }{U^{4-2m}}+\frac{\phi_{xx}^2 }{U^2}\right)\nonumber\\&\leq C\left( 1\!+\!d_0^{\beta+3}\right)\left(\int\langle\xi_0-\xi_\star\rangle^\beta\phi_0^2\!+\!\int\langle\xi_0-\xi_\star\rangle ^{\beta+1} \phi_{0x}^2\!+\!\int\langle\xi_0-\xi_\star\rangle ^{\beta+2} \phi_{0xx}^2\!+\!\int\langle\xi_0-\xi_\star\rangle ^{\beta+3}\phi_{0xxx}^2\right.\nonumber\\&\left.\quad+\!\int\frac{\phi_0^2 }{U^{3m-1}}\!+\!\int\frac{\phi_{0x}^2}{U^{1+m}}\!+\!\int\frac{\phi_{0xx}^2}{U^{3-m}}\!\right)+Cd_0^{-1}\!+\!C\left(d_0^{-\frac{1}{2}}\!+\!d_0^{-\beta}\!+\!N(T) \right)\int_{0}^{t}\left(d_0\!+\!\tau \right) ^{\beta+3}\phi_{x x}^2(0,\tau)\nonumber.
\end{align}
Since $|\phi_{xx}(0,t)|\sim|\phi_{xxx}(0,t)|+\mathrm{e}^{-\lambda_-(st+d(t))}$ from \eqref{phixxx 0}, if $d_0\gg1$, we have the desired estimate \eqref{H3 estimate}.
%\begin{align}
%	&\int\langle\xi-\xi_\star\rangle ^{\beta+3} \phi_{xxx}^2+(d_0+t) ^{\beta+3}\phi_{xx}^2(0,t)+\int_{0}^{t}(d_0+\tau) ^{\beta+3}\phi_{xx}^2(0,\tau)+\int_{0}^{t}\int\langle\xi-\xi_\star\rangle ^{\beta+3}\frac{\phi_{xxxx}^2}{U^{1-m}}\nonumber\\&\leq C\left( 1\!+\!d_0^{\beta+3}\right)\left(\int\langle\xi-\xi_\star\rangle^\beta\phi_0^2\!+\!\int\langle\xi-\xi_\star\rangle ^{\beta+1} \phi_{0x}^2\!+\!\int\langle\xi-\xi_\star\rangle ^{\beta+2} \phi_{0xx}^2\!+\!\int\langle\xi-\xi_\star\rangle ^{\beta+3}\phi_{0xxx}^2\right.\nonumber\\&\left.\quad+\int\frac{\phi_0^2 }{U^{3m-1}}\!+\!\int\frac{\phi_{0x}^2}{U^{1+m}}\!+\!\int\frac{\phi_{0xx}^2}{U^{3-m}}\!\right)+Cd_0^{-1}\nonumber,
%\end{align}
%provided $N(T)$ is sufficiently small.
\end{proof}

\begin{proof}[Proof of Proposition \ref{proposition priori estimate}]
	As $x-st-d(t)\sim x-st$ for all $t\in[0,T]$ due to the uniform boundedness of $d(t)$, the desired estimate \eqref{priori estimate} follows from \eqref{L2 estimate}, \eqref{L2-1 estimate}, \eqref{L2-2 estimate}, \eqref{H1 estimate}, \eqref{H1-2 estimate},  \eqref{H2 estimate}, \eqref{H2-2 estimate} and \eqref{H3 estimate}.
\end{proof}

\section{Proof of Theorem \ref{phi stability}}
With the local existence and the \textit{a priori} estimate in hand, we now prove Theorem \ref{phi stability}. The proof is established in the solution space $X$ (see \eqref{space X}). By virtue of the equivalence between $X$ and $\tilde{X}$ (defined in \eqref{space tilde X}), the result implies its validity in $\tilde{X}$ as well.

\begin{proof}[Proof of Theorem \ref{phi stability}]
	From Lemmas \ref{L2}-\ref{H3}, we can choose a constant $\epsilon_{2}\!>\!0$ such that $\epsilon_{2}\!<\!\epsilon_{0}$, where $\epsilon_{0}$ is mentioned in local existence analysis. Let $\epsilon_3\!=\!\min\left\{\epsilon_{2}/3, (1\!+\!d_0^{\frac{\beta+3}{2}})^{-1}\epsilon_{2}/3\sqrt{ C}\right\}$ and $N(0)+d_0^{-\frac{1}{2}}\leq\epsilon_3$. Since $N(0)+d_0^{-\frac{1}{2}}\leq\epsilon_3\leq \epsilon_{2}/3<\epsilon_{0}$, it follows from Proposition \ref{local existence} that
	the system \eqref{d equ} admits a unique solution $(\phi,d)\in X(0,T_0)$. Moreover, for any $0\leq t\leq T_0$, we
	have $N(t)+d_0^{-\frac{1}{2}}\leq 2(N(0)+d_0^{-\frac{1}{2}})+d_0^{-\frac{1}{2}}\leq 3\epsilon_3\leq3\epsilon_{2}/3=\epsilon_{2}.$
	Subsequently, by employing Proposition \ref{proposition priori estimate}, we establish that $N(T_0)+d_0^{-\frac{1}{2}}\leq\sqrt{C}(1+d_0^{\frac{\beta+3}{2}})N(0)+\left(\sqrt{C}+1 \right)d_0^{-\frac{1}{2}}\leq \sqrt{ C}(1+d_0^{\frac{\beta+3}{2}})\epsilon_3\leq\epsilon_{2}/3.$
	Now, treating $T_0$ as the new initial time and applying
	Proposition \ref{local existence} again, we extend the solution uniquely to the interval $[0,2T_0]$ with $(\phi,d)\in X(T_0,2T_0)$. Similarly, for any $0\leq t\leq 2T_0$, we obtain $N(t)+d_0^{-\frac{1}{2}}\leq 2(N(T_0)+d_0^{-\frac{1}{2}})+d_0^{-\frac{1}{2}}\leq 3\epsilon_{2}/3=\epsilon_{2}.$ Then, applying Proposition \ref{proposition priori estimate}  with $T=2T_0$ again, we deduce that $N(2T_0)+d_0^{-\frac{1}{2}}\leq \sqrt{C}(1+d_0^{\frac{\beta+3}{2}})N(0)+\left(\sqrt{C}+1 \right)d_0^{-\frac{1}{2}}\leq \sqrt{C}(1+d_0^{\frac{\beta+3}{2}})\epsilon_3\leq\epsilon_{2}/3<\epsilon_{0}.$
	 Hence, by repeating this continuation process, we can derive a unique global solution $(\phi,d)\in X(0,\infty)$ that satisfies the estimate \eqref{priori estimate} for all $t\in[0,\infty)$.
	
	To complete the proof, it remains to show \eqref{asymptotic} and \eqref{asymptotic1}. We denote $q(t)\triangleq\|\phi_x(\cdot,t)\|^2$. Given the established estimate \eqref{priori estimate} for all $t\in[0,\infty)$, it is easy to see that
	$$
	\int_{0}^{\infty}\|\phi_x(\cdot,\tau)\|_1^2 \mathrm{d}\tau \leq \int_{0}^{\infty}\left\|\phi_x(\cdot,\tau)\right\|_{w_4}^2\mathrm{d}\tau\leq C\left( 1\!+\!d_0^{\beta+3}\right) N^2(0)+Cd_0^{-1} \leq C,
	$$	
	which implies $q(t)\in L^1(0,\infty)$. Noting that $\|\phi_{x}(\cdot,t)\|_{L^\infty}\leq CN(T)$ and $\|\phi_{xx}(\cdot,t)\|_{L^\infty}\leq CN(T)$, combining \eqref{priori estimate} with \eqref{phizequ}, a direct  calculation yields
		\begin{align}
\int_{0}^{\infty }\int\phi_{xt}^2&\leq C\int_{0}^{\infty }\int \left( U_x^2\phi_{x}^2+\phi_{x x}^2+\frac{\phi_{xxx}^2}{U^{2-2m}}+\frac{U_x^2}{ U^{4-2m}}\phi_{xx}^2+\frac{U_{xx}^2}{ U^{4-2m}}\phi_{x}^2+\frac{U_x^4}{ U^{6-2m}}\phi_{x}^2\right.\nonumber\\&\quad\left.+d'(\tau)^2U_x^2+F_x^2+G_{xx}^2\right)\nonumber\\&\leq C\int_{0}^{\infty }\int \left(\phi_{x}^2+\frac{\phi_{x x}^4}{U^{4-2m}}+\frac{\phi_{x x}^2}{U^2}+\frac{\phi_{xxx}^2}{U^{4-2m}}+d_0^{-1} \right)+C\int_{0}^{\infty }|\phi_{xx}(0,\tau)|^2\nonumber\\&\leq C,\nonumber
		\end{align}
where we have used
\begin{equation}\nonumber
	\begin{aligned}
\int_{0}^{\infty }\int\frac{\phi_{x x}^4}{U^{4-2m}}	\leq CN(T)\int_{0}^{\infty }\int\frac{\phi_{x x}^3}{U^{4-2m}}&\leq C\int_{0}^{\infty}\left\|\frac{\phi_{xx} }{U^{\frac{3-m}{2}} }\right\|_{L^{\infty}}\left\|\frac{\phi_{xx} }{U^{\frac{3-m}{2}} }\right\|_{L^{2}}\left\|\frac{\phi_{xx} }{U^{-m}} \right\|_{L^{2}}\\&
\leq C\left(  \int_{0}^{\infty}\int\frac{\phi_{xxx}^2}{U^{4-2m}}+ \int_{0}^{t}\int\frac{\phi_{xx}^2 }{U^{2}}\right)
	\end{aligned}
\end{equation}
in the last inequality. Thus,	
	\begin{equation}\nonumber
			\int_{0}^{\infty } |q'(\tau)|\mathrm{d}\tau
			=2 \int_{0}^{\infty } \left|\int\phi_x\phi_{xt}\right|\mathrm{d}\tau \leq C\int_{0}^{\infty }\int \left(\phi_{x}^2 +\phi_{xt}^2\right)\mathrm{d}\tau\leq C.
	\end{equation}
Then,
	\begin{equation}\nonumber
		\|\phi_x(\cdot,t)\|\rightarrow 0, \quad \text { as } t \rightarrow +\infty.
	\end{equation}
	This along with the H\"{o}lder's inequality implies that
	\begin{equation}\nonumber
		\begin{aligned}
			\phi_x^2(x, t) =-2 \int_x^{+\infty} \phi_y \phi_{yy}(y, t) d y
			& \leq 2\left(\int_{0}^{\infty} \phi_y^2 d y\right)^{1 / 2}\left(\int_{0}^{\infty} \phi_{yy}^2 d y\right)^{1 / 2} \\
			& \leq C\left\|\phi_x(\cdot, t)\right\| \rightarrow 0,\quad \text { as } t \rightarrow+\infty.
		\end{aligned}
	\end{equation}
Therefore \eqref{asymptotic} is proved.

We next show that \eqref{asymptotic1}. By virtue of \eqref{priori estimate} and \eqref{e1}, one has
\begin{equation}\nonumber
	\begin{aligned}
	 \int_{0}^{\infty}|d'(\tau)|\mathrm{d}\tau &\leq \int_{0}^{\infty}\left(\mathrm{e}^{-\gamma(d_0+\tau)}+|\phi_{x x}(0,\tau)|\right)\mathrm{d}\tau\\&\leq Cd_0^{-1}+\left( \int_{0}^{\infty}\left(d_0+\tau\right)^{\beta+3}|\phi_{xx}(0,\tau)|^2\right)^{\frac{1}{2}}\left( \int_{0}^{\infty}\left(d_0+\tau\right)^{-(\beta+3)}\right)^{\frac{1}{2}}&\leq C,
	 	\end{aligned}
\end{equation}
which implies $d'(t)\in L^1(0,\infty)$. Then,
\begin{equation}\nonumber
	 d(t)-d_0=\int_{0}^{t}d'(\tau)d\tau<+\infty.
\end{equation}
This gives $|d(t)|<+\infty$ and
\begin{equation}\nonumber
	\lim_{t\rightarrow+\infty}d(t)=d_0+\int_{0}^{\infty}d'(\tau)d\tau\triangleq d_\infty.
\end{equation}
The proof of Theorem \ref{phi stability} is complete.
\end{proof}

\section*{Acknowledgements}
The research of M. Mei is supported in part by Natural Sciences and Engineering Research Council of Canada (RGPIN 2022-03374) and National Natural Science Foundation of China (W2431005).

\end{document}